\documentclass{svmult}

\usepackage{amsmath}

\usepackage[utf8]{inputenc}     
\usepackage[english]{babel}     

\usepackage{mathptmx}       
\usepackage{helvet}         
\usepackage{courier}        

\usepackage{amsfonts}
\usepackage{amssymb}

\usepackage{xcolor}          
\usepackage[bottom]{footmisc}
\usepackage{stmaryrd}        
\usepackage{mathtools}

\definecolor{linkred}{RGB}{255,1,1}
\definecolor{citegreen}{RGB}{1,190,1}
\usepackage[linkcolor=linkred
           ,citecolor=citegreen
           ,ocgcolorlinks        
           ,bookmarksopen=True,  
           ]{hyperref}

\usepackage{microtype, hyperref}           

\usepackage[all]{xy}

\makeatletter
\def\toclevel@authorch{1000}
\def\toclevel@author{1000}
\makeatother

\spnewtheorem{theorem}{Theorem}[section]{\bfseries}{\itshape}

\spnewtheorem{lemma}[theorem]{Lemma}{\bfseries}{\itshape}
\spnewtheorem{definition}[theorem]{Definition}{\bfseries}{\upshape}
\spnewtheorem{proposition}[theorem]{Proposition}{\bfseries}{\itshape}
\spnewtheorem{corollary}[theorem]{Corollary}{\bfseries}{\itshape}
\spnewtheorem*{remark}{Remark}{\itshape}{\rmfamily}
\spnewtheorem*{remarks}{Remarks}{\itshape}{\rmfamily}
\spnewtheorem{example}[theorem]{Example}{\bfseries}{\upshape}
\spnewtheorem{examples}[theorem]{Examples}{\bfseries}{\upshape}

\newcommand{\N}{\mathbb N}
\newcommand{\Z}{\mathbb Z}
\newcommand{\R}{\mathbb R}
\newcommand{\Q}{\mathbb Q}
\newcommand{\F}{\mathbb F}
\newcommand{\C}{\mathbb C}
\newcommand{\bdot}{\boldsymbol{\cdot}}

\newcommand{\ito}{\overset\sim\to}
\newcommand{\BF}{\text{\rm BF}}

 \DeclareMathOperator{\ind}{ind}
 \DeclareMathOperator{\ord}{ord}
\DeclareMathOperator{\lcm}{lcm} 

 \DeclareMathOperator{\supp}{supp}
 
\DeclareMathOperator{\Ker}{Ker} \DeclareMathOperator{\Hom}{Hom}
\DeclareMathOperator{\GL}{GL}
\DeclareMathOperator{\Sym}{Sym}

\renewcommand{\t}{\, | \,}
\newcommand{\red}{{\text{\rm red}}}
\newcommand{\DP}{\negthinspace : \negthinspace}

\begin{document}

\title*{The interplay of Invariant Theory with Multiplicative Ideal Theory and with Arithmetic Combinatorics}
\titlerunning{Invariant Theory, Multiplicative Ideal Theory, and Arithmetic Combinatorics}

\author{K\'alm\'an Cziszter and M\'aty\'as Domokos and Alfred Geroldinger}

\institute{MTA Alfr\'ed R\'enyi Institute of Mathematics,  Re\'altanoda u. 13 -- 15, 1053 Budapest, Hungary, {\it Email addresses: cziszter.kalman@gmail.com, domokos.matyas@renyi.mta.hu}; Institute for Mathematics and Scientific Computing, University of Graz, NAWI Graz,
Heinrichstra{\ss}e 36, 8010 Graz, Austria, {\it Email address: alfred.geroldinger@uni-graz.at}}


%
%
\maketitle

\renewcommand*{\theHtheorem}{\theHsection.\arabic{theorem}}

{\it Dedicated to Franz Halter-Koch on the occasion of his 70th birthday}

\bigskip

\abstract{This paper surveys and develops links between polynomial invariants of finite groups, factorization theory of Krull domains, and product-one sequences over finite groups.
The goal is to gain a better understanding of the multiplicative ideal theory  of invariant rings, and  connections between the Noether number and the Davenport constants of finite groups.
{}}

\section{\bf Introduction} \label{sec:1}

The goal of this paper is to deepen the links between the areas in the title. Invariant theory is concerned with the study of group actions on algebras, and in the present article we entirely concentrate on  actions of finite groups on polynomial algebras via linear substitution of the variables.

To begin with, let us briefly sketch the already existing links between the mentioned areas. For a finite-dimensional vector space $V$ over a field $\F$ and a finite group $G \le \GL(V)$, let $\F[V]^G \subset \F[V]$ denote the ring of invariants. Since E. Noether we know that $\F[V]^G \subset \F[V]$ is an integral ring extension and that $\F[V]^G$ is a finitely generated $\F$-algebra. In particular, $\F[V]^G$ is an integrally closed noetherian domain and hence a Krull domain. Benson \cite{Be93a} and Nakajima \cite{Na82z}  determined  its class group. Krull domains (their ideal theory and their class groups) are a central topic in multiplicative ideal theory (see the monographs \cite{Gi92,HK98} and the recent survey \cite{HK11b}). B. Schmid \cite{Sc91a} observed that the Noether number of a finite abelian group $G$ equals the Davenport constant of $G$ (a constant of central importance in zero-sum theory) and this established a first link between invariant theory and  arithmetic combinatorics. Moreover, ideal and factorization theory of Krull domains are most closely linked with zero-sum theory via transfer homomorphisms (see \cite{Ge-HK06a, Ge09a} and Subsection \ref{3.B}).

These links serve as our starting point. It is well-known that a domain $R$ is a Krull domain if and only if its monoid $R^{\bullet}$ of nonzero elements is a Krull monoid if and only if $R$ (resp. $R^{\bullet}$) has a divisor theory. To start with Krull monoids, a monoid $H$ is Krull if and only if its associated reduced monoid $H/H^{\times}$ is Krull, and every Krull monoid $H$ is a direct product $H^{\times} \times H_0$ where $H_0$ is isomorphic to $H/H^{\times}$. A reduced Krull monoid is uniquely determined (up to isomorphism) by its characteristic (roughly speaking by its class group $\mathcal C (H)$ and the distribution of the prime divisors in its classes; see the end of Subsection \ref{4.A}).
By definition of the class group, a Krull monoid $H$ is factorial if and only if $\mathcal C(H)$ is trivial. Information on the subset $\mathcal C(H)^* \subset \mathcal C(H)$  of classes containing prime divisors is the crucial ingredient  to understand the arithmetic of $H$, and hence in order to study the arithmetic of Krull monoids the first and most important issue is to determine $\mathcal C (H)^*$.
By far the  best understood setting in factorization theory are Krull monoids with finite class groups where every class contains a prime divisor. Indeed, there has been an abundance of work on them  and we refer the reader to the survey by W.A. Schmid in this proceedings \cite{Sc16a}.
A canonical method to obtain information on $\mathcal C (H)^*$  is to identify explicitly a divisor theory for $H$.
A divisor theory of a monoid (or a domain) $H$  is a divisibility preserving homomorphism from $H$ to a free abelian monoid which  satisfies a certain minimality property (Subsection \ref{2.1}). The concept of a divisor theory stems from algebraic number theory and it has found far-reaching generalizations in multiplicative ideal theory (\cite{HK98}). Indeed, divisor-theoretic tools, together with ideal-theoretic and valuation-theoretic ones, constitute a highly developed  machinery  for the structural description of  monoids and domains.

All the above mentioned concepts and problems from multiplicative ideal theory are studied for the ring of invariants.  Theorem~\ref{main-theorem} (in  Subsection \ref{4.A}) provides an explicit divisor theory of the  ring of invariants $R=\F[V]^G$.
The divisibility preserving homomorphism from $R^{\bullet}$ goes into a free abelian monoid which can be naturally described in the language of invariant theory,
and the associated canonical transfer homomorphism $\theta\colon R^\bullet  \to \mathcal{B}(\mathcal{C}(R)^*)$ from the multiplicative monoid of the ring  $R$ onto the monoid of zero-sum sequences over the class group of $R$ also has a natural invariant theoretic interpretation. In addition to  recovering the result of Benson and Nakajima on the class group $\mathcal C(\mathbb{F}[V]^G)$ (our treatment is essentially self-contained), we gain further information on  the multiplicative structure of $R$, and we pose the problem to determine its characteristic  (Problem \ref{characteristic-problem}). In particular, whenever we can show  -- for a given ring of invariants -- that every class contains at least one prime divisor, then all results of factorization theory (obtained for Krull monoids with finite class group and prime divisors in all classes) apply to the ring of invariants.

In Subsection \ref{subsec:abelian} we specialize to abelian groups whose order is  not divisible by the characteristic of $\F$.
The  Noether number $\beta(G)$ is the supremum over all finite dimensional $G$-modules $V$ of the maximal degree of an element in a minimal homogeneous generating system of $\F[V]^G$, and  the Davenport constant $\mathsf D (G)$  is the maximal length of a minimal zero-sum sequence over $G$.
We start with a result on the  structural connection between $\F[V]^G$ and the
monoid of zero-sum sequences over $G$, that lies behind the equality $\beta(G)=\mathsf D(G)$. Clearly,  the  idea here is well known (as far as we know, it was first used by B. Schmid \cite{Sc91a}, see also  \cite{F-M-P-T08a}). The benefit of the detailed presentation as given in Proposition \ref{prop:bschmid} is twofold. First, the past 20 years have seen great progress in zero-sum theory (see Subsection \ref{3.D} for a sample of  results) and Proposition \ref{prop:bschmid} allows to carry over all  results on the structure of (long) minimal zero-sum sequences to the structure of $G$-invariant monomials.
Second, we observe that the submonoid $M^G$ of $R^\bullet$  consisting of the invariant monomials is again a Krull monoid, and restricting the transfer homomorphism  $\theta\colon R^\bullet  \to \mathcal{B}(\mathcal{C}(R)^*)$ (mentioned in the above paragraph) to $M^G$ we obtain essentially the canonical transfer homomorphism $M^G\to \mathcal{B}(\mathcal{C}(M^G)^*)$.
This turns out to be rather close to the transfer homomorphism $\psi\colon M^G\to \mathcal{B}(\widehat G)$ into the  monoid of
zero-sum sequences over the character group of $G$ (see Proposition  \ref{prop:bschmid}), which is responsible for the equality $\beta(G)=\mathsf D(G)$.
The precise statement is  given in Proposition  \ref{prop:diagram}, which explains how the transfer homomorphism $\psi$ (existing only for abelian groups) relates to the more general transfer homomorphism $\theta$ from the above paragraph which exists for an arbitrary finite group.
In Proposition \ref{prop:diagram} we point out that every class of $\mathcal C ( \F[V]^G )$ contains a prime divisor which contributes to Problem \ref{characteristic-problem}.

Let now $G$ be a finite non-abelian group. Until recently the precise value of the Noether number $\beta (G)$ was known only for the dihedral groups and very few small groups (such as $A_4)$. In the last couple of years the first two authors have determined the precise value of the Noether number for groups having a cyclic subgroup of index two and for non-abelian groups of order $3p$ \cite{Cz-Do14a, Cz14d, Cz-Do15a}.
In this work results on zero-sum sequences over finite abelian groups (for example, information on the structure of long minimal zero-sum sequences and on the $k$th Davenport constants) were successfully applied. Moreover, a decisive step was the introduction of the $k$th Noether numbers, a concept inspired by the $k$th Davenport constants of abelian groups. The significance  of this concept  is that it furnishes some reduction lemmas (listed in  Subsection \ref{sec:5.1} ) by which the ordinary Noether number of a group can be bounded via structural reduction in the group.

The concept of the
$k$th Davenport constants $\mathsf D_k (G)$ has been introduced by Halter-Koch \cite{HK92c} for abelian groups in order to study the asymptotic behavior of arithmetical counting functions in rings of integers of algebraic number fields (see \cite[Theorem 9.1.8]{Ge-HK06a}, \cite[Theorem 1]{Ra05a}). They have been further studied in \cite{De-Or-Qu01,Fr-Sc10}.
In the last years the third author and Grynkiewicz \cite{Ge-Gr13a, Gr13b} studied the (small and the large) Davenport constant of non-abelian groups, and among others determined their precise values for groups having a cyclic subgroup of index two. It can be observed that for these groups  the Noether number  is between the small and the large Davenport constant.

This motivated  a new and more abstract view at the Davenport constants, namely $k$th Davenport constants   of   BF-monoids  (Subsection \ref{2.E}). The goal is  to relate the Noether number with Davenport constants of suitable monoids as a generalization of the equation $\beta (G) = \mathsf D (G)$ in the abelian case. Indeed,
the $k$th Davenport constant $\mathsf D_k (G)$  of an abelian group $G$ is recovered as our $k$th Davenport constant of  the monoid $\mathcal B (G)$ of zero-sum sequences over $G$.

We apply the new  concept of the $k$th Davenport constants  to two classes of BF-monoids.
First, to the monoid $\mathcal{B}(G,V)$ associated to a $G$-module $V$ in Subsection~\ref{5.D}  (when $G$ is abelian we recover the monoid $M^G$ of $G$-invariant monomials from Subsection \ref{subsec:abelian}), whose Davenport constants provide a lower bound for the  corresponding Noether numbers (see Proposition \ref{prop:module-davenport}).
Second, we study  the monoid of product-one sequences over  finite  groups (Subsections \ref{3.1} and \ref{3.C}).
We derive a variety of   features of  the $k$th Davenport constants of the monoid of product-one sequences over $G$ and observe that they are  strikingly similar to the corresponding features of the $k$th Noether numbers  (see Subsection \ref{sec:5.1} for a comparison).

We pose a problem on the relationship between Noether numbers and Davenport constants of non-abelian groups (Problem \ref{Noether-Davenport}) and we
illustrate the efficiency of the above methods by Examples \ref{example:C_pq rtimes C_q}, \ref{S_4}, and \ref{Pauli}  (appearing for the first time), where the explicit value of  Noether numbers and Davenport constants of some non-abelian groups are determined.

\smallskip

\centerline{\it Throughout this paper, let $G$ be a finite group, $\F$ be a  field, and}
\centerline{\it   $V$ be a finite dimensional $\F$-vector space endowed with a linear action of $G$.}

\section{\bf Multiplicative Ideal Theory: Krull monoids, C-monoids, and Class Groups} \label{sec:2}

We denote by $\N$ the set of positive integers, and we put $\N_0 =
\N \cup \{0\}$. For every $n \in \mathbb N$, we denote by $C_n$ a
cyclic group with $n$ elements. For real numbers $a, b \in \R$, we
set $[a, b] = \{ x \in \Z \colon a \le x \le b \}$. If $A, B$ are sets, we write $A \subset B$ to mean that $A$ is contained in $B$ but may be equal to $B$. In Subsections \ref{2.A} -- \ref{2.D} we gather basic material on Krull monoids and C-monoids. In  Subsection \ref{2.E} we introduce a new concept, namely Davenport constants of BF-monoids.

\subsection{\bf  Monoids and Domains: Ideal theoretic and divisor theoretic concepts} \label{2.A}~

Our notation and terminology follows \cite{Ge-HK06a} and \cite{HK98} (note that the monoids in \cite{HK98} do contain a zero-element, whereas the monoids in \cite{Ge-HK06a} and in the present manuscript do not contain a zero-element).
By a {\it monoid}, we mean a commutative, cancellative semigroup with unit element. Then the multiplicative semigroup $R^{\bullet}=R\setminus \{0\}$ of non-zero elements of a domain is a monoid. Following the philosophy of multiplicative ideal theory  we describe the arithmetic and the theory of divisorial ideals of domains by means of their multiplicative monoids. Thus we start with monoids.

Let $H$ be a multiplicatively written monoid. An element $u \in H$ is called
\begin{itemize}
\item {\it invertible} if there  is an element $v \in H$ with $u v = 1$.

\item {\it irreducible} (or an {\it atom}) if $u$ is not invertible and, for all $a, b \in H$, $u = ab$ implies  $a$ is invertible or $b$ is invertible.

\item {\it prime} if $u$ is not invertible and, for all $a, b \in H$, $u \t ab$ implies $u \t a$ or $u \t b$.
\end{itemize}
We denote by $\mathcal A (H)$ the set of atoms of $H$, by $H^{\times}$ the group of invertible elements, and by $H_{\red} = \{aH^{\times} \colon a \in H \}$ the associated reduced monoid of $H$. We say that $H$ is reduced if $|H^{\times}| = 1$. We denote by $\mathsf q (H)$ a quotient group of $H$ with $H \subset \mathsf q (H)$, and for a prime element $p \in H$, let $\mathsf v_p \colon \mathsf q (H) \to \Z$ be the $p$-adic valuation.
Each monoid homomorphism $\varphi \colon H \to D$ induces  a group homomorphism $\mathsf q (H) \colon \mathsf q (H) \to \mathsf q (D)$.
For a subset $H_0 \subset H$, we denote by $[H_0] \subset H$ the submonoid generated by $H_0$, and by $\langle H_0 \rangle \le \mathsf q (H)$ the subgroup generated by $H_0$. We denote by
$\widetilde H  = \big\{ x \in \mathsf q (H) \, \colon \, x^n \in H \text{~for
               some~}  n \in \N \big\}$
the \ {\it root closure} \ of $H$, and by
$
\widehat H  = \big\{ x \in \mathsf q (H) \, \colon \,  \text{there exists~} c
\in H  \text{~such that~}  cx^n \in H \text{~for all~}
              n \in \N \big\}
$
the  {\it complete integral closure} \ of $H$. Both
$\widetilde{H}$ and $\widehat{H}$ are monoids, and we have $H
\subset \widetilde H \subset \widehat H \subset \mathsf q (H)$. We say that $H$ is root closed (completely integrally closed resp.) if $H = \widetilde H$ ($H=\widehat H$ resp.).
For a set $P$, we denote by $\mathcal F (P)$ the free abelian monoid
with basis $P$. Then every $a \in  \mathcal F (P)$ has a unique
representation in the form
\[
a = \prod_{p \in P} p^{\mathsf v_p (a)} \,, \ \text{where} \ \mathsf v_p (a) \in \mathbb N_0
\ \text{and} \ \mathsf v_p (a) = 0 \ \text{for almost all} \ p \in P \,.
\]
The monoid $H$ is said to be
\begin{itemize}
\item {\it atomic} if  every $a \in H \setminus H^{\times}$ is a product of finitely many atoms of $H$.

\item {\it factorial} if every $a \in H \setminus H^{\times}$ is a product of finitely many primes of $H$ (equivalently, $H=H^{\times} \times \mathcal F (P)$ where $P$ is a set of representatives of primes of $F$).

\item {\it finitely generated} if $H = [E]$ for some finite subset $E \subset H$.
\end{itemize}
If $H = H^{\times} \times \mathcal F (P)$ is factorial and $a \in H$, then
$
|a| = \sum_{p \in P} \mathsf v_p (a) \in \mathbb N_0
$
is called the length of $a$.
If $H$ is reduced, then it is finitely generated if and only if it is atomic and $\mathcal A (H)$ is finite.
Since every prime is an atom, every factorial monoid is atomic. For every non-unit $a\in H$,
\[
\mathsf L_H (a) = \mathsf L (a) = \{k \in \N \colon a \ \text{may be written as a product of $k$ atoms} \} \subset \N
\]
denotes the {\it set of lengths} of $a$. For convenience, we set $\mathsf L (a) = \{0\}$ for $a \in H^{\times}$. We say that $H$ is a BF-monoid if it is atomic and all sets of lengths are finite.
A monoid homomorphism $\varphi \colon H \to D$ is said to be
\begin{itemize}
\item a {\it divisor homomorphism} if $\varphi (a) \t \varphi (b)$ implies that $a \t b$ for all $a, b \in H$.

\item {\it cofinal} if for every $\alpha \in D$ there is an $a \in H$ such that $\alpha \t \varphi (a)$.

\item a {\it divisor theory} (for $H$) if $D = \mathcal F (P)$ for some set $P$, $\varphi$ is a divisor homomorphism, and for every $p \in P$, there exists a finite nonempty subset $X \subset H$ satisfying $p = \gcd \big( \varphi (X) \big)$.
\end{itemize}
Obviously, every divisor theory is cofinal. Let $H \subset D$ be a submonoid. Then  $H \subset D$ is called
\begin{itemize}
\item  {\it saturated} if the embedding $H \hookrightarrow D$ is a divisor homomorphism.
\item {\it divisor closed} if $a \in H$, $b \in D$ and $b \t a$ implies $b \in H$.
\item {\it cofinal} if the embedding $H \hookrightarrow D$ is cofinal.
\end{itemize}
It is easy to verify that $H \hookrightarrow D$ is a divisor homomorphism if and only if $H = \mathsf q (H) \cap D$, and if this holds, then $H^{\times} = D^{\times} \cap H$. If $H \subset D$ is divisor closed, then $H \subset D$ is saturated.

For subsets $A, B \subset \mathsf q (H)$, we denote by
$(A \DP B) = \{ x \in \mathsf q (H) \colon x B \subset A \}$, by $A^{-1} = (H \DP A)$, and by $A_v = (A^{-1})^{-1}$.  A subset $\mathfrak a \subset H$ is called an $s$-ideal of $H$ if $\mathfrak a H = \mathfrak a$. A subset $X \subset \mathsf q (H)$ is called a fractional  $v$-ideal (or a {\it fractional  divisorial ideal}) if there is a $c \in H$ such that $cX \subset H$ and $X_v = X$. We denote by $\mathcal F_v (H)$ the set of all fractional $v$-ideals and by $\mathcal I_v (H)$ the set of all $v$-ideals of $H$. Furthermore, $\mathcal I_v^* (H)$ is the monoid of $v$-invertible $v$-ideals (with $v$-multiplication) and $\mathcal F_v (H)^{\times} = \mathsf q \big( \mathcal I_v^* (H) \big)$ is its quotient group of fractional invertible $v$-ideals. The monoid $H$ is completely integrally closed if and only if every non-empty $v$-ideal of $H$ is $v$-invertible, and $H$ is called $v$-noetherian if it satisfies the ACC (ascending chain condition) on $v$-ideals. If $H$ is $v$-noetherian, then $H$ is a BF-monoid. We denote by $\mathfrak X (H)$ the set of all minimal nonempty prime $s$-ideals of $H$.

The map $\partial \colon H \to \mathcal I_v^* (H)$, defined by $\partial (a) = aH$ for each $a \in H$, is a cofinal divisor homomorphism. Thus, if  $\mathcal H = \{aH
\colon a \in H \}$ is the monoid of principal ideals of $H$, then $\mathcal H \subset \mathcal I_v^* (H)$ is saturated and cofinal.

\subsection{\bf  Class groups and class semigroups}~ \label{2.B}

Let $\varphi \colon H \to D$ be a monoid homomorphism. The group $\mathcal C (\varphi) = \mathsf q (D)/ \mathsf q ( \varphi (H))$ is called the {\it class group} of $\varphi$. For $a \in \mathsf q (D)$, we denote by $[a]_{\varphi} = a \mathsf q ( \varphi (H)) \in \mathcal C ( \varphi)$ the class containing $a$. We use additive notation for $\mathcal C ( \varphi )$ and so $[1]_{\varphi}$ is the zero element of $\mathcal C ( \varphi )$.

Suppose that $H \subset D$ and that $\varphi = (H \hookrightarrow D)$. Then $\mathcal C ( \varphi ) = \mathsf q (D)/\mathsf q (H)$, and for $a \in D$ we set $[a]_{\varphi}=[a]_{D/H} = a \mathsf q (H)$. Then
\[
D/H = \{ [a]_{D/H} \colon a \in D\} \subset \mathcal C ( \varphi )
\]
is a submonoid with quotient group $\mathsf q (D/H)=\mathcal C (\varphi)$. It is easy to check that $D/H$ is a group if and only if $H \subset D$ is cofinal. In particular, if $D/H$ is finite or if $\mathsf q (D)/\mathsf q (H)$ is a torsion group, then $D/H= \mathsf q (D)/\mathsf q (H)$.
Let $H$ be a monoid. Then $\mathcal H \subset \mathcal I_v^* (H)$ is saturated and cofinal, and
\[
\mathcal C_v(H) = \mathcal I_v^* (H)/\mathcal H = \mathcal F_v (H)^{\times}/ \mathsf q (\mathcal H)
\]
is the {\it $v$-class group} of $H$.

We will also need  the concept of class semigroups which are a refinement of ordinary class groups in commutative algebra.
Let $D$ be a monoid and $H \subset D$ a submonoid. Two elements $y, y' \in D$
are called $H$-equivalent, if $y^{-1}H \cap D = {y'}^{-1} H \cap
D$. $H$-equivalence is a congruence relation on $D$. For $y \in
D$, let $[y]_H^D$ denote the congruence class of $y$, and let
\[
\mathcal C (H,D) = \{ [y]_H^D \colon y \in D \} \quad \text{and}
\quad \mathcal C^* (H,D) = \{ [y]_H^D \colon y \in (D \setminus
D^{\times}) \cup \{1\} \}.
\]
Then $\mathcal C (H,D)$ is a semigroup with unit element $[1]_H^D$
(called the {\it class semigroup} of $H$ in $D$) and $\mathcal C^* (H,D)
\subset \mathcal C (H,D)$ is a subsemigroup (called the {\it reduced
class semigroup} of $H$ in $D$).
The map
\[
\theta \colon \mathcal C (H,D) \to D/H \,, \quad \text{defined by} \quad \theta ( [a]_H^D) = [a]_{D/H} \quad \text{for all} \ a \in D \,,
\]
is an epimorphism, and it is an isomorphism if and only if $H \subset D$ is saturated.

\subsection{\bf  Krull monoids and Krull domains}~ \label{2.C}

\begin{theorem} \label{2.1}~
Let $H$ be a monoid. Then the following statements are equivalent{\rm \,:}
      \begin{itemize}
      \item[(a)] $H$ is $v$-noetherian and completely integrally closed,

      \item[(b)] $\partial \colon H \to \mathcal I_v^* (H)$ is a divisor theory.

      \item[(c)] $H$ has a divisor theory.

      \item[(d)] There is a divisor homomorphism $\varphi \colon H \to D$ into a factorial monoid $D$.

      \item[(e)] $H_{\red}$ is a saturated submonoid of a free abelian monoid.
      \end{itemize}

      \noindent
      \smallskip
      {\rm If $H$ satisfies these conditions, then $H$ is called a } Krull monoid.
\end{theorem}

\begin{proof}
See \cite[Theorem 2.4.8]{Ge-HK06a} or \cite[Chapter 22]{HK98}.
\end{proof}

Let $H$ be a Krull monoid. Then $\mathcal I_v^* (H)$ is free abelian with basis $\mathfrak X (H)$. Let $\mathfrak p \in \mathfrak X (H)$. Then $\mathsf v_{\mathfrak p}$ denotes the $\mathfrak p$-adic valuation of $\mathcal F_v(H)^{\times}$. For $x \in \mathsf q (H)$, we set $\mathsf v_{\mathfrak p}(x)= \mathsf v_{\mathfrak p} (xH)$ and we call $\mathsf v_{\mathfrak p}$ the $\mathfrak p$-adic valuation of $H$. Then $\mathsf v \colon H \to \N_0^{( \mathfrak X (H))}\,, \quad \text{defined by} \quad \mathsf v (a) = \bigl( \mathsf v_\mathfrak
p (a) \bigr)_{\mathfrak p \in \mathfrak X (H)}$
is  a divisor theory and $H = \{ x \in \mathsf q (H) \colon \mathsf v_{\mathfrak p} (x) \ge 0 \ \text{for all} \ \mathfrak p \in \mathfrak X (H) \}$.

If $\varphi \colon H \to D=\mathcal F (P)$ is a divisor theory, then there is an isomorphism $\Phi \colon \mathcal I_v^* (H) \to D$ such that $\Phi \circ \partial = \varphi$, and it induces an isomorphism $\overline{\Phi} \colon \mathcal C_v (H) \to \mathcal C (\varphi)$.  Let $D= \mathcal F (P)$ be such that $H_{\red} \hookrightarrow D$ is a divisor theory. Then $D$ and $P$ are uniquely determined by $H$,
\[
\mathcal C (H) = \mathcal C (H_{\red}) = D/H_{\red}
\]
is called the {\it $($divisor$)$ class group} of $H$, and  its elements are called the classes of $H$. By definition, every class $g \in \mathcal C (H)$ is a subset of $\mathsf q (D)$ and $P \cap g$ is the set of prime divisors lying in $g$. We denote by $\mathcal C (H)^* = \{ [p]_{D/H_{\red}} \colon p \in P \} \subset \mathcal C (H)$ the subset of classes containing prime divisors (for more details we refer to the discussion after Definition 2.4.9 in \cite{Ge-HK06a}).

\begin{proposition} \label{prop:torsion-divtheory}
Let $H$ be a Krull monoid, and let $\varphi \colon H \to D=\mathcal F (P)$ be a divisor homomorphism.
\begin{enumerate}
\item There is a submonoid $C_0 \subset \mathcal C(\varphi)$ and an epimorphism $C_0 \to \mathcal C_v (H)$.

\smallskip
\item Suppose that $H \subset D$ is saturated and that $\mathsf q (D)/\mathsf q (H)$ is a torsion group. We set $D_0 =\{ \gcd_D (X) \colon X\subset H  \mbox{ finite}\}$, and for $p\in P$ define $e(p)=\min\{\mathsf{v}_p(h)\colon h\in H \ \text{with} \ \mathsf v_p (h)>0 \}$.
      \begin{itemize}
      \smallskip
      \item[(a)] $D_0$ is a free abelian monoid with basis  $\{ p^{e(p)} \colon p\in P \}$.

      \smallskip
      \item[(b)] The embedding $H\hookrightarrow D_0$ is a divisor theory for $H$.
      \end{itemize}
\end{enumerate}
 \end{proposition}

\begin{proof}
1. follows from \cite[Theorem 2.4.8]{Ge-HK06a},  and 2. from  \cite[Lemma 3.2]{Sc10a}.
\end{proof}

Let $R$ be a domain with quotient field $K$. Then $R^{\bullet}=R\setminus \{0\}$ is a monoid, and all notions defined for monoids so far will be applied for domains. To mention a couple of explicit examples, we denote by $\mathsf q (R)$ the quotient field of $R$ and we have $\mathsf q (R)=\mathsf q (R^{\bullet})\cup \{0\}$, and for the complete integral closure we have $\widehat R = \widehat{R^{\bullet}} \cup \{0\}$ (where $\widehat R$ is the integral closure of $R$ in its quotient field). We denote by $\mathfrak X (R)$ the set of all minimal nonzero prime ideals of $R$, by $\mathcal I_v (R)$ the set of divisorial ideals of $R$, by $\mathcal I_v^* (R)$ the set of $v$-invertible divisorial ideals of $R$, and by $\mathcal F_v (R)$ the set of fractional divisorial ideals of $R$. Equipped with $v$-multiplication, $\mathcal F_v (R)$ is a semigroup, and the map
\[
\iota^{\bullet} \colon \mathcal F_v (R) \to \mathcal F_v (R^{\bullet}) \,, \quad \text{defined by} \quad \mathfrak a \mapsto \mathfrak a \setminus \{0\} \,,
\]
is a semigroup isomorphism mapping $\mathcal I_v (R)$ onto $\mathcal I_v (R^{\bullet})$ and fractional principal ideals of $R$ onto fractional principal ideals of $R^{\bullet}$.  Thus $R$ satisfies the ACC on divisorial ideals of $R$ if and only if $R^{\bullet}$ satisfies the ACC on divisorial ideals of $R^{\bullet}$. Furthermore, $R$ is completely integrally closed if and only if $R^{\bullet}$ is completely integrally closed. A domain  $R$ is a Krull domain if it is completely integrally closed and satisfies the ACC on divisorial ideals of $R$, and thus $R$ is a Krull domain if and only if $R^{\bullet}$ is a Krull monoid. If $R$ is a Krull domain, we set $\mathcal C (R) = \mathcal C (R^{\bullet})$.
The group $\mathcal F_v (R)^{\times}$ is the group of $v$-invertible fractional ideals and the set $\mathcal I_v^* (R)=\mathcal F_v(R)^{\times}\cap \mathcal I_v (R)$ of all $v$-invertible $v$-ideals of $R$ is a monoid with quotient group $\mathcal F_v (R)^{\times}$. The embedding of the non-zero principal  ideals $\mathcal H (R) \hookrightarrow \mathcal I_v^* (R)$ is a cofinal divisor homomorphism, and the factor group
\[
\mathcal C_v(R) =  \mathcal F_v(R)^\times / \{aR \colon a \in K^\times\} = \mathcal I_v^*(R)/\mathcal H(R)
\]
is called the  {\it $v$-class group}   of $R$. The map $\iota^\bullet$ induces isomorphisms $\mathcal
F_v(R)^\times \ito \mathcal F_v(R^\bullet)^\times$, \ $\mathcal
I_v^*(R) \, \ito \, \mathcal I_v^*(R^\bullet)$, and $\mathcal
C_v(R) \, \ito \, \mathcal C_v(R^\bullet)$, and in the sequel we shall identify these monoids and groups.

The above correspondence between domains and their monoids of non-zero elements can be extended to  commutative rings with zero-divisors and their monoids of regular elements (\cite[Theorem 3.5]{Ge-Ra-Re15c}), and there is  an analogue for prime Goldie rings (\cite[Proposition 5.1]{Ge13a}).

\begin{examples}~ \label{2.3}

1. (Domains) As mentioned above, the multiplicative monoid $R^{\bullet}$ of a domain $R$ is a Krull monoid if and only if $R$ is a Krull domain. Thus Property (a) in Theorem \ref{2.1} implies that a noetherian domain is Krull if and only if it is normal (i.e. integrally closed in its field of fractions). In particular, rings of invariants are Krull, as we shall see in  Theorem \ref{4.1}.

\smallskip
2. (Submonoids of domains) Regular congruence submonoids of Krull domains are Krull (\cite[Proposition 2.11.6]{Ge-HK06a}.

\smallskip
3. (Monoids of modules) Let $R$ be a (possibly noncommutative) ring and let $\mathcal C$ be a class of finitely generated (right) $R$-modules which is closed under finite direct-sums, direct summands, and isomorphisms. Then the set $\mathcal V (\mathcal C)$ of isomorphism classes of modules is a commutative semigroup with operation induced by the direct sum. If the endomorphism ring of each module in $\mathcal C$ is semilocal, then $\mathcal V ( \mathcal C)$ is a Krull monoid (\cite[Theorem 3.4]{Fa02}). For more information we refer to \cite{Fa06a,Fa12a, Ba-Ge14b}.

\smallskip
4. (Monoids of product-one sequences) In Theorem \ref{3.2} we will characterize the monoids of product-one sequences which are Krull.

\end{examples}

\subsection{\bf  C-monoids and C-domains}~ \label{2.D}

A monoid $H$ is called a C-{\it monoid} if it is a submonoid of a factorial monoid $F$ such that $H \cap F^{\times} = H^{\times}$ and the reduced class semigroup $\mathcal C^* (H,F)$ is finite. A domain is called a C-{\it domain} if $R^{\bullet}$ is a C-monoid.

\begin{proposition} \label{2.4}
Let $F$ be a factorial monoid and $H \subset F$ a submonoid such that $H \cap F^{\times} = H^{\times}$.
\begin{enumerate}
\item If $H$ is a {\rm C}-monoid, then $H$ is $v$-noetherian with $(H \DP \widehat H) \ne \emptyset$, and the complete integral closure $\widehat H$ is a Krull monoid with finite class group $\mathcal C (\widehat H)$.

\smallskip
\item Suppose that $F/F^{\times}$ is finitely generated, say $F = F^\times \times [p_1, \ldots, p_s]$  with pairwise non-associated prime elements $p_1, \ldots, p_s$. Then the following statements are equivalent{\rm \,:}
     \begin{enumerate}
     \item[(a)]
     $H$ is a \ $\text{\rm C}$-monoid defined in $F$.

     \smallskip
     \item[(b)] There exist some $\alpha \in \N$ and a subgroup $W \le F^\times$ such that $(F^\times \DP W) \t \alpha$, \ $W(H \setminus H^\times) \subset H$, and for all $j \in [1,s]$ and $a \in p^\alpha_j F$ we have $a \in H$ if and only if \ $p^\alpha_j a \in H$.
     \end{enumerate}
\end{enumerate}
\end{proposition}

\begin{proof}
For 1., see \cite[Theorems 2.9.11 and 2.9.13]{Ge-HK06a} and for 2. see \cite[Theorems 2.9.7]{Ge-HK06a}.
\end{proof}

\begin{examples}~ \label{2.5}

1. (Krull monoids) A Krull monoid is a C-monoid if and only if the class group is finite (\cite[Theorem 2.9.12]{Ge-HK06a}).

\smallskip
2. (Domains) Let $R$ be a domain. Necessary conditions for $R$ being a C-domain are given in Proposition \ref{2.4}. Thus suppose that $R$ is a Mori domain (i.e., a $v$-noetherian domain) with nonzero conductor $\mathfrak f = (R \DP \widehat R)$ and suppose that $\mathcal C (\widehat R)$ is finite. If $R/\mathfrak f$ is finite, then $R$ is a C-domain by \cite[Theorem 2.11.9]{Ge-HK06a}. This result generalizes to rings with zero-divisors (\cite{Ge-Ra-Re15c}), and in special cases we know that $R$ is a C-domain if and only if $R/\mathfrak f$ is finite (\cite{Re13a}).

\smallskip
3. (Congruence monoids) Let $R$ be Krull domain with finite class group $\mathcal C (R)$ and $H \subset R$ a congruence monoid such that $R/\mathfrak f$ is finite where $\mathfrak f$ is an ideal of definition for $H$. If $R$ is noetherian or $\mathfrak f$ is divisorial, then $H$ is a C-monoid (\cite[Theorem 2.11.8]{Ge-HK06a}). For a survey on arithmetical congruence monoids see \cite{Ba-Ch14a}.

\smallskip
4. In Subsection \ref{3.A} we shall prove that monoids of product-one sequences are C-monoids (Theorem \ref{3.2}), and we will meet C-monoids again in Proposition \ref{B(G,V)-is-C} dealing with the monoid $\mathcal{B}(G,V)$.
\end{examples}

Finitely generated monoids allow simple characterizations when they are Krull or when they are C-monoids. We summarize these characterizations in the next lemma.

\begin{proposition} \label{finitelygenerated}
Let $H$ be a monoid such that $H_{\mathrm{red}}$ is finitely generated.
\begin{enumerate}
\item Then $H$ is $v$-noetherian with $(H \DP \widehat H) \ne \emptyset$, $\widetilde H = \widehat H$, $\widetilde H/H^{\times}$ is finitely generated, and  $\widehat H$ is a Krull monoid.
      In particular, $H$ is a Krull monoid  if and only if $H = \widehat H$.

\smallskip
\item $H$ is a {\rm C}-monoid if and only if \ $\mathcal C ( \widehat H)$ is finite.

\smallskip
\item Suppose that $H$ is a submonoid of a factorial monoid $F = F^{\times} \times \mathcal F (P)$. Then the following statements are equivalent{\rm \,:}
      \begin{enumerate}
      \item $H$ is a \text{\rm C}-monoid defined in $F$, $F^{\times}/H^{\times}$ is a torsion group, and for every $p \in P$ there is an $a \in H$ such that $\mathsf v_p (a) > 0$.

      \smallskip
      \item For every $a \in F$, there is an $n_a \in \N$ with $a^{n_a} \in H$.
      \end{enumerate}
      If $(a)$ and $(b)$ hold, then $P$ is finite and $\widetilde H = \widehat H = \mathsf q (H) \cap F \subset F$ is saturated and cofinal.
\end{enumerate}
\end{proposition}

\begin{proof}
1.  follows from \cite[2.7.9 -  2.7.13]{Ge-HK06a}, and 2.  follows from \cite[Proposition 4.8]{Ge-Ha08b}.

\smallskip
3.  (a) $\Rightarrow$\, (b) \ For every $p \in P$, we set $d_p = \gcd \big( \mathsf v_p (H) \big)$, and by assumption we have $d_p > 0$. We set $P_0 = \{p^{d_p} \colon p \in P \}$ and $F_0 = F^{\times} \times \mathcal F (P_0)$. By \cite[Theorem 2.9.11]{Ge-HK06a}, $H$ is a C-monoid defined in $F_0$ and there is a divisor theory $\partial \colon \widehat H \to \mathcal F (P_0)$. By construction of $F_0$, it is sufficient to prove the assertion for all $a \in F_0$. Since $F^{\times}/H^{\times}$ is a torsion group, it is sufficient to prove the assertion for all $a \in \mathcal F (P_0)$. Let $a \in \mathcal F (P_0)$. Since $\mathcal C (\widehat H)$ is finite, there is an $n_a' \in \N$ such that $a^{n_a'} \in \widehat H$. Since $\widehat H = \widetilde H$, there is an $n_a'' \in \N$ such that $(a^{n_a'})^{n_a''} \in H$.

(b) $\Rightarrow$\, (a) \ For every $p \in P$ there is an $n_p \in \N$ such that $p^{n_p} \in H$ whence $\mathsf v_p (p^{n_p})=n_p > 0$.
Clearly, we have $\widehat H \subset \widehat F = F$ and hence $\widehat H \subset \mathsf q (\widehat H) \cap F = \mathsf q (H) \cap F$. Since for each $a \in F$ there is an $n_a \in \N_0$ with $a^{n_a} \in H$, we infer that $\mathsf q (H) \cap F \subset \widetilde H = \widehat H$ and hence $\widehat H = \mathsf q (H) \cap F$. Furthermore,  $H \subset F$ and $\widehat H \subset F$ are cofinal, and  $\mathsf q (F)/\mathsf q (H) = F/H$ is a torsion group. Clearly, $\mathsf q (H) \cap F \subset F$ is saturated, and thus $\widehat H$ is Krull. Since ${\widehat H}^{\times} = \widehat H \cap F^{\times}$ and $H^{\times} = \widehat H^{\times} \cap H$, it follows that $H^{\times} = H \cap F^{\times}$ and then we obtain that $F^{\times}/H^{\times}$ is a torsion group.

By 1., $\widehat H/H^{\times}$ is finitely generated, say $\widehat H/H^{\times} = \{u_1H^{\times}, \ldots, u_n H^{\times} \}$, and set $P_0 = \{ p \in P \colon p \ \text{divides} \ u_1 \cdot \ldots \cdot u_n \ \text{in} \ F \}$. Then $P_0$ is finite, and we assert that $P_0=P$. If there would exist some $p \in P \setminus P_0$, then there is an $n_p \in \N$ such that $p^{n_p} \in H$ and hence $p^{n_p}H^{\times}$ is a product of $u_1H^{\times}, \ldots , u_nH^{\times}$, a contradiction. Therefore $P$ is finite,
$F/F^{\times}$ is a finitely generated monoid, $\mathsf q (F)/F^{\times}$ is a finitely generated group, and therefore $\mathsf q (F)/\mathsf q (H)F^{\times}$ is a finitely generated torsion group and thus finite. Since $\varphi \colon \widehat H \to F \to F/F^{\times}$ is a divisor homomorphism and $\mathcal C (\varphi) = \mathsf q (F)/\mathsf q (H)F^{\times}$, Proposition \ref{prop:torsion-divtheory}.1 implies that  $\mathcal C (\widehat H)$ is an epimorphic image of a submonoid of $\mathsf q (F)/\mathsf q (H)F^{\times}$ and thus $\mathcal C ( \widehat H)$ is finite. Thus 2. implies that $H$ is a {\rm C}-monoid (indeed, Property 2.(b) of Proposition \ref{2.4} holds and hence $H$ is a C-monoid defined in $F$).
\end{proof}

\subsection{\bf  Davenport constants of BF-monoids}~ \label{2.E}

Let $H$ be a BF-monoid. For every $k \in \N$, we  study the sets
\[
\mathcal M_k (H) = \{ a \in H \colon \max \mathsf L (a) \le k \} \quad \text{and} \quad
\overline{\mathcal M}_k (H) = \{ a \in H \colon \max \mathsf L (a)=k \} \,.
\]
A monoid homomorphism $| \cdot | \colon H \to (\N_0,+)$ will be called a {\it degree function} on $H$. In this section we study abstract monoids having a  degree function. The results will be applied in particular to monoids of product-one sequences and to monoids $\mathcal B (G,V)$ (see Subsections \ref{3.C} and \ref{5.D}). In all our applications the monoid $H$ will be a submonoid of a factorial monoid $F$ and if not stated otherwise the degree function on $H$ will be the restriction of the length function on $F$.

If $\theta \colon H \to B$ is a homomorphism and $H$ and $B$ have degree functions, then we say that $\theta$ is {\emph{ degree preserving}} if $|a |_H = |\theta (a)|_B$ for all $a \in H$.
Suppose we are given a  degree function on $H$ and $k \in \N$, then
\[
\mathsf D_k (H) = \sup \{ |a| \colon a \in \mathcal M_k (H) \} \in \N_0 \cup \{\infty\}
\]
is called the  {\it large $k$th  Davenport constant} of $H$ (with respect to $|\cdot|_H$). Clearly, $\mathcal M_1 (H) = \mathcal A (H) \cup H^{\times}$.  We call
$\mathsf D (H) = \mathsf D_1 (H) = \sup \{ |a| \colon a \in \mathcal A (H) \} \in \N_0 \cup \{\infty\}$ the {\it Davenport constant} of $H$. For every $k \in \N$, we have $\mathcal M_k (H) \subset \mathcal M_{k+1} (H)$,  $\mathsf D_k (H) \le \mathsf D_{k+1} (H)$, and $\mathsf D_k (H) \le k \mathsf D (H)$.
Furthermore, we have $|u|=0$ for every unit $u\in H^\times$. Therefore the degree function on $H$ induces automatically a degree function
$|\cdot|:H_{\mathrm{red}}\to (\N_0,+)$, and so the $k$th Davenport constant of $H_{\mathrm{red}}$ is defined.
Obviously we have
$\mathsf D_k(H)=\mathsf D_k(H_{\mathrm{red}})$.
Let $\mathsf e (H)$ denote the smallest  $\ell \in \N_0\cup \{\infty\}$ with the following property:
\begin{itemize}
\item[] There is a $K \in \N_0$ such that  every $a \in H$ with $|a| \ge K$ is divisible by an element $b \in H\setminus H^\times $ with $|b| \le \ell$.
\end{itemize}
Clearly, $\mathsf e (H) \le \mathsf D (H)$.

\begin{proposition} \label{2.7}
Let $H$ be a \BF-monoid and $| \cdot | \colon H \to (\N_0,+)$ be a degree function.
\begin{enumerate}
\item If $H_{\mathrm{red}}$ is finitely generated, then the sets $\mathcal M_k (H_{\mathrm{red}})$ are finite and $\mathsf D_k (H) < \infty$ for every $k \in \N$.

\smallskip
\item If \ $\mathsf D (H) < \infty$, then there exist constants $D_H, K_H \in \N_0$ such that $\mathsf D_k (H) = k \mathsf e (H) + D_H$ for all $k \ge K_H$.

\smallskip
\item If \ $\mathsf D (H) < \infty$, then the map $\N \to \Q$, $k\mapsto \frac{\mathsf D_k(H)}{k}$ is non-increasing.

\smallskip
\item Suppose that $H$ has a prime element.  Then
      \[
      \mathsf  D_k (H) \ = \ \max \bigl\{ |a| \colon a \in \overline{\mathcal M}_k (H)  \bigr\}\le \ k \mathsf D(H)
      \]
      and
      \[
      k \mathsf D(H) \  = \ \max  \bigl\{ |a| \colon  a \in H, \ \min \mathsf L (a) \le k \bigr\}
       = \ \max \bigl\{ |a| \colon a \in H, \ k \in \mathsf L (a) \bigr\}\,.
      \]
\end{enumerate}
\end{proposition}

\begin{proof}
1. Suppose that $H_{\mathrm{red}}$ is finitely generated. Then $\mathcal A (H_{\mathrm{red}})$ is finite whence  $\mathcal M_k (H)$ is finite for every $k \in \N$.
It follows that $\mathsf D (H) < \infty$ and $\mathsf D_k (H) \le k \mathsf D (H) < \infty$ for all $k \in \N$.

\smallskip
2. Suppose that $\mathsf D (H) < \infty$ and note that  $\mathsf e (H) \le \mathsf D (H)$. Let $\mathsf f (H) \in \N_0$ be the smallest $K \in \N_0$  such that every $a \in H$ with $|a| \ge K$ is divisible by an element $b \in H$ with $|b| \le \mathsf e (H)$.
We define $A = \{ a \in \mathcal A (H) \colon |a| = \mathsf e (H) \}$. Let $k \in \N$  and continue  with the following assertion.

\begin{enumerate}
\item[{\bf A.}\,] There exist $a_1, \ldots, a_k \in A$ such that $a_1 \ldots a_k \in \mathcal M_k (H)$. In particular, $\mathsf D_k (H) \ge |a_1 \ldots a_k| = k \mathsf e (H)$.
\end{enumerate}

{\it Proof of \,{\bf A}}.\, Assume to the contrary that for all $a_1, \ldots, a_k \in A$ we have $\max \mathsf L (a) > k$. Thus the product $a_1 \ldots a_k$ is divisible by an atom $u \in \mathcal A (H)$ with $|u| < \mathsf e (H)$. We set $K = \mathsf f (H) + (k-1) \mathsf e (H)$ and choose $a \in H$ with $|a| \ge K$. Then $a$ can be written in the form $a = a_1 \ldots a_k b$ where $a_1, \ldots, a_k, b \in H$ and $|a_i| \le \mathsf e (H) $ for all $i \in [1,k]$. If there is some $i \in [1,k]$ with $|a_i| < \mathsf e (H)$, then $a_i$ is a divisor of $a$ with $|a_i| < \mathsf e (H)$. Otherwise, $a_1, \ldots, a_k \in A$ and by our assumption the product $a_1 \ldots a_k$ and hence $a$ has a divisor  of degree strictly smaller than $\mathsf e (H)$. This is a contradiction to the definition of $\mathsf e (H)$. \qed{(Proof of {\bf A})}

Now let $k \ge \mathsf f (H)/\mathsf e (H) - 1$. Then {\bf A} implies that $\mathsf D_k (H) + \mathsf e (H) \ge (k+1)\mathsf e (H) \ge \mathsf f (H)$. Let $a \in H$ with $|a| > \mathsf D_k (H) + \mathsf e (H)$. Then, by definition of $\mathsf f (H)$, there are $ b, c \in H$ such that $a = b c$ with $|c| \le \mathsf e (H)$ and hence $|b| > \mathsf D_k (H)$. This implies that $\max \mathsf L (b) > k$, whence $\max \mathsf L (a) > k+1$ and $a \notin \mathcal M_{k+1} (H)$. Therefore we obtain that $\mathsf D_{k+1} (H) \le \mathsf D_k (H) + \mathsf e (H)$ and thus
\[
0 \le \mathsf D_{k+1} (H) - (k+1) \mathsf e (H) \le \mathsf D_k (H) - k \mathsf e (H) \,.
\]
Since a non-increasing sequence of non-negative integers stabilizes, the assertion follows.

\smallskip
3. Suppose that $\mathsf D (H) < \infty$. Let $k \in \N$, $a \in \mathcal M_{k+1} (H)$ with $|a|=\mathsf D_{k+1} (H)$, and set $l = \max \mathsf L (a)$. Then $l \le k+1$. If $l \le k$, then $a \in \mathcal M_k (H)$ and $\mathsf D_{k+1} (H) \ge \mathsf D_k (H) \ge |a|=\mathsf D_{k+1} (H)$ whence $\mathsf D_k (H) =\mathsf D_{k+1}(H)$. Suppose that $l=k+1$.
We set $a = a_1 \ldots a_{k+1}$ with $a_1, \ldots, a_{k+1} \in \mathcal A (H)$ and $|a_1| \ge \ldots \ge |a_{k+1}|$ whence $|a_{k+1}| \le (|a_1|+\ldots+|a_k|)/k$. It follows that
\[
\frac{\mathsf D_{k+1}(H)}{k+1} = \frac{|a_1|+\dots+|a_{k+1}|}{k+1}\le \frac{|a_1|+\dots+|a_k|}{k}\le \frac{\mathsf D_k(H)}{k} \,,
\]
where the last inequality holds because $a_1 \ldots a_k \in \mathcal M_k (H)$.

\smallskip
4. Let $p \in H$ be a prime element.  We assert that
\[
\mathsf D_k(H) \le \max \bigl\{ |a| \colon a \in H, \ \max \mathsf L(a)=k \bigr\} \,. \tag{$*$}
\]
Indeed, if $a \in \mathcal M_k(H)$ and
$\max \mathsf L(a) =l \le k$, then $a p^{k-l} \in \mathcal M_k(H)$
and
\[
|a| \le |a p^{k-l}| \le \max \bigl\{ |a| \colon a \in H, \ \max \mathsf L(a)=k \bigr\}\,,
\]
and hence $(*)$ follows. Next we assert that
\[
\max \bigl\{ |a| \colon a \in H,
\ \min \mathsf L(a) \le k \bigr\} \le k \mathsf D(H) \,. \tag{$**$}
\]
Let $a \in H$ with
$\min \mathsf L(a) =l \le k$, say $a = u_1  \ldots
u_l$, where $u_1,\ldots,  u_l \in \mathcal A(H)$. Then $|a| =
|u_1| + \ldots + |u_l| \le l \mathsf D(H) \le k \mathsf D(H)$, and thus $(**)$ follows.
Using $(*)$ and $(**)$ we infer that
\[
\begin{aligned}
\mathsf D_k (H) & \le \max \bigl\{ |a| \colon  a \in H, \ \max \mathsf L(a) =
k \bigr\} \  \le \ \max \bigl\{ |a| \colon  a \in H, \
\max \mathsf L(a) \le k \bigr\} \\ & = \ \mathsf D_k(H) \  \le \ \max
\bigl\{ |a| \colon a \in H, \ \min \mathsf L(a) \le k
\bigr\}
\end{aligned}
\]
and that
\[
k \mathsf D(H) = \max \bigl\{ |a| \colon a \in H, \ k
\in \mathsf L(a) \bigr\}  \le \max \bigl\{ |a| \colon a \in
H, \ \min \mathsf L(a) \le k \bigr\} \le k \mathsf D (H) \,.
\]
\end{proof}

Let $F$ be a factorial monoid and $H \subset F$ a submonoid such that $H^{\times} = H \cap F^{\times}$. Then $H$ is a BF-monoid by \cite[Corollary 1.3.3]{Ge-HK06a}. For $k \in \N$, let $\mathcal M_k^* (H)$ denote the set of all $a \in F$ such that $a$ is not divisible by a product of $k$ non-units of $H$.
The restriction of the usual length function $|\cdot |\colon F\to \N_0$ on $F$ (introduced in Subsection~\ref{2.A}) gives a degree function on $H$.
 We define the {\it small $k$th Davenport constant} $\mathsf d_k (H)$ as
\begin{equation} \label{defining-d_k}
\mathsf d_k (H) = \sup \{ |a| \colon a \in \mathcal M_k^* (H) \} \in \N_0 \cup \{\infty\}.
\end{equation}
In other words, $1+\mathsf d_k (H)$ is the smallest integer $\ell \in \N$ such that every  $a \in F $ of length $|a| \ge \ell$ is divisible by a product of $k$ non-units of $H$.
We call $\mathsf d(H)=\mathsf d_1(H)$ the \emph{small Davenport constant} of $H$.
Clearly we have $\mathcal{M}_k^*(H)\subset \mathcal{M}_{k+1}^*(H)$ hence $\mathsf d_k(H)\le \mathsf d_{k+1}(H)$.

Furthermore, let $\eta (H)$ denote the smallest integer $\ell \in \N\cup \{\infty\}$ such that every $a \in F$ with $|a| \ge \ell$ has a divisor $b \in H \setminus H^{\times}$ with $|b| \in [1, \mathsf e (H)]$.
For $p\in \mathcal{A}(F)$ denote by $o_p$ the smallest integer $\ell \in \N\cup \{\infty\}$ such that $p^{o_p}\in H$.
Clearly, we have $o_p\le \eta(H)$ for all $p\in \mathcal{A}(F)$.

\begin{proposition} \label{2.8}
Let $F=F^{\times} \times \mathcal F (P)$ be a factorial monoid and $H \subset F$ a submonoid such that $H^{\times} = H \cap F^{\times}$, and let $k \in \N$.
\begin{enumerate}
\item If for every $a \in F$ there is a prime $p \in F$ such that $ap \in H$, then $1+\mathsf d_k (H) \le \mathsf D_k (H)$.

\smallskip
\item Suppose that  $H_{\mathrm{red}}$ is finitely generated and that for every $a \in F$ there is an $n_a \in H$ such that $a^{n_a} \in H$.  Then $H$ is a \text{\rm C}-monoid and we have
\begin{itemize}
\item[(a)] \ $\mathsf e(H)=\max\{o_p\colon p\in P \}$ and $\eta(H)<\infty$.
\item[(b)] \ $\mathsf d_k(H)+1\ge k\mathsf e(H)$ and there exist constants $d_H\in \Z_{\ge -1}, k_H \in \N_0$ such that  $\quad \mathsf d_k (H) = k \mathsf e (H) + d_H$ for all $k \ge k_H$.
\end{itemize}
\end{enumerate}
\end{proposition}

\begin{proof}
1. Let $a \in \mathcal M_k^* (H)$ such that $|a|=\mathsf d_k(H)$. We choose a prime $p \in F$ such that $a p \in H$.
Take any factorization  $ap = u_1 \ldots u_{\ell}$ where $u_i\in \mathcal{A}(H)$.
We may assume  that  $p \t u_1$ in $F$. Then $u_2 \ldots u_{\ell} \t a$ in $F$ and hence $\ell - 1 < k$. Thus it follows that $ap \in \mathcal M_k (H)$ and
$ \mathsf D_k (H)\ge |ap|=|a|+1\ge \mathsf d_k (H)+1$.

2.(a)    By  Proposition~\ref{finitelygenerated}.3, $H$ is a C-monoid, $P$ is finite and hence $\mathsf e (H) < \infty$.
If $p \in P$, then $p^{o_p} \in \mathcal A (H)$ and by the minimality of $o_p$, $p^{o_p}$ does not have a divisor $b \in H \setminus H^{\times}$ such that $|b| < o_p$. Thus it follows that $\mathsf e(H)\ge \max\{o_p\colon p\in P\}$. 
For the reverse inequality,  note that by Proposition~\ref{2.4}.2 there exists an $\alpha\in \N$ such that for all $p\in P$ and all
$a\in p^{\alpha}F$ we have $a\in H$ if and only if $p^{\alpha}a\in H$. Since any multiple of $\alpha$ has the same property,  we may assume that
$\alpha$ is divisible by  $o_p$ for all $p \in P$. Let $b\in H$ with $|b|>|P|(2\alpha-1)$. Then there exists a $p\in P$ such that  $b\in p^{2\alpha}F\cap H$. Hence
$b$ is divisible in $H$ by $p^{\alpha}$, implying in turn that $p^{o_p} \in \mathcal{A}(H)$ divides $b$ in $H$. Therefore we obtain that
$\mathsf e(H)\le \max\{o_p \colon p\in P\}$.

If $a \in F$ with $|a|\ge \sum_{p\in P}(o_p-1)$, then there is a $p\in P$ such that $p^{o_p}$ divides $a$ in $F$, and thus $\eta(H)\le 1+ \sum_{p\in P}(o_p-1)$.

2.(b) Let $p\in P$ with $o(p)=\mathsf e(H)$. Then $p^{ko_p-1}\in \mathcal{M}_k^*(H)$ and $|p^{ko_p-1}|=k\mathsf e(H)-1$, showing the inequality
$\mathsf d_k(H)+1\ge k\mathsf e(H)$ for all $k\in \N$.
Now let $k \in \N$ be such that $1+\mathsf d_k (H)+\mathsf e (H) \ge \eta (H)$, and
let $a \in F$ with $|a| \ge  \mathsf d_k (H) +\mathsf e (H)+1$. Then, by definition of $\eta (H)$, there are $b \in F$ and $c \in H \setminus H^{\times}$ such that $a=bc$ with $|c| \le \mathsf e (H)$ and $|b| > \mathsf d_k (H)$. This implies that $b$ is divisible by a product of $k$ non-units of $H$ whence $a$ is divisible by a product of $k+1$ non-units of $H$. Therefore it follows that $1+\mathsf d_{k+1}(H) \le \mathsf d_k (H) + \mathsf e (H) + 1$ and hence
\[
0 \le \mathsf d_{k+1}(H) - k \mathsf e (H) \le \mathsf d_k (H) - (k-1)\mathsf e (H) \quad \text{for all sufficiently large} \quad k  \,.
\]
Since a non-increasing sequence of non-negative integers stabilizes, the assertion follows.
\end{proof}

\section{\bf Arithmetic Combinatorics: Zero-Sum Results with a focus on Davenport constants}  \label{sec:3}

This section is devoted   to Zero-Sum Theory, a vivid  subfield of Arithmetic Combinatorics (see  \cite{Ga-Ge06b, Ge09a, Gr13a}).
In Subsection \ref{3.A} we give an algebraic study of the monoid of product-one sequences over  finite but not necessarily abelian groups. In Subsection \ref{3.B} we put together well-known material on transfer homomorphisms used in Subsections \ref{4.A} and \ref{subsec:abelian}.   In   Subsections \ref{3.C} and \ref{3.D} we consider the $k$th Davenport constants of finite groups. In particular, we gather results which will be needed in Subsection \ref{5.A} and  results having relevance in invariant theory by Proposition \ref{prop:bschmid}.

\subsection{\bf   The monoid of product-one sequences}~ \label{3.A}

Let  $G_0 \subset G$ be a subset and let $G' = [G,G]=\langle g^{-1}h^{-1}g h \colon g, h \in G\rangle$ denote the commutator subgroup of $G$.
A {\it sequence} over $G_0$
means a finite sequence of terms from $G_0$ which is unordered and
repetition of terms is allowed, and it will be considered as an element of the free abelian monoid $\mathcal F
(G_0)$.
In order to distinguish between the group operation in $G$ and the operation in $\mathcal F (G_0)$, we use the symbol $\bdot$ for the multiplication in $\mathcal F (G_0)$, hence $\mathcal F (G_0) = \big( \mathcal F (G_0), \bdot \big)$ ---this coincides with the convention in the monographs \cite{Ge-HK06a, Gr13a}--and we denote multiplication in $G$ by juxtaposition of elements. To clarify this, if $S_1, S_2 \in \mathcal F (G_0)$ and $g_1, g_2 \in G_0$, then $S_1 \bdot S_2 \in \mathcal F (G_0)$ has length $|S_1|+|S_2|$, $S_1 \bdot g_1 \in \mathcal F (G_0)$ has length $|S_1|+1$, $g_1 \bdot g_2 \in \mathcal F (G_0)$ is a sequence of length $2$, but $g_1 g_2$ is an element of $G$. Furthermore, in order to avoid confusion between exponentiation in $G$ and exponentiation in $\mathcal F (G_0)$, we use brackets for the exponentiation in $\mathcal F (G_0)$. So for $g \in G_0$, $S \in \mathcal F (G_0)$, and $k \in \mathbb N_0$, we have
\[
g^{[k]}=\underset{k}{\underbrace{g\bdot\ldots\bdot g}}\in \mathcal F (G) \quad \text{with} \quad |g^{[k]}| = k \,, \quad \ \text{and}  \quad S^{[k]}=\underset{k}{\underbrace{S\bdot\ldots\bdot S}}\in \mathcal F (G) \,.
\]
Now let
\[
S = g_1 \bdot \ldots \bdot g_{\ell} = \prod_{g \in G_0} g^{\mathsf v_g (S)} \,,
\]
be a sequence over $G_0$ (in this notation,  we tacitly assume that $\ell \in \mathbb N_0$ and $g_1, \ldots, g_{\ell} \in G_0$).
Then $|S|=\ell = 0$ if and only if $S=1_{\mathcal F ( G_0)}$ is the identity element in $\mathcal F (G_0)$, and then $S$ will also be called the {\it trivial sequence}. The elements in $\mathcal F (G_0) \setminus \{1_{\mathcal F ( G_0)} \}$ are called {\it nontrivial sequences}.
We use all notions of divisibility theory in general free abelian monoids. Thus, for an element $g \in G_0$,  we refer to $\mathsf v_g (S)$ as the {\it multiplicity} of $g$ in $S$. A divisor $T$ of $S$ will also be called  a subsequence of $S$. We call $\supp (S) = \{ g_1, \ldots, g_{\ell}\} \subset G_0$ the {\it support} of $S$.
When  $G$ is written multiplicatively (with unit element $1_G \in G$),  we use
\[
\pi (S) = \{ g_{\tau (1)} \ldots  g_{\tau (\ell)} \in G \colon \tau\mbox{ a permutation of $[1,\ell]$} \} \subset G
\]
to denote the {\it set of products} of $S$ (if $|S|= 0$, we use the convention that $\pi (S) = \{1_G \}$). Clearly, $\pi (S)$ is contained in a $G'$-coset.
When $G$ is written additively with commutative operation, we likewise let $$\sigma(S)=g_1+\ldots+g_\ell\in G$$ denote the \emph{sum} of $S$. Furthermore, we denote by
\[
\Sigma (S) = \{\sigma(T) \colon T \t S\;\text{and} \; 1 \ne T  \}  \subset G \quad\text{and}\quad \Pi (S) = \underset{1 \ne T}{\bigcup_{T \t S}} {\pi}(T)  \subset G \,,
\]
the {\it subsequence sums } and \emph{subsequence products} of $S$.
The sequence $S$ is called
\begin{itemize}
\item a {\it product-one sequence} if $1_G \in \pi (S)$,

\item {\it product-one free} if $1_G \notin \Pi (S)$.
\end{itemize}

Every map of finite groups $\varphi \colon G_1 \to G_2$ extends to  a homomorphism $\varphi \colon \mathcal F (G_1) \to \mathcal F (G_2)$ where
$\varphi (S) = \varphi (g_1) \bdot \ldots \bdot \varphi (g_{\ell})$. If $\varphi$ is a group homomorphism, then $\varphi (S)$ is a product-one sequence if and only if $\pi (S) \cap \Ker ( \varphi) \ne \emptyset$.
We denote by
\[
\mathcal B (G_0) = \{ S \in \mathcal F (G_0) \colon 1_G \in  \pi (S) \}
\]
the set of all product-one sequences over $G_0$, and clearly $\mathcal B (G_0) \subset \mathcal F (G_0)$ is a submonoid. We will use all concepts introduced in Subsection \ref{2.E} for the monoid $\mathcal B (G_0)$ with the degree function stemming from the length function on the free abelian monoid $\mathcal F (G_0)$. For all notations $*(H)$ introduced for a monoid $H$ we write -- as usual -- $*(G_0)$ instead of $*(\mathcal B (G_0))$. In particular, for $k \in \N$, we set $\mathcal M_k (G_0)= \mathcal M_k ( \mathcal B (G_0))$, $\mathsf D_k (G_0) = \mathsf D_k (\mathcal B (G_0))$, $\eta (G_0)=\eta (\mathcal B (G_0))$,  $\mathsf e (G_0) = \mathsf e ( \mathcal B (G_0))$, and so on. By Proposition~\ref{2.8}.2(a), $\mathsf e (G_0) = \max \{ \ord (g) \colon g \in G_0 \}$. Note that $\mathcal M_1^* (G_0)$ is the set of all product-one free sequences over $G_0$.  In particular,

\[
\mathsf D (G_0) = \sup \{ |S| \colon S \in \mathcal A (G_0) \} \in \mathbb N \cup \{\infty\}
\]
is the {\it large Davenport constant} of $G_0$, and
\[
\mathsf d (G_0) = \sup \{ |S| \colon S \in \mathcal F (G_0) \ \text{is product-one free} \} \in \mathbb N_0 \cup \{\infty\}
\]
is the {\it small Davenport constant} of $G_0$. Their study will be the focus of the Subsections \ref{3.C} and \ref{3.D}.

\begin{lemma} \label{3.1}
Let  $G_0 \subset G$ be a subset.
\begin{enumerate}
\item $\mathcal B (G_0) \subset \mathcal F (G_0)$ is a reduced finitely generated submonoid, $\mathcal A (G_0)$ is finite, and $\mathsf D (G_0) \le |G|$. Furthermore, $\mathcal M_k (G_0)$ is finite and $\mathsf D_k (G_0) < \infty$ for all $k \in \N$.

\item Let $S \in \mathcal F (G)$ be product-one free.
      \begin{enumerate}

      \item If $g_0 \in \pi (S)$, then $g_0^{-1} \bdot S \in \mathcal A (G)$. In particular, $\mathsf d (G)+1 \le \mathsf D (G)$.

      \item If $|S|=\mathsf d (G)$, then $\Pi (S) = G \setminus \{1_G\}$ and hence \newline $\mathsf d (G) = \max \{ |S| \colon S \in \mathcal F (G) \ \text{with} \ \Pi (S) = G \setminus \{1_G\} \}$.
      \end{enumerate}

\item If $G$ is cyclic, then $\mathsf d (G)+1 = \mathsf D (G) = |G|$.
\end{enumerate}
\end{lemma}

\begin{proof}
1. We assert that for every $U \in \mathcal A (G)$ we have $|U| \le |G|$. Then $\mathcal A (G_0) \subset \mathcal A (G)$ is finite and $\mathsf D (G_0) \le \mathsf D (G) \le |G|$. As already mentioned, $\mathcal B (G_0) \subset \mathcal F (G_0)$ is a submonoid, and clearly $\mathcal B (G_0)^{\times} = \{1_{\mathcal F (G_0)}\}$. Since $\mathcal F (G_0)$ is factorial and $\mathcal B(G_0)^{\times} = \mathcal B (G_0) \cap \mathcal F (G_0)^{\times}$, $\mathcal B (G_0)$ is atomic by \cite[Corollary 1.3.3]{Ge-HK06a}. This means that  $\mathcal B(G_0) = [\mathcal A (G_0) \cup \mathcal B (G_0)^{\times}]$, and thus
$\mathcal B (G_0)$ is finitely generated. Since $\mathcal B (G_0)$ is reduced and finitely generated, the sets $\mathcal M_k (G_0)$ are finite by Proposition \ref{2.7}.

Now let $U \in \mathcal B (G)$, say $U = g_1 \bdot \ldots \bdot g_{\ell}$ with $g_1 g_2  \ldots  g_{\ell} = 1_G$. We suppose that $\ell > |G|$ and show that $U \notin \mathcal A (G)$. Consider the set
\[
M = \{ g_1 g_2  \ldots  g_i \colon i \in [1,\ell] \} \,.
\]
Since $\ell > |G|$, there are $i, j \in [1,\ell]$ with $i < j$ and $g_1  \ldots  g_i = g_1  \ldots  g_j$. Then $g_{i+1} \ldots  g_j=1_G$ and thus $g_1  \ldots g_i g_{j+1} \ldots g_{\ell}=1_G$ which implies that $U$ is the product of two nontrivial product-one subsequences.

\smallskip
2.(a) If $g_0 \in \pi (S)$, then $S$ can be written as $S = g_1 \bdot \ldots \bdot g_{\ell}$ such that $g_0 = g_1  \ldots  g_{\ell}$, which implies that $g_0^{-1} \bdot g_1 \bdot \ldots \bdot g_{\ell} \in \mathcal A (G)$.

2.(b) If $S$ is product-one free with $|S| = \mathsf d (G)$, and if there would be an $h \in G \setminus \{ \Pi (S) \cup \{1_G\} \}$, then $T = h^{-1} \bdot S$ would be product-one free of length $|T| = |S|+1 > \mathsf d (G)$, a contradiction. Thus every product-one free sequence $S$ of length $|S| = \mathsf d (G)$ satisfies $\Pi (S) = G \setminus \{1_G\}$.
If $S$ is a sequence with $\Pi (S) = G \setminus \{1_G\}$, then $S$ is product-one free and hence $|S| \le \mathsf d (G)$.

\smallskip
3. Clearly, the assertion holds for $|G|=1$. Suppose that $G$ is cyclic of order $n \ge 2$, and let $g \in G$ with $\ord (g) = n$. Then $g^{[n-1]}$ is product-one free, and thus 1. and 2. imply that $n \le 1+ \mathsf d (G) \le \mathsf D (G) \le n$.
\end{proof}

The next result gathers the algebraic properties of monoids of product-one sequences and highlights the difference between the abelian and the non-abelian case.

\begin{theorem} \label{3.2}
Let  $G_0 \subset G$ be a subset and let $G'$ denote the commutator subgroup of $\langle G_0 \rangle$.
\begin{enumerate}
\item $\mathcal B (G_0) \subset \mathcal F (G_0)$ is cofinal and $\mathcal B (G_0)$ is a finitely generated {\rm C}-monoid. $\widetilde{\mathcal B (G_0)}=\widehat{\mathcal B (G_0)}$ is a finitely generated Krull monoid, the embedding $\widehat{\mathcal B (G_0)} \hookrightarrow \mathcal F (G_0)$ is a cofinal divisor homomorphism with class group $\mathcal F (G_0)/\mathcal B (G_0)$, and the map
    \[
    \begin{aligned}
    \Phi \colon \mathcal F (G_0)/ \mathcal B (G_0) & \quad \longrightarrow \quad \langle G_0 \rangle / G' \\
        [S]_{\mathcal F (G_0)/\mathcal B (G_0)} & \quad \longmapsto \quad g G' \quad \text{for any} \ g \in \pi (S)
    \end{aligned}
    \]
    is a group epimorphism. Suppose that $G_0=G$. Then $\Phi$ is an isomorphism, every class of $\mathcal C ( \widehat{\mathcal B (G)})$ contains a prime divisor, and if $|G|\ne 2$, then $\widehat{\mathcal B (G)} \hookrightarrow \mathcal F (G)$ is a divisor theory.

\smallskip
\item The following statements are equivalent{\rm \,:}
      \begin{enumerate}
      \smallskip
      \item[(a)] $\mathcal B (G_0)$ is a Krull monoid.

      \smallskip
      \item[(b)] $\mathcal B (G_0)$ is root closed.

      \smallskip
      \item[(c)] $\mathcal B (G_0) \subset \mathcal F (G_0)$ is saturated.
      \end{enumerate}

\smallskip
\item $\mathcal B (G)$ is a Krull monoid if and only if $G$ is abelian.

\smallskip
\item $\mathcal B (G)$ is factorial if and only if $|G| \le 2$.
\end{enumerate}
\end{theorem}

\begin{proof}
1. $\mathcal B (G_0)$ is finitely generated by Lemma \ref{3.1}. If $n = \lcm \{ \ord (g) \colon g \in G_0 \}$, then $S^{[n]} \in \mathcal B (G_0)$ for each $S \in \mathcal F (G_0)$. Thus $\mathcal B (G_0) \subset \mathcal F (G_0)$ and $\widehat{\mathcal B (G_0)} \hookrightarrow \mathcal F (G_0)$ are cofinal,  $\mathcal F (G_0)/ \mathcal B (G_0)$ is a group and
\[
\mathcal F (G_0)/ \mathcal B (G_0) = \mathsf q \big( \mathcal F (G_0) \big) / \mathsf q \big( \mathcal B (G_0) \big) = \mathsf q \big( \mathcal F (G_0) \big) / \mathsf q \big( \widehat{\mathcal B (G_0)} \big)
\]
is the class group of the embedding $\widehat{\mathcal B (G_0)} \hookrightarrow \mathcal F (G_0)$. All statements on the structure of $\mathcal B (G_0)$ and $\widehat{\mathcal B (G_0)}$ follow from Proposition \ref{finitelygenerated}.3, and it remains to show the assertions on $\Phi$.

Let $S, S' \in \mathcal F (G_0)$, $g \in \pi (S), g' \in \pi (S')$, and $B \in \mathcal B (G_0)$. Then $\pi (S) \subset gG'$, $\pi (S') \subset g'G'$, $\pi (B) \subset G'$, and $\pi (S \bdot B) \subset gG'$. We use the abbreviation $[S]=[S]_{\mathcal F (G_0)/\mathcal B (G_0)}$, and note that $[S]=[S']$ if and only if there are $C, C' \in \mathcal B (G_0)$ such that $S \bdot C = S' \bdot C'$.

In order to show that $\Phi$ is well-defined, suppose that $[S]=[S']$ and that $S \bdot C= S \bdot C'$ with $C, C' \in \mathcal B (G_0)$. Then $\pi (S \bdot C)=\pi ( S' \bdot C') \subset gG' \cap g'G'$ and hence $gG'=g'G'$. In order to show that $\Phi$ is surjective, let $g \in \langle G_0 \rangle$ be given. Clearly, there is an $S \in \mathcal F (G_0)$ such that $g \in \pi (S)$ whence $\Phi ( [S]) =  gG'$.

Suppose that $G_0=G$. First we show that $\Phi$ is injective. Let $S, S' \in \mathcal F (G)$ with $g \in \pi (S)$, $g' \in \pi (S')$ such that $gG'=g'G'$. Then there are $k \in \N$, $a_1, b_1, \ldots, a_k, b_k \in G$ such that
\[
g {g'}^{-1} = \prod_{i=1}^k (a_i^{-1}b_i^{-1}a_ib_i) \,.
\]
We define $T= \prod_{i=1}^k (a_i^{-1} \bdot b_i^{-1} \bdot a_i \bdot b_i )$ and obtain that
\[
S \bdot (S' \bdot g^{-1} \bdot T) = S' \bdot (S \bdot g^{-1} \bdot T) \in \mathcal F (G) \,.
\]
Since $1 \in \pi (T)$ and $g {g'}^{-1} \in \pi (T)$, it follows that $1 \in \pi ( S' \bdot g^{-1} \bdot T)$ and $1 \in \pi (S \bdot g^{-1} \bdot T)$ which implies that $[S]=[S']$.

If $|G| \le 2$, then 4. will show that $\mathcal B (G)$ is factorial and clearly the trivial class contains a prime divisor. Suppose that $|G| \ge 3$. In order to show that $\widehat{\mathcal B (G)} \hookrightarrow \mathcal F (G)$ is a divisor theory, let $g \in G \setminus \{1_G\}$ be given. Then there is an $h \in G \setminus \{g^{-1}, 1_G \}$, $U = g \bdot g^{-1} \in \mathcal A (G) \subset \widehat{\mathcal B (G)}$, $U' = g \bdot h \bdot (h^{-1}g^{-1}) \in \mathcal A (G) \subset \widehat{\mathcal B (G)}$, and $g = \gcd_{\mathcal F (G)} (U, U')$.
Thus $\widehat{\mathcal B (G)} \hookrightarrow \mathcal F (G)$ is a divisor theory.

Let $S \in \mathcal F (G)$ with $g \in \pi (S)$. Then $g \in \mathcal F (G)$ is a prime divisor and we show that $[g]=[S]$. Indeed, if $g=1_G$, then $S \in \mathcal B (G)$, $1_G \in \mathcal B (G)$, $S \bdot 1_G = g \bdot S$ whence $[g]=[S]$. If $\ord (g)=n \ge 2$, then $g^{[n]} \in \mathcal B (G)$, $S \bdot g^{[n-1]} \in \mathcal B (G)$, $S \bdot g^{[n]} = g \bdot S \bdot g^{[n-1]}$ whence $[S]=[g]$.

\smallskip
2. (a) $\Rightarrow$\, (b) \ Every Krull monoid is completely integrally closed and hence root closed.

(b) \,$\Rightarrow$\, (c) \ Let $S, T \in \mathcal B (G_0)$ with $T \t S$ in $\mathcal F (G_0)$, say $S = T \bdot U$ where $U = g_1 \bdot \ldots \bdot g_{\ell} \in \mathcal F (G_0)$. If $n = \lcm \big(\ord (g_1), \ldots, \ord (g_{\ell}) \big)$, then $(T^{[-1]} \bdot S)^{[n]} = U^{[n]} \in \mathcal B (G_0)$. Since $\mathcal B (G_0)$ is root closed, this implies that $U = T^{[-1]} \bdot S \in \mathcal B (G_0)$ and hence $T \t S$ in $\mathcal B (G_0)$.

(c) \,$\Rightarrow$\, (a) \ Since $\mathcal F (G_0)$ is free abelian, $\mathcal B (G_0)$ is Krull by Theorem \ref{2.1}.

\smallskip
3. If $G$ is a abelian, then it is obvious that $\mathcal B (G) \subset \mathcal F (G)$ is saturated, and thus $\mathcal B (G)$ is a Krull monoid by 2.
Suppose that $G$ is not abelian. Then there are $g, h \in G$ with $g h \ne h g$. Then $g h g^{-1} \ne h$, $S = g \bdot h \bdot g^{-1} \bdot (g h g^{-1} )^{-1}  \in \mathcal B (G)$, $T = g \bdot g^{-1} \in \mathcal B (G)$ divides $S$ in $\mathcal F (G)$ but $T^{[-1]} \bdot S = h \bdot (g h g^{-1})^{-1} $ does  not have product-one. Thus $\mathcal B (G) \subset \mathcal F (G)$ is not saturated and hence $\mathcal B (G)$ is  not Krull by 2.

\smallskip
4. If $G = \{0\}$, then $\mathcal B (G)=\mathcal F (G)$ is factorial. If $G = \{0,g\}$, then $\mathcal A (G) = \{0, g^{[2]} \}$, each atom is a prime, and $\mathcal B (G)$ is factorial. Conversely, suppose that $\mathcal B (G)$ is factorial. Then $\mathcal B (G)$ is a Krull monoid by \cite[Corollary 2.3.13]{Ge-HK06a}, and hence $G$ is abelian by 3. Suppose that $|G| \ge 3$. We show that $\mathcal B (G)$ is not factorial. If there is an element $g \in G$ with $\ord (g)=n \ge 3$, then $U=g^{[n]}, -U= (-g)^{[n]}, W= (-g)\bdot g \in \mathcal A (G)$, and $U \bdot (-U) = W^{[n]}$. Suppose there is no $g \in G$ with $\ord (g) \ge 3$. Then there are $e_1, e_2 \in G$ with $\ord (e_1)=\ord (e_2)=2$ and $e_1+e_2\ne 0$. Then $U = e_1 \bdot e_2 \bdot (e_1+e_2) , W_1=e_1^{[2]}, W_2=e_2^{[2]}, W_0 = (e_1+e_2)^{[2]} \in \mathcal A (G)$, and $U^{[2]}= W_0 \bdot W_1 \bdot W_2$.
\end{proof}

For a subset $G_0 \subset G$, the monoid $\mathcal B (G_0)$ may be Krull or just seminormal but it need not be Krull. We provide examples for both situations.

\begin{proposition} \label{3.4}
Let  $G_0 \subset G$ be a subset satisfying the following property {\bf P}{\rm \,:}
\begin{itemize}
      \smallskip
      \item[{\bf P.}] For each two elements $g, h \in G_0$,  $\langle h \rangle \subset \langle g, h \rangle$ is normal or $\langle g \rangle \subset \langle g, h \rangle$ is normal.
      \end{itemize}
      \smallskip
Then $\mathcal B (G_0)$ is a Krull monoid if and only if $\langle G_0 \rangle$ is abelian.
\end{proposition}

\begin{proof}
If $\langle G_0 \rangle$ is a abelian, then it is obvious that $\mathcal B (G_0) \subset \mathcal F (G_0)$ is saturated, and thus $\mathcal B (G_0)$ is Krull by Theorem \ref{3.2}.2.

Conversely, suppose that $\mathcal B (G_0)$ is Krull and that $G_0$ satisfies Property {\bf P}. In order to show that $\langle G_0 \rangle$ is abelian, it is sufficient to prove that $g h = hg$ for each two elements $g, h \in G_0$. Let $g, h \in G_0$ be given such that $\langle h \rangle \subset \langle g, h \rangle$ is normal, $\ord (g)=m$, $\ord (h)=n$, and assume to the contrary that $ghg^{-1} \ne h$. Since $g \langle h \rangle g^{-1} = \langle h \rangle$, it follows that $ghg^{-1} = h^{\nu}$ for some $\nu \in [2, n-1]$. Thus $g h g^{m-1}h^{n-\nu}=1$ and $S = g \bdot h \bdot g^{[m-1]} \bdot h^{[n-\nu]} \in \mathcal B (G_0)$. Clearly, $T = g^{[m]} \in \mathcal B (G_0)$ but $S \bdot T^{[-1]} = h^{[n-\nu+1]} \notin \mathcal B (G_0)$. Thus $\mathcal B (G_0) \subset \mathcal F (G_0)$ is not saturated, a contradiction.
\end{proof}

\begin{proposition} \label{3.5}
Let  $G = D_{2n}$ be the dihedral group, say $G = \langle a, b \rangle =$ \newline $\{1, a, \ldots, a^{n-1}, b, ab, \ldots, a^{n-1}b \}$,  where $\ord (a) = n \ge 2$, $\ord (b) = 2$, and set $G_0 = \{ab, b\}$. Then $\mathcal B (G_0)$ is a Krull monoid if and only if $n$ is even.
\end{proposition}

\begin{proof}
Clearly, we have $\ord (ab)=\ord(b)=2$ and $\langle G_0 \rangle = G$.
Suppose that $n$ is odd and consider the sequence $S = (ab)^{[n]} \bdot b^{[n]}$. Since $\big( (ab) b \big)^n = 1$, it follows that $S$ is a product-one sequence. Obviously, $S_1 = (ab)^{[n-1]} \bdot b^{[n-1]} \in \mathcal B (G_0)$ and $S_2 = (ab) \bdot b \notin \mathcal B (G_0)$. Since $S = S_1 \bdot S_2$, it follows that $\mathcal B (G_0) \subset \mathcal F (G_0)$ is not saturated, and hence $\mathcal B (G_0)$ is not Krull by Theorem \ref{3.2}.2.

Suppose that $n$ is even. Then $\mathcal A (G_0) = \{ (ab)^{[2]}, b^{[2]} \}$ and $\mathcal B (G_0) = \{ (ab)^{[\ell]} \bdot b^{[m]} \colon \ell, m \in \N_0 \ \text{even} \}$.
This description of $\mathcal B (G_0)$ implies immediately that $\mathcal B (G_0) \subset \mathcal F (G_0)$ is saturated, and hence $\mathcal B (G_0)$ is Krull by Theorem \ref{3.2}.2.
\end{proof}

\begin{remark} ({\bf Seminormality of $\mathcal B (G_0)$}) \label{3.3}
A monoid $H$ is called seminormal if for all $x \in \mathsf q (H)$ with $x^2, x^3 \in H$ it follows that $x \in H$. Thus, by definition, every root closed monoid is seminormal.

\smallskip
1. Let $n \equiv 3 \mod 4$ and $G = D_{2n}$ the dihedral group, say $G = \langle a, b \rangle =$ \newline $ \{1, a, \ldots, a^{n-1}, b, ab, \ldots, a^{n-1}b \}$,  where $\ord (a) = n$, $\ord (b) = 2$, and
\[
a^k b a^l b = a^{k-l} \quad \text{for all} \ k, l \in \mathbb Z \,.
\]
We consider the sequence
\[
S = a^{\big[ \frac{n-1}{2}\big]}  \bdot b^{[2]} \in \mathcal F (G) \,.
\]
Then
\[
S^{[2]} = \big( a^{\big[ \frac{n-1}{2}\big]}  \bdot b \bdot a^{\big[ \frac{n-1}{2}\big]}  \bdot b \big) \bdot ( b \bdot b )  \ \text{and} \
S^{[3]} = a^{[n]} \bdot \big( a^{\big[ \frac{n-3}{4}\big]}  \bdot b \bdot a^{\big[ \frac{n-3}{4}\big]}  \bdot b \big) \bdot b^{[4]}
\]
are both in $\mathcal B (G)$ whence $S \in \mathsf q \big( \mathcal B (G) \big)$, but obviously $S \notin \mathcal B (G)$. Thus $\mathcal B (\{a,b\})$ and $\mathcal B (G)$ are not seminormal.

\smallskip
2. Let $G = H_8 = \{E, I, J, K, -E, -I, -J, -K \}$ be the quaternion group with the relations
\[
IJ = -JI = K, \ JK = -KJ = I , \quad \text{and} \quad KI = - I K = J \,,
\]
and set $G_0 = \{I, J\}$. By Theorem \ref{3.2},  $\mathcal B (G)$ is  not Krull and by Proposition \ref{3.4}, $\mathcal B (G_0)$ is not Krull. However, we  assert that $\mathcal B (G_0)$ is seminormal.

First, we are going to derive an explicit description of $\mathcal B (G_0)$.
Since $E = (-E)(-E) = (K K)(I I) = (IJ)(IJ)(II)$, it follows that $U = I^{[4]}\bdot J^{[2]} \in \mathcal B (G_0)$.  Assume
that $U = U_1 \cdot U_2$ with $U_1, U_2 \in \mathcal A (G_0)$ and $|U_1| \le |U_2|$. Then $|U_1| \in \{2,3\}$, but $U$ does not have a subsequence with product one and length two or three. Thus $U \in \mathcal A (G_0)$ and similarly we obtain that $I^{[2]} \bdot J^{[4]} \in \mathcal A (G_0)$.
Since $\mathsf D (G_0) \le \mathsf D (G)=6$, it is  easy to check that
\[
\mathcal A (G_0) = \{ I^{[4]}, J^{[4]}, I^{[2]} \bdot J^{[2]}, I^{[4]}\bdot J^{[2]}, I^{[2]} \bdot J^{[4]} \} \,.
\]
This implies that
\[
\mathcal B (G_0) = \{ I^{[k]} \bdot J^{[l]} \colon k=l=0 \ \text{or} \ k,l \in \N_0 \ \text{are both even with} \ k+l\ge 4 \} \,.
\]

In order to show that $\mathcal B (G_0)$ is seminormal, let $x \in \mathsf q \big( \mathcal B (G_0) \big)$ be given such that $x^{[2]}, x^{[3]} \in \mathcal B (G_0)$. We have to show that $x \in \mathcal B (G_0)$. Since $x^{[2]}, x^{[3]} \in \mathcal B (G_0) \subset \mathcal F (G_0)$ and $\mathcal F (G_0)$ is seminormal, it follows that $x \in \mathcal F (G_0)$. If $x = I^{[k]}$ with $k \in \N_0$, then $I^{[3k]} \in \mathcal B (G_0)$ implies that $4 \t 3k$, hence $4 \t k$, and thus $x \in \mathcal B (G_0)$. Similarly, if $x = J^{[l]} \in \mathcal B (G_0)$ with $l \in \N_0$, then $x \in \mathcal B (G_0)$. It remains to consider the case $x = I^{[k]} \bdot J^{[l]}$ with $k, l \in \N$. Since $x^{[3]} = I^{[3k]} \bdot J^{[3l]} \in \mathcal B (G_0)$, it follows that $k,l$ are both even, and thus $x \in \mathcal B (G_0)$. Therefore $\mathcal B (G_0)$ is seminormal.
\end{remark}

\subsection{\bf  Transfer Homomorphisms}~ \label{3.B}

A well-established strategy for investigating the arithmetic of a given monoid $H$ is to construct a transfer homomorphism $\theta \colon H \to B$, where $B$ is a simpler monoid than $H$ and the
transfer homomorphism $\theta$ allows to shift arithmetical results from  $B$ back to the (original, more complicated) monoid $H$. We will use transfer homomorphisms in Section \ref{sec:4} in order to show that properties of the monoid of $G$-invariant monomials can be studied in a monoid of zero-sum sequences (see Propositions \ref{prop:bschmid} and \ref{prop:diagram}).

\begin{definition} \label{3.6}
A monoid homomorphism \ $\theta \colon H \to B$ \ is called a \ {\it transfer homomorphism} \ if it has the following properties:

\smallskip
\item[]
\begin{enumerate}
\item[{\bf (T\,1)\,}] $B = \theta(H) B^\times$ \ and \ $\theta
^{-1} (B^\times) = H^\times$.

\smallskip

\item[{\bf (T\,2)\,}] If $u \in H$, \ $b,\,c \in B$ \ and \ $\theta
(u) = bc$, then there exist \ $v,\,w \in H$ \ such that \ $u = vw$,
\ $\theta (v) B^{\times} =  b B^{\times}$ \ and \ $\theta (w) B^{\times} = c B^{\times}$.
\end{enumerate}
\end{definition}

We will use the simple fact that, if $\theta \colon H \to B$ and $\theta' \colon B \to B'$ are transfer homomorphisms, then their composition $\theta' \circ \theta \colon H \to B'$ is a transfer homomorphism too.  The next proposition summarizes key properties of transfer homomorphisms.

\begin{proposition} \label{3.7}
Let $\theta \colon H \to B$ be a transfer homomorphism and $a \in H$.
\begin{enumerate}
\item  $a$ is an atom of $H$ if and only if $\theta (a)$ is an atom of $B$.

\smallskip
\item $\mathsf L_H (a) = \mathsf L_B \big( \theta (a) \big)$, whence $\theta \big( \mathcal M_k (H) \big)  = \mathcal M_k (  B  )$ and $\theta^{-1} \big( \mathcal M_k (  B  ) \big) = \mathcal M_k (H)$.

\smallskip
\item If $\theta$ is degree preserving, then $\mathsf D_k (H) = \mathsf D_k (B)$ for all $k \in \N$.
\end{enumerate}
\end{proposition}

\begin{proof}
1. and 2. follow from \cite[Proposition 3.2.3]{Ge-HK06a}. In order to prove 3., note that for all $k \in \N$ we have
\[
\begin{aligned}
\mathsf D_k (H) & = \sup \{ |a|_H \colon a \in \mathcal M_k (H) \}   = \sup \{ |\theta (a)|_B \colon \theta (a) \in \mathcal M_k (B) \} \\
 &  = \sup \{ |b|_B \colon b \in \mathcal M_k (B) \} = \mathsf D_k (B) \,.
\end{aligned}
\]
\end{proof}

The first examples of  transfer homomorphisms in the literature  starts from a Krull monoid to its associated monoid of zero-sum sequences which is a Krull monoid having a combinatorial flavor. These ideas were generalized widely, and there are transfer homomorphisms from weakly Krull monoids to (simpler) weakly Krull monoids (having a combinatorial flavor) and the same is true for C-monoids.

\begin{proposition} \label{3.8}
Let $H$ be a Krull monoid, $\varphi \colon H \to \mathcal F (P)$ be a cofinal divisor homomorphism with class group $G = \mathcal C ( \varphi)$, and let $G^* \subset G$ denote the set of classes containing prime divisors. Let $\widetilde{\theta} \colon \mathcal F (P) \to \mathcal F (G^*)$ denote the unique homomorphism defined by $\widetilde \theta (p) = [p]$ for all $p \in P$, and set $\theta = \widetilde{\theta} \circ \varphi \colon H \to \mathcal B (G^*)$.
\begin{enumerate}
\item $\theta$ is a transfer homomorphism.

\item For $a \in H$, we set $|a|=|\varphi (a)|$ and for $S \in \mathcal B (G^*)$ we set $|S|=|S|_{\mathcal F (G^*)}$. Then $|a|=|\theta (a)|$ for all $a \in H$, $\theta (\mathcal M_k^* (H)) = \mathcal M_k^* (G^*)$ and $\theta^{-1} ( \mathcal M_k^* (G^*))=\mathcal M_k^* (H)$ for all $k \in \N$. Furthermore, $\mathsf e (H) = \mathsf e (G^*)$, $\eta (H) = \eta (G^*)$, and $\mathsf D_k (H) = \mathsf D_k (G^*)$ for all $k \in \N$.
\end{enumerate}
\end{proposition}

\begin{proof}
1. follows from  \cite[Section 3.4]{Ge-HK06a}. By definition, we have $|a|=|\theta (a)|$ for all $a \in H$. Thus the assertions on $\mathsf D_k (H)$ follow from Proposition \ref{2.7}, and the remaining statements can be derived in a similar way.
\end{proof}

The above transfer homomorphism $\theta \colon H \to \mathcal B (G^*)$ constitutes the link between the arithmetic of Krull monoids on the one side and zero-sum theory on the other side. In this way methods from  Arithmetic Combinatorics can be used to obtain precise results for arithmetical invariants describing the arithmetic of $H$. For an overview of this interplay see \cite{Ge09a}.

There is a variety of transfer homomorphisms from monoids of zero-sum sequences to monoids of zero-sum sequences in order  to simplify specific structural features of the involved subsets of groups. Below we present a simple example of such a transfer homomorphism  which we will meet again in Proposition \ref{prop:diagram} (for more of this nature we refer to \cite{Sc10a} and to \cite[Theorem 6.7.11]{Ge-HK06a}).
Let $G$ be  abelian  and let $G_0 \subset G$ be a subset. For $g \in G_0$ we define
\[
e (G_0, g) = \gcd \big( \{ \mathsf v_g (B) \colon B \in \mathcal B (G_0) \} \big) \,,
\]
and it is easy to check that (for details see \cite[Lemma 3.4]{Ge-Zh15a})
\[
\begin{aligned}
      &  e (G_0, g) =  \gcd \big( \{ \mathsf v_g (A) \colon A \in \mathcal A (G_0) \} \big)\\ =&
      \min  \big( \{ \mathsf v_g (A) \colon \mathsf v_g (A)>0,  A  \in \mathcal A (G_0) \} \big)   =  \min \big( \{ \mathsf v_g (B) \colon \mathsf v_g (B) > 0, B \in \mathcal B (G_0) \} \big) \\ =& \min  \big( \{ k \in \N  \colon kg \in \langle G_0 \setminus \{g\} \rangle  \} \big)=\gcd \big(  \{ k \in \N  \colon kg \in \langle G_0 \setminus \{g\} \rangle  \} \big) \,.
\end{aligned}
\]

\begin{proposition} \label{3.9}
Let $G$ be  abelian  and $G_0, G_1, G_2  \subset G$ be  subsets such that $G_0 = G_1 \uplus G_2$. For $g \in G_0$ we set $e (g) = e (G_0, g)$ and we define
$G_0^* = \{ e(g) g \colon g \in G_1 \} \cup G_2$. Then the map
\[
\begin{aligned}
\theta \colon \mathcal B (G_0) & \quad \longrightarrow \quad \mathcal B (G_0^*) \\
 B = \prod_{g \in G_0} g^{[\mathsf v_g (B)]} & \quad \longmapsto \quad \prod_{g \in G_1} (e(g) g)^{[\mathsf v_g (B) / e(g)]} \prod_{g \in G_2} g^{[\mathsf v_g (B)]}
\end{aligned}
\]
is a transfer homomorphism.
\end{proposition}

\begin{proof}
Clearly, $\theta$ is a surjective homomorphism satisfying $\theta^{-1} (1_{\mathcal F (G_0)})=\{1_{\mathcal F (G_0)} \}$. In order to verify property {\bf (T2)} of Definition \ref{3.6}, let $B \in \mathcal B (G_0)$ and $C_1, C_2 \in \mathcal B (G_0^*)$ be such that $\theta (B) = C_1 \bdot C_2$. We have to show that there are $B_1, B_2 \in \mathcal B (G_0)$ such that $B = B_1 \bdot B_2$, $\theta (B_1)=C_1$, and $\theta (B_2)=C_2$.
This can be checked easily.
\end{proof}

\subsection{\bf  The $k$th Davenport constants: the general case} \label{3.C}~

Let  $G_0 \subset G$ be a subset, and $k \in \N$. Recall that $\mathsf e (G) = \max \{ \ord (g) \colon g \in G \}$. If $G$ is nilpotent, then $G$ is the direct sum of its $p$-Sylow subgroups and hence $\mathsf e (G) = \lcm \{ \ord (g) \colon g \in G \} =  \exp (G)$. Let
\begin{itemize}
\item $\mathsf E (G_0)$ be the smallest integer $\ell \in \mathbb N$ such that every sequence $S \in \mathcal F (G_0)$ of length $|S| \ge \ell$ has a product-one subsequence of length $|G|$.

\smallskip
\item $\mathsf s (G_0)$ denote the smallest integer $\ell \in \N$ such that every sequence $S \in \mathcal F (G_0)$ of length $|S|\ge \ell$ has a  product-one subsequence  of length $\mathsf e (G)$.

\end{itemize}

The Davenport constants, together with the Erd{\H{o}}s-Ginzburg-Ziv constant $\mathsf s (G)$, the constants $\eta (G)$ and $\mathsf E (G)$, are the most classical zero-sum invariants whose study (in the abelian setting) goes back to the early 1960s. The  $k$th Davenport constants $\mathsf D_k (G)$ were introduced by Halter-Koch \cite{HK92c} and further studied in \cite[Section 6.1]{Ge-HK06a} and \cite{Fr-Sc10} (all this work is done in the abelian setting). First results in the non-abelian setting were achieved in \cite{De-Or-Qu01}.

If $G$ is abelian, then W. Gao proved that $\mathsf E (G) = |G| + \mathsf d (G)$. For cyclic groups this is the Theorem of Erd{\H{o}}s-Ginzburg-Ziv which dates back to 1961 (\cite[Proposition 5.7.9]{Ge-HK06a}).  W. Gao and J. Zhuang conjectured that the above equality holds true for all finite groups (\cite[Conjecture 2]{Zh-Ga05a}), and their conjecture has been verified in a variety of special cases \cite{Ba07b, Ga-Lu08a, Ga-Li10b, Ha15a}. For more in the non-abelian setting see
\cite{Wa-Zh12a,Wa-Qu14a}.

We verify two simple properties occurring in the assumptions of Propositions \ref{2.7} and \ref{2.8}.
\begin{itemize}
\item If $S \in \mathcal F (G)$ and $g_0 \in \pi (S)$, then $h = g_0^{-1} \in G$ is a prime in $\mathcal F (G)$ and $h \bdot S \in \mathcal B (G)$.

\item Clearly, $1_G \in \mathcal B (G)$ is  a prime element of $\mathcal B (G)$.
\end{itemize}
Therefore all properties proved in Propositions \ref{2.7} and \ref{2.8} for $\mathsf D_k (H)$ and $\mathsf d_k (H)$ hold for the constants $\mathsf D_k (G)$ and $\mathsf d_k (G)$ (the linearity properties as given in Proposition \ref{2.7}.2 and Proposition \ref{2.8}.2.(b) were first proved by Freeze and W.A. Schmid in case of abelian groups $G$  \cite{Fr-Sc10}).
We continue with properties which are more specific.

\begin{proposition} \label{gen-dav-5}
Let  $H \le G$ be a subgroup, $N \triangleleft G$ be a normal subgroup, and $k, \ell \in \N$.
\begin{enumerate}
\item $\mathsf d_k (N) + \mathsf d_{\ell} (G/N) \le \mathsf d_{k+\ell-1} (G)$.

\smallskip
\item  $\mathsf d_k (G) \le \mathsf d_{\mathsf d_k (N)+1} (G/N)$.

\smallskip
\item $\mathsf d_k (G)+1 \le [G \DP H] (\mathsf d_k (H)+1)$.

\smallskip
\item $\mathsf d_k (G)+1 \le k ( \mathsf d (G)+1)$.

\smallskip
\item $\mathsf D_k (G) \le [G \DP H] \mathsf D_k (H)$.
\end{enumerate}
\end{proposition}

\begin{proof}
1. Let $\overline S = (g_1N) \bdot \ldots \bdot (g_{s}N) \in \mathcal M_{\ell}^* (G/N)$ with $|\overline S|=s=\mathsf d_{\ell} (G/N)$ and let $T = h_1 \bdot \ldots \bdot h_t \in \mathcal M_k^* (N)$ with $t=\mathsf d_k (N)$. We consider the sequence $W = g_1 \bdot \ldots \bdot g_s \bdot h_1 \bdot \ldots \bdot h_t \in \mathcal F (G)$ and suppose that it is divisible by $S_1 \bdot \ldots \bdot S_a \bdot T_1 \bdot \ldots \bdot T_b$ where $S_i, T_j \in \mathcal B (G) \setminus \{1_{\mathcal F (G)}\}$, $\supp (S_i) \cap \{g_1, \ldots, g_s\} \ne \emptyset$ and $T_1 \bdot \ldots \bdot T_b \t h_1 \bdot \ldots \bdot h_t$ for all $i \in [1, a]$ and all $j \in [1,b]$. For $i \in [1,a]$, let $\overline{S_i} \in \mathcal F (G/N)$ denote the sequence obtained from $S_i$ by replacing each $g_{\nu}$ by $g_{\nu}N$ and by omitting the elements of $S_i$ which lie in $\{h_1, \ldots, h_t\}$. Then $\overline{S_1}, \ldots, \overline{S_a} \in \mathcal B(G/N)\setminus \{1_{\mathcal F (G)}\}$ and $\overline{S_1} \bdot \ldots \bdot \overline{S_a} \t \overline S$ whence $a \le \ell-1$. By construction, we have $b \le k-1$ whence $a+b < k+\ell-1$, $W \in \mathcal M_{k+\ell-1}^* (G)$, and $|W|=s+t=\mathsf d_k (N)+\mathsf d_{\ell}(G/N) \le \mathsf d_{k+\ell-1} (G)$.

\smallskip
2. We set $m = \mathsf d_{\mathsf d_k (N)+1} (G/N)+1$. By \eqref{defining-d_k}, we have to show that every sequence $S$ over $G$ of length $|S| \ge m$ is divisible by a product of $k$ nontrivial product-one sequences. Let $f \colon G \to G/N$ denote the canonical epimorphism and let $S \in \mathcal F (G)$ be a sequence of length $|S| \ge m$. By definition of $m$, there exist sequences $S_1, \ldots, S_{\mathsf d_k (N)+1}$ such that $S_1 \bdot \ldots \bdot S_{\mathsf d_k (N)+1} \t S$ and $f (S_1), \ldots, f (S_{\mathsf d_k (N)+1})$ are product-one sequences over $G/N$.
Thus, for each $\nu \in [1, \mathsf d_k (N)+1]$, there are elements $h_{\nu} \in N$ such that $h_{\nu} \in \pi (S_{\nu})$. Then $T = h_1 \bdot \ldots \bdot h_{\mathsf d_k (N)+1}$ is a sequence over $N$, and it has $k$ nontrivial product-one subsequences $T_1, \ldots, T_k$ whose product $T_1 \bdot \ldots \bdot T_k$ divides $T$. Therefore we obtain $k$ nontrivial product-one sequences whose product divides $S$.

\smallskip
3. We set $m = [G \DP H]$ and start with the following assertion.

\begin{enumerate}
\item[{\bf A.}\,] If $S \in \mathcal F (G)$ with $|S| \ge m$, then $\Pi (S) \cap H \ne \emptyset$.
\end{enumerate}

{\it Proof of \,{\bf A}}.\, Let $S = g_1 \bdot \ldots \bdot g_n \in \mathcal F (G)$ with $|S|=n \ge m$. We consider the left cosets $g_1H, g_1g_2H, \ldots, $ $g_1 \ldots g_m H$. If one of these cosets equals $H$, then we are done. If this is not the case, then there are $k, \ell \in [1,m]$ with $k < \ell$ such that $g_1 \ldots g_kH = g_1 \ldots g_kg_{k+1} \ldots g_{\ell}H$ which implies that $g_{k+1} \ldots g_{\ell} \in H$. \qed{(Proof of {\bf A})}

Now let $S \in \mathcal F (G)$ be a sequence of length $|S| = [G \DP H] (\mathsf d_k (H)+1)$. We have to show that $S$ is divisible by a product of $k$ nontrivial product-one sequences. By {\bf A}, there are $\mathsf d_k (H)+1$ sequences $S_1, \ldots, S_{\mathsf d_k (H)+1}$ and elements $h_1, \ldots, h_{\mathsf d_k (H)+1} \in H$  such that $S_1 \bdot \ldots \bdot S_{\mathsf d_k (H)+1} \t S$ and $h_{\nu} \in \pi (S_{\nu})$ for each $\nu \in [1, \mathsf d_k (H)+1]$. By definition, the sequence $h_1 \bdot \ldots \bdot h_{\mathsf d_k (H)+1} \in \mathcal F (H)$ is divisible by a product of $k$ nontrivial product-one sequences. Therefore $S$ is divisible by a product of $k$ nontrivial product-one sequences.

\smallskip
4. Let $S \in \mathcal F (G)$ be a sequence of length $|S|=k ( \mathsf d (G)+1)$. Then $S$ may be written as a product $S = S_1 \bdot \ldots \bdot S_k$ where $S_1, \ldots , S_k \in \mathcal F (G)$ with $|S_{\nu}|=\mathsf d (G)+1$ for every $\nu \in [1,k]$. Then each $S_{\nu}$ is divisible by a nontrivial product-one sequence $T_{\nu}$ and hence $S$ is divisible by $T_1 \bdot \ldots \bdot T_k$. Thus by \eqref{defining-d_k} we infer  that $\mathsf d_k (G)+1 \le k ( \mathsf d (G)+1)$.

\smallskip
5. Let $A = g_1\bdot\ldots\bdot g_{\ell}\in\mathcal{B}(G)$ with $g_1\dots g_{\ell}=1$ and $\ell> [G:H]\mathsf D_k(H)$. We show that  $\ell>\mathsf D_k(G)$.
We set $d=\mathsf D_k(H)$ and consider the left $H$-cosets $C_j=g_1\dots g_jH$ for each $j\in [1,\ell]$. By the pigeonhole principle there exist
$1\le i_1<\dots<i_{d+1}\le\ell$ such that $C_{i_1}=\dots =C_{i_{d+1}}$.
We set $h_s =g_{i_s+1}\dots g_{i_{s+1}}$ for each $s\in [1,d]$ and $h_{d+1}=g_{i_{d+1}+1}\ldots g_{\ell}g_1\dots g_{i_1-1}$.
Clearly $h_1,\dots,h_{d+1}\in H$, and $g_1\cdots g_{\ell}=1$ implies $h_1\cdots h_{d+1}=1$ whence
$h_1\bdot \ldots \bdot h_{d+1}\in \mathcal{B}(H)$.
The inequality $d+1> \mathsf D_k(H)$ implies that $h_1\bdot \ldots \bdot h_{d+1}=S_1\bdot \ldots \bdot S_{k+1}$, where $1_{\mathcal F (H)} \ne S_i\in \mathcal{B}(H)$ for $i\in [1,k+1]$.
Let $T_i\in  \mathcal{F}(G)$ denote the sequence obtained from $S_i$ by replacing each occurrence of $h_s$ by  $g_{i_s+1}\bdot \ldots \bdot g_{i_{s+1}}$ for $s\in [1,d]$ and
$h_{d+1}$ by $g_{i_{d+1}+1}\bdot \ldots \bdot g_{\ell}\bdot g_1\bdot \ldots \bdot g_{i_1-1}$. Then $T_1, \ldots , T_{k+1} \in \mathcal{B}(G)$ and
 $A = g_1\bdot\ldots\bdot g_{\ell}=T_1\bdot\ldots\bdot T_{k+1}$, which implies that $\ell>\mathsf D_k(G)$.
\end{proof}

Much more is known for the classical Davenport constants $\mathsf D_1 (G)=\mathsf D (G)$ and $\mathsf d_1 (G)=\mathsf d(G)$. We start with metacyclic groups of index two. The following result was  proved in \cite[Theorem 1.1]{Ge-Gr13a}.

\begin{theorem}\label{3.12}
Suppose that $G$ has a cyclic, index $2$ subgroup.  Then
\[
\mathsf D(G)=\mathsf d(G)+|G'|\quad\text{and}\quad\mathsf d(G)=\left\{
                                                            \begin{array}{ll}
                                                              |G|-1 & \hbox{if $G$ is cyclic} \\
                                                              \frac12|G| & \hbox{if $G$ is non-cyclic,}
                                                            \end{array}
                                                          \right.
\]
where $G'=[G,G]$ is the commutator subgroup of $G$.
\end{theorem}

The next result  gathers upper bounds for the large Davenport constant (for $\mathsf d (G)$ see \cite{Ga-Li-Pe14a}).

\begin{theorem} \label{3.13}
Let  $G' = [G,G]$ denote the  commutator subgroup of $G$.
\begin{enumerate}
\item $\mathsf D (G) \le \mathsf d (G) + 2 |G'| -1$, and equality holds if and only if $G$ is abelian.

\item If $G$ is a non-abelian $p$-group, then $\mathsf D (G) \le \frac{p^2+2p-2}{p^3} |G|$.

\item If $G$ is non-abelian of order $pq$, where $p, q$ are primes with $p < q$, then $\mathsf D (G) = 2q$ and $\mathsf d (G) = q+p-2$.

\item If $N \triangleleft G$ is a normal subgroup with $G/N \cong C_p \oplus C_p$ for some prime $p$, then
      \[
      \mathsf d (G) \le (\mathsf d (N) +2)p - 2 \le \frac{1}{p}|G|+p-2 \,.
      \]

\item If $G$ is non-cyclic and $p$ is the smallest prime  dividing $|G|$, then $\mathsf D (G) \le \frac{2}{p}|G|$.

\item If $G$ is neither cyclic nor isomorphic to a dihedral group of order $2n$ with odd $n$, then $\mathsf D (G) \le \frac{3}{4}|G|$.
\end{enumerate}
\end{theorem}

\begin{proof}
All results can be found in \cite{Gr13b}: see Lemma 4.2, Theorems 3.1, 4.1, 5.1,7.1, 7.2, and  Corollary 5.7.
\end{proof}

\begin{corollary} \label{3.14}
The following statements are equivalent{\rm \,:}
\begin{enumerate}
\item[(a)] $G$ is cyclic or isomorphic to a dihedral group of order $2n$ for some odd $n \ge 3$.

\smallskip
\item[(b)] $\mathsf D (G) = |G|$.
\end{enumerate}
\end{corollary}

\begin{proof}
If $G$ is not as in (a), then $\mathsf D (G) \le \frac{3}{4}|G|$ by Theorem \ref{3.13}.6. If $G$ is cyclic, then $\mathsf D (G) = |G|$ by Lemma \ref{3.1}.3. If $G$ is dihedral of order $2n$ for some odd $n \ge 3$, then the commutator subgroup $G'$ of $G$ has order $n$ and hence $\mathsf D (G)=|G|$ by Theorem \ref{3.12}.
\end{proof}

\subsection{\bf  The $k$th  Davenport constants: The abelian case}~ \label{3.D}

\centerline{\it Throughout this subsection, all groups are abelian and will be written additively.}

\smallskip
We have $G \cong C_{n_1} \oplus \ldots \oplus C_{n_r}$, with $r \in \mathbb N_0$ and $1 < n_1 \t \ldots \t n_r$, $\mathsf r (G) = r$ is the {\it rank} of $G$ and  $n_r = \exp (G)$ is the {\it exponent} of $G$. We define
\[
\mathsf d^* (G) = \sum_{i=1}^r (n_i-1)  \,.
\]
If $G = \{0\}$, then $r=0 = \mathsf d^* (G)$.  An $s$-tuple $(e_1, \ldots, e_s)$ of elements of
$G \setminus \{0\}$  is said to be
 a \ {\it basis} \ of $G$ if $G = \langle e_1 \rangle \oplus \ldots \oplus \langle e_s \rangle$.
First we provide a lower bound for the Davenport constants.

\begin{lemma} \label{dandDabeliancase}
Let $G$ be  abelian.
\begin{enumerate}
\item $\mathsf D_k (G) = 1+ \mathsf d_k (G)$ for every $k \in \N$.

\smallskip
\item $\mathsf d^* (G) + (k-1)\exp (G) \le \mathsf d_k (G)$.
\end{enumerate}
\end{lemma}

\begin{proof}
1.  Let $k \in \N$. By Proposition \ref{2.8}.1, we have $1 + \mathsf d_k (G) \le \mathsf D_k (G)$. Obviously, the map
\[
\psi \colon \mathcal M_k^*(G) \to \mathcal M_k (G) \setminus
\{1\}\,, \quad \text{given by} \quad \psi (S) = (-\sigma (S)) \bdot S\,,
\]
is surjective and we have $|\psi (S)|=1+|S|$ for every $S \in \mathcal M_k^* (G)$. Therefore we have $1 + \mathsf d_k (G) = \mathsf D_k (G)$.

\smallskip
2. Suppose that $G \cong C_{n_1} \oplus \ldots \oplus C_{n_r}$, with $r \in \mathbb N_0$ and $1 < n_1 \t \ldots \t n_r$. If $(e_1, \ldots, e_r)$ is a basis of $G$ with $\ord (e_i) = n_i$ for all $i \in [1, r]$, then
\[
S = e_r^{[n_r(k-1)]} \prod_{i=1}^r e_i^{[n_i-1]}
\]
is not divisible by  a product of $k$ nontrivial zero-sum sequences whence $\mathsf d^* (G) + (k-1)\exp (G) = |S| \le  \mathsf d_k (G)$.
\end{proof}

We continue with a result on the $k$th Davenport constant which refines the more general results in Subsection \ref{2.E}. It provides an explicit formula for $\mathsf d_k (G)$ in terms of $\mathsf d (G)$ (see \cite[Theorem 6.1.5]{Ge-HK06a}).

\begin{theorem} \label{gen-dav-abelian}
Let $G$ be  abelian,   $\exp (G) = n$, and  $k \in \N$.
\begin{enumerate}
\item Let $H \le G$ be a subgroup such that $G = H \oplus
      C_n$. Then
      \[
      \mathsf d (H) +kn-1 \le \mathsf d_k (G) \le (k-1)n + \max \{
      \mathsf d(G), \eta(G) -n-1\} \,.
      \]
      In particular, if \ $\mathsf d (G) = \mathsf d (H) + n-1$ \ and \ $\eta (G) \le \mathsf d (G) + n+1$, then
      \ $\mathsf d_k (G) = \mathsf d (G) + (k-1)n$.

\smallskip
\item If \ $\mathsf r (G) \le 2$, then \ $\mathsf d_k (G) = \mathsf
      d (G) + (k-1)n$.

\smallskip
\item If \ $G$ is a $p$-group and \ $\mathsf D (G) \le 2 n-1$, then
      \ $\mathsf d_k (G) = \mathsf d (G) + (k-1)n$.
\end{enumerate}
\end{theorem}

For the rest of this section we focus on the classical Davenport constant $\mathsf D (G)$. By Lemma \ref{dandDabeliancase}.2, there is  the crucial inequality
\[
\mathsf d^* (G) \le \mathsf d (G) \,.
\]
We continue with a  list of groups for which equality holds. The list  is incomplete but the remaining groups for which $\mathsf d^* (G) = \mathsf d (G)$ is known are of a similar special nature as those listed in Theorem \ref{d*equalsd}.3 (see \cite{Sc11b} for a more detailed discussion).  In particular,  it is still open whether equality
holds for all groups of rank three (see \cite[Section 4.1]{Sc11b}) or for all groups of the form $G = C_n^r$ (see \cite{Gi08b}).

\begin{theorem} \label{d*equalsd}
We have $\mathsf d^* (G) = \mathsf d (G)$ in each of
the following cases{\rm \;:}
\begin{enumerate}
\item $G$ is a $p$-group or has rank $\mathsf r (G) \le 2$.

\item  $G = K \oplus C_{km}$ \ where $k, m \in \N$, $p \in \mathbb P$ a prime, $m$ a power of $p$ and $K \le G$ is a $p$-subgroup with $\mathsf d (K) \le m-1$.

\item $G = C_m^2 \oplus C_{mn}$ where $m \in \{2,3,4,6\}$ and $n \in \N$.
\end{enumerate}
\end{theorem}

\begin{proof}
For 1. see \cite{Ge-HK06a} (in particular, Theorems 5.5.9 and 5.8.3) for proofs and historical
comments. For 2. see \cite[Corollary 4.2.13]{Ge09a}, and 3. can be found in
\cite{Bh-SP07a} and \cite[Theorem 4.1]{Sc11b}.
\end{proof}

There are   infinite series of groups  $G$ with $\mathsf d^* (G) < \mathsf d (G)$.  However, the
true reason for the phenomenon $\mathsf d^* (G) < \mathsf d (G)$ is
not understood. Here is a simple observation. Suppose that $G = C_{n_1} \oplus \ldots \oplus C_{n_r}$ with $1 < n_1 \t \ldots \t n_r$, $I \subset [1, r]$,  and
let $G' = \oplus_{i \in I}C_{n_i}$. If $\mathsf d^* (G') < \mathsf d (G')$, then $\mathsf d^* (G) < \mathsf d (G)$.
For  series of groups  $G$ which have rank four and five  and satisfy $\mathsf d^* (G) < \mathsf d (G)$ we refer to
\cite{Ge-Sc92, Ge-Li-Ph12}.
A standing conjecture for an upper bound on $\mathsf D (G)$ states that $\mathsf d (G) \le \mathsf d^* (G) + \mathsf r (G)$. However, the available results are much weaker (\cite[Theorem 5.5.5]{Ge-HK06a}, \cite{Bh-SP13a}).

The remainder of this subsection is devoted to inverse problems with respect to the Davenport constant.
Thus the objective is to study the  structure of  zero-sum free sequences $S$ whose lengths $|S|$ are close to the maximal possible value $\mathsf d (G)$.

If $G$
is cyclic of order $n \ge 2$, then an easy exercise shows that $S$
is zero-sum free of length $|S| = \mathsf d (G)$ if and only if $S =
g^{[n-1]}$ for some $g \in G$ with $\ord (g) = n$. After many
contributions since the 1980s, S. Savchev and F. Chen
could finally prove a (sharp) structural result. In order to formulate it we need some more terminology.
If $g \in G$ is a nonzero element of order $\ord (g) = n$ and
\[
S = (n_1g) \bdot \ldots \bdot (n_{\ell}g), \quad \text{where} \quad \ell \in \mathbb N_0 \quad \text{and} \quad n_1, \ldots, n_{\ell} \in [1,n] \,,
\]
we define
\[
\| S \|_g = \frac{n_1+ \ldots + n_{\ell}}n  \,.
\]
Obviously, $S$ has sum zero if and only if $\|S\|_g \in \mathbb N_0$, and the {\it index of} $S$ is defined as
      \[
      \ind (S) = \min \{ \| S \|_g \colon g \in G \ \text{with} \ G  = \langle g \rangle \} \in \mathbb Q_{\ge 0} \,.
      \]

\begin{theorem} \label{inverse1}
Let $G$ be  cyclic  of order $|G|=n \ge 3$ .
\begin{enumerate}
\item If $S$ is a zero-sum free
      sequence over $G$ of length $|S| \ge (n+1)/2$, then there exist $g
      \in G$ with $\ord (g) = n$ and integers $1 = m_1, \ldots, m_{|S|}
      \in [1, n-1]$ such that
      \begin{itemize}
      \item $S = (m_1g) \bdot \ldots \bdot (m_{|S|}g)$

      \item $m_1 + \ldots + m_{|S|} < n$ and $\Sigma (S) = \{ \nu g \colon
            \nu \in [1, m_1 + \ldots + m_{|S|}] \}$.
      \end{itemize}

\smallskip
\item If $U \in \mathcal A (G)$ has length $|U| \ge \lfloor \frac{n}{2} \rfloor + 2$, then $\ind (U) = 1$.
\end{enumerate}
\end{theorem}

\begin{proof}
1. See \cite{Sa-Ch07a} for the original paper. For the
history of the problem and a proof in the present terminology see
\cite[Chapter 5.1]{Ge09a} or \cite[Chapter 11]{Gr13a}.

\smallskip
2. This is a simple consequence of the first part (see \cite[Theorem 5.1.8]{Ge09a}).
\end{proof}

The above result was generalized to groups of the form $G = C_2 \oplus C_{2n}$ by S. Savchev and F. Chen (\cite{Sa-Ch12a}).
Not much is known about the number of all minimal zero-sum  sequences of a given group. However, the above result allows to give a  formula for the number of minimal zero-sum sequences of  length $\ell \ge \lfloor \frac{n}{2} \rfloor + 2$ (this formula was first proved by  Ponomarenko \cite{Po04a} for $\ell > 2n/3$).

\begin{corollary} \label{number-of-minimal}
Let $G$ be  cyclic  of order $|G|=n \ge 3$, and let $\ell \in \Big[ \lfloor \frac{n}{2} \rfloor + 2, n \Big]$. Then the number of minimal zero-sum sequences $U \in \mathcal A (G)$ of length $\ell$ equals $\Phi (n) \mathsf p_{\ell} (n)$, where $\Phi (n) = |(\Z/n\Z)^{\times}|$ is Euler's Phi function and $\mathsf p_{\ell} (n)$ is the number of integer partitions of $n$ into $\ell$ summands.
\end{corollary}

\begin{proof}
Clearly, every generating element $g \in G$ and every integer partition $n = m_1 + \ldots + m_{\ell}$ gives rise to a minimal zero-sum sequence
$U = (m_1g) \bdot \ldots \bdot (m_{\ell}g)$. Conversely,
if $U \in \mathcal A (G)$ of length $|U| = \ell$, then Theorem \ref{inverse1}.2 implies that there is an element $g \in G$ with $\ord (g) = n$ such that
\[
U = (m_1g) \bdot \ldots \bdot (m_{\ell}g) \quad \text{where} \quad m_1, \ldots, m_{\ell} \in [1, n-1] \ \text{with} \ n = m_1 + \ldots + m_{\ell} \,. \tag{$*$}
\]
Since $G$ has precisely $\Phi (n)$ generating elements, it remains to show that for every $U \in \mathcal A (G)$ of length $|U| = \ell$ there is precisely one generating element $g \in G$ with $\|U\|_g = 1$. Let $U$ be as in $(*)$, and assume to the contrary that there are $a \in [2, n-1]$ with $\gcd (a, n)=1$ and $m_1', \ldots, m_{\ell}' \in [1,n]$ such that $m_1'+ \ldots + m_{\ell}'=n$ and
\[
U = \big(m_1' (a g)\big) \cdot \ldots \cdot \big(m_{\ell}' (ag) \big) \,.
\]
Let $a'\in [2,n-1]$ such that $aa'\equiv 1\pmod n$. Since
\[
\begin{aligned}
n&=m_1+\ldots+m_{\ell}\ge \mathsf v_g(U)+a\mathsf v_{ag}(U)+2(\ell-\mathsf v_g(U)-\mathsf v_{ag}(U))\\ & =2\ell-\mathsf v_g(U)+(a-2)\mathsf  v_{ag}(U)\quad \text{and} \\
n&=m_1'+\ldots+m_{\ell}'\ge a'\mathsf v_g(U)+\mathsf v_{ag}(U)+2(\ell-\mathsf v_g(U)-\mathsf v_{ag}(U)) \\ & =2\ell+(a'-2)\mathsf v_g(U)-\mathsf  v_{ag}(U)\,,\\
\end{aligned}
\]
it follows that
\[
\begin{aligned}
(a-1)n & =n+(a-2)n \\ & \ge  2\ell-\mathsf v_g(U)+(a-2)\mathsf  v_{ag}(U)+(a-2)(2\ell+(a'-2)\mathsf v_g(U)-\mathsf  v_{ag}(U))\\
& = (a-1)2\ell+((a-2)(a'-2)-1)\mathsf v_g(U)\,,
\end{aligned}
\]
whence $a=2, a'=\frac{n+1}{2}$ or $a'=2, a=\frac{n+1}{2}$ because $\ell\ge \lfloor\frac{n}{2}\rfloor+2$.
By symmetry, we may assume that $a=2$. Then $\mathsf v_g(U)\ge 2\ell-n\ge 2\lfloor\frac{n}{2}\rfloor+4-n\ge 3$, and thus $n\ge a'\mathsf v_g(U)\ge 3\frac{n+1}{2}$, a contradiction.
\end{proof}

The structure of all minimal zero-sum sequences of maximal length $\mathsf D (G)$ has been completely determined for rank two groups
(\cite{Ga-Ge03b, Ga-Ge-Gr10a, Sc10b, Re10c}), for groups of the form $G = C_2 \oplus C_2 \oplus C_{2n}$ with $n \ge 2$ (\cite[Theorem 3.13]{Sc11b}),
and for groups of the form $G = C_2^4 \oplus C_{2n}$ with $n \ge 70$ (\cite[Theorems 5.8 and 5.9]{Sa-Ch14a}).

\section{\bf Multiplicative Ideal Theory of Invariant Rings} \label{sec:4}

After gathering basic material from invariant theory in Subsection \ref{4.0}  we   construct an explicit divisor theory for the algebra of polynomial invariants of a finite group (see Subsection \ref{4.A}). In Subsection \ref{subsec:abelian} we present a detailed study of the abelian case as outlined  in the Introduction. In Subsection \ref{5.D} we associate a BF-monoid to a $G$-module whose $k$th Davenport constant is a lower bound for the $k$th Noether number.

\subsection{\bf Basics of invariant theory} \label{4.0}~

Let $n = \dim_{\F} (V)$ and let $\rho \colon G \to \GL(n, \F)$ be a group homomorphism. Consider the action of $G$ on the polynomial ring $\F[x_1,\ldots,x_n]$ via $\F$-algebra automorphisms induced by
$g\cdot x_j=\sum_{i=1}^n\rho(g^{-1})_{ji}x_i$.
Taking a slightly more abstract point of departure, we suppose that $V$ is a  $G$-module (i.e. we suppose that $V$ is  endowed with an action of $G$  via linear transformations). Choosing a basis of $V$, $V$
is identified with  $\F^n$, the group $\GL(n,\F)$ is identified with the group $\GL(V)$ of invertible linear transformations of $V$, and
$\F[V]=\F[x_1, \ldots, x_n]$ can be thought of as the symmetric algebra of $V^*$, the dual $G$-module of $V$, in which $(x_1,\ldots,x_n)$ is a basis dual to the standard basis in $V$.
The action on $V^*$ is given by $(g\cdot x)(v)=x(\rho(g^{-1})v)$, where $g\in G$, $x\in V^*$, $v\in V$.
Note that, if  $\F$ is infinite, then $\F[V]$ is the algebra of polynomial functions $V\to \F$, and the
action of $G$ on $\F[V]$ is the usual action on functions $V\to F$ induced by the action of $G$ on $V$ via $\rho$.
Denote by $\F(V)$ the quotient field of $\F[V]$, and extend the $G$-action on $\F[V]$ to $\F (V)$ by
\[
g\cdot  \frac{f_1}{f_2} = \frac{g\cdot  f_1}{g \cdot f_2} \quad \text{for} \ f_1, f_2 \in \F[V] \quad \text{and} \ g \in G \,.
\]
We define
\[
\F(V)^G = \{ f\in \F(V) \colon g\cdot f=f \ \text{for all} \ g\in G \} \subset \F(V) \quad \text{and} \quad \F[V]^G = \F(V)^G \cap \F[V] \,.
\]
Then $\F (V)^G \subset \F (V)$ is a subfield and $\F[V]^G \subset \F[V]$
is an $\F$-subalgebra of $\F[V]$, called the {\it ring of polynomial invariants} of $G$
(the group homomorphism $\rho:G\to\GL(V)$ giving the $G$-action on $V$ is usually suppressed from the notation).
Since every element of $\F (V)$ can be written in the form $f_1 f_2^{-1}$ with $f_1 \in \F[V]$ and $f_2 \in \F[V]^G$, it follows that $\F (V)^G$ is the quotient field of $\F[V]^G$.
Next we summarize some well-known ring theoretical properties of $\F[V]^G$ going back  to E.  Noether \cite{No26}.

\begin{theorem} \label{4.1}
Let all notations be as above.
\begin{enumerate}
\item $\F[V]^G \subset \F[V]$ is an integral ring extension and $\F[V]^G$ is normal.

\smallskip
\item $\F[V]$ is a finitely generated $\F[V]^G$-module, and $\F[V]^G$ is a finitely generated $\F$-algebra $($hence in particular a noetherian domain$)$.

\smallskip
\item $\F[V]^G$ is a Krull domain with Krull dimension $\dim_{\F} (V)$.
\end{enumerate}
\end{theorem}

\begin{proof}
1. To show that $\F[V]^G$ is normal, consider an element $f \in \F(V)^G$ which is integral over $\F[V]^G$. Then $f$ is integral over $\F[V]$ as well,
and since $\F[V]$ is normal, it follows that $f \in \F[V] \cap \F(V)^G = \F[V]^G$.

To show that $\F[V]^G \subset \F[V]$ is an integral ring extension, consider an element $f \in \F[V]$ and the polynomial
\begin{equation}\label{eq:integral}
\Phi_f = \prod_{g \in G} (X - gf) \in \F[V][X] \,.
\end{equation}
The coefficients of $\Phi_f$ are the elementary symmetric functions (up to sign) evaluated at $(gf)_{g \in G}$, and hence they are in $\F[V]^G$. Thus $f$ is a root of a monic polynomial with coefficients in  $\F[V]^G$.

\smallskip
2.  For $i \in [1,n]$, we consider the polynomials
$\Phi_{x_i} (X)$ (cf. \eqref{eq:integral}),
and denote by $A \subset \F[V]^G \subset \F[V]$ the $\F$-algebra generated by the coefficients of $\Phi_{x_1}, \ldots, \Phi_{x_n}$. By definition,
$A$ is a finitely generated $\F$-algebra and hence a noetherian domain. Since $x_1, \ldots, x_n$ are integral over $A$, $\F[V] = A[x_1, \ldots, x_n]$
is a finitely generated (and hence noetherian) $A$-module. Therefore the  $A$-submodule $\F[V]^G$ is a finitely generated $A$-module, and hence a finitely generated $\F$-algebra.

\smallskip
3. By 1. and 2.,  $\F[V]^G$ is an normal noetherian domain, and hence a Krull domain by Theorem \ref{2.1}. Since $\F[V]^G \subset \F[V]$ is an integral ring extension, the Theorem of Cohen-Seidenberg implies that their Krull dimensions coincide, and  hence
$\dim (\F[V]^G) = \dim (\F[V]) = \dim_{\F} (V)$.
\end{proof}

The algebra $\F[V]$ is graded in the standard way (namely, $\deg (x_1)= \ldots = \deg (x_n)=1$), and the subalgebra $\F[V]^G$ is generated by homogeneous elements.
For $\F$-subspaces $S,T\subset \F[V]$ we  write $ST$ for the $\F$-subspace in $\F[V]$ spanned by all the products $st$ $(s\in S,t\in T)$,
and write $S^k=S\dots S$ (with $k$ factors). The factor algebra of $\F[V]$ by the ideal generated by $\F[V]^G_+$ is usually called the {\it algebra of coinvariants}.
It inherits the grading of $\F[V]$ and is finite dimensional.

\begin{definition}\label{def:beta_k}
Let $k \in \N$.
\begin{enumerate}
\item Let $\beta_k(G,V)$ be the top degree of the factor space $\F[V]_+^G/(\F[V]_+^G)^{k+1}$, where $\F[V]_+^G$ is the maximal ideal of $\F[V]^G$ spanned by the positive degree homogeneous elements. We call
    \[
    \beta_k(G) = \sup\{ \beta_k(G,W): W \text{ is a $G$-module over $\F$} \}
    \]
    the $k$th {\it Noether number} of $G$.
\item  Let $b_k(G,V)$ denote the top degree of the factor algebra $\F[V]/(\F[V]^G_+)^k\F[V]$ and set
       \[
       b_k(G)= \sup\{ b_k(G,W): W \text{ is a $G$-module over $\F$} \} \,.
       \]
\end{enumerate}
\end{definition}

In the special case $k=1$ we set
\[
\beta(G,V)=\beta_1(G,V) \ ,  \beta(G)=\beta_1(G) \ ,  \ b(G,V)=b_1(G,V) \ ,  \text{and} \ b(G)=b_1(G) \,,
\]
and $\beta (G)$ is the {\it Noether number} of $G$.
If $\{f_1, \ldots, f_m\}$ and $\{h_1, \ldots, h_l\}$ are two minimal homogeneous generating sets of $\F[V]^G$,
then $m=l$ and, after renumbering if necessary, $\deg (f_i) = \deg (h_i)$ for all $i \in [1, m]$ (\cite[Proposition 6.19]{Ne07a}).
Therefore by the Graded Nakayama Lemma (\cite[Proposition 8.31]{Ne07a}) we have
\[
\beta (G,V) = \max \{ \deg (f_i) \colon i \in [1, m] \} \,,
\]
where $\{f_1,\ldots,f_m\}$ is a minimal homogeneous generating set of $\F[V]^G$.
 Again by the Graded Nakayama Lemma,
 $b(G,V)$ is the maximal degree of a generator in a minimal system of homogeneous generators of $\F[V]$ as an $\F[V]^G$-module.
 If $\mathrm{char}(\F) \nmid |G|$, then by \cite[Corollary 3.2]{Cz-Do13c} we have
 \begin{equation}\label{eq:beta=b+1} \beta_k(G)=b_k(G)+1 \quad \mbox{ and } \quad \beta(G,V)\le b(G,V)+1 \,,
 \end{equation}
where the second inequality can be strict.
If $G$ is abelian, then  $\beta_k(G,V)$ and $b_k(G,V)$ will be interpreted as  $k$th Davenport constants (see Proposition~\ref{prop:bschmid}).

The \emph{regular $G$-module $V_{\mathrm{reg}}$} has a basis $\{e_g\colon g\in G\}$ labelled by the group elements, and the group action is given by
$g\cdot e_h=e_{gh}$ for $g,h\in G$. More conceptually, one can identify $V_{\mathrm{reg}}$ with the space of $\F$-valued functions on $G$, on which $G$ acts linearly via the action induced by the left multiplication action of $G$ on itself. In this interpretation the basis element $e_g$ is the characteristic function of the set $\{g\}\subset G$.
It was proved in \cite{Sc91a} that, if  $\mathrm{char}(\F)=0$, then $\beta(G)=\beta(G,V_{\mathrm{reg}})$. If $\F$ is  algebraically closed, each irreducible $G$-module occurs in $V_{\mathrm{reg}}$ as a direct summand with multiplicity equal to its dimension.

\begin{theorem}\label{thm:noether}~

\begin{enumerate}
\item If  $\mathrm{char}(\F) \nmid |G|$, then  $\beta(G)  \le |G| $.

\item If  $\mathrm{char} (\F) \mid |G|$, then $\beta (G)=\infty$.
\end{enumerate}
\end{theorem}

\begin{proof}
1. The case   $\mathrm{char} (\F) =0 $ was proved by E. Noether \cite{No16} in 1916, and her argument works as well when the characteristic of $\F$ is greater than $|G|$. The general case was shown independently by P. Fleischmann \cite{Fl00a} and J. Fogarty \cite{fogarty}  (see also
\cite[Theorem 2.3.3]{Ne-Sm02a} and \cite{knop}.
For 2.  see \cite{richman}.
\end{proof}

Bounding the Noether number has always been an objective of invariant theory
(for recent surveys  we refer to  \cite{wehlau, Ne07b};  degree bounds are discussed in  \cite{Do-He00a, Se02a,  F-S-S-W06,Cz14d, heg-pyb};  see \cite{derksen-kemper} for algorithmic aspects).
Moreover, the main motivation to introduce the $k$th Noether numbers $\beta_k(G)$ (\cite{Cz-Do13c, Cz-Do15a, Cz-Do14a}) was to bound the ordinary Noether number $\beta(G)$ via structural reduction (see Subsection~\ref{sec:5.1}).

\subsection{\bf  The divisor theory of invariant rings} \label{4.A}~

Let $G\subset \GL(V)$ and $\chi\in\Hom(G,\F^\bullet)$. Then
\[
\F[V]^{G,\chi} = \{f\in\F[V]\colon  g\cdot f=\chi(g)f \ \text{for all} \ g \in G \}
\]
denotes the space of {\it relative invariants of weight} $\chi$, and we set
\[
\F[V]^{G,\mathrm{rel}} = \bigcup_{\chi\in\Hom(G,\F^\bullet)}\F[V]^{G,\chi}.
\]
Clearly, we have $\F[V]^G \ \subset \ \F[V]^{G,\mathrm{rel}} \ \subset \ \F[V]$, and to simplify notation, we set
\[
H =(\F[V]^G \setminus \{0\})_{\red} \,, \quad  D=(\F[V]^{G,\mathrm{rel}}\setminus\{0\})_{\red} \,, \quad \text{and} \quad E= (\F[V] \setminus \{0\})_{\red} \,.
\]
Since $\F[V]$ is a factorial domain with $\F^\bullet$  as its  set of units, $E=\mathcal{F}(P)$ is the  free abelian  monoid  generated by $P=\{\F^\bullet f\colon f\in \F[V] \mbox{ is irreducible}\}$.
The action of $G$ on $\F[V]$ is via $\F$-algebra automorphisms, so it induces a permutation action of $G$ on
$E$ and $P$. Denote by $P/G$ the set of $G$-orbits in $P$. We shall identify $P/G$ with a subset of $E$ as follows: assign to the $G$-orbit $\{f_1,\ldots,f_r\}$
the element $f_1\ldots f_r\in E$ (here $f_1,\dots,f_r\in P$ are distinct).

We say that a non-identity element $g \in G\subset \GL(V)$ is a {\it pseudoreflection}  if a hyperplane in $V$ is fixed pointwise by $g$, and $g$ is not unipotent (this latter condition holds automatically if $\mathrm{char}(\F)$ does not divide $|G|$, since then a non-identity unipotent transformation cannot have finite order).
We denote by $\Hom^0(G,\F^\bullet) \le \Hom (G,\F^\bullet)$ the subgroup of the character group  consisting of the characters that contain all pseudoreflections in their kernels.
For each $p\in P$, choose a representative $\tilde{p}\in \F[V]$ in the associate class $p=\F^\bullet \tilde{p}$. We have
$\mathfrak X (\F[V]) = \{\tilde{p}\F[V]\colon p\in P\}$ because $\F[V]$ is factorial.
We set $\mathsf v_{\tilde p} = \mathsf v_p \colon \mathsf q (\F[V]^{\bullet})= \F(V)^{\bullet} \to \Z$, and
for a subset $X\subset \F(V)$ we write $\mathsf v_p(X) = \inf \{\mathsf v_p(f) \colon f\in X\setminus \{0\} \}$.
The {\it ramification index} of the prime ideal $\tilde{p}\F[V]$ over $\F[V]^G$ is
$e(p)=\mathsf v_p(\tilde{p}\F[V]\cap \F[V]^G)$.
The ramification index $e(p)$  can be expressed in terms of the {\it inertia subgroup}
\[
I_p = \{g\in G \colon  g\cdot f-f\in \tilde{p}\F[V] \quad \text{for all} \quad f\in \F[V] \} \,.
\]
Since $V^\star$ is a $G$-stable subspace in $\F[V]$, the inertia subgroup $I_p$ acts trivially on $V^\star/(V^\star\cap \tilde{p}\F[V])$. On the other hand $I_p$ acts faithfully on $V^\star$. So if $I_p$ is non-trivial, then $V^\star\cap \tilde{p}\F[V]\neq 0$, implying $\tilde{p}\in V^\star$. Clearly $I_p$ must act trivially on the hyperplane $\mathcal{V}(\tilde{p})=\{v\in V\colon \tilde{p}(v)=0\}$, and hence acts via multiplication by a character $\delta_p\in\Hom(I_p,\F^\bullet)$ on the $1$-dimensional factor space $V/\mathcal{V}(\tilde{p})$.
So $\ker(\delta_p)$ is a  normal subgroup
of $I_p$ (necessarily unipotent hence trivial if $\mathrm{char}(\F)\nmid |G|$) and $I_p=\ker(\delta_p)Z$ decomposes as a semi-direct product of $\ker(\delta_p)$ and a cyclic subgroup $Z$
consisting of pseudoreflections fixing pointwise $\mathcal{V}(\tilde{p})$. So $Z\cong I_p/\ker(\delta_p)$ is isomorphic to a finite subgroup of $\F^\bullet$.

The next
Lemma~\ref{lemma:nakajima}  is  extracted from Nakajima's paper \cite{Na82z}.

\begin{lemma}\label{lemma:nakajima}~

\begin{enumerate}
\item We have the equality $e(p)=|Z|$.
\item $\mathsf v_p(\F[V]^{G,\chi})<e(p)$ for all $\chi\in \Hom(G,\F^\bullet)$.
\item $\mathsf v_p(\F[V]^{G,\chi})=0$ for all $\chi\in \Hom^0(G,\F^\bullet)$.
\end{enumerate}
\end{lemma}

\begin{proof}
 1. By  \cite[9.6, Proposition (i)]{neukirch}, we have that $e(p)=\mathsf v_p(\tilde{p}\F[V]\cap \F[V]^{I_p})$, the ramification index of the prime ideal $\tilde{p}\F[V]$ over the subring of $I_p$-invariants. Thus if $I_p$ is trivial, then $e(p)=1$, and of course $|Z|=1$. If $I_p$ is non-trivial, then as it was explained above, $\tilde{p}$ is a linear form, which is a relative $I_p$-invariant with weight  $\delta_p^{-1}$, hence $\tilde{p}^{|Z|}$ is an $I_p$-invariant,
implying $\mathsf v_p(\tilde{p}\F[V]\cap \F[V]^{I_p})\le |Z|$.  On the other hand $\F[V]^{I_p}$ is contained in $\F[V]^Z$, and the algebra of invariants of the cyclic group $Z$ fixing pointwise the hyperplane $\mathcal{V}(\tilde{p})$ is generated by $\tilde{p}^{|Z|}$
and a subspace of $V^*$ complementary to $\F \tilde{p}$. Thus $\mathsf v_p(\tilde{p}\F[V]\cap \F[V]^{I_p})\ge \mathsf v_p(\tilde{p}\F[V]\cap \F[V]^{Z})=|Z|$, implying in turn that
$e(p)=|Z|$.

2. Take an $h\in \F[V]^G$ with $e(p)=\mathsf v_p(h)$.
Note that $\mathsf v_q(h)=\mathsf v_p(h)$ and $\mathsf v_q(\F[V]^{G,\chi})=\mathsf v_p(\F[V]^{G,\chi})$ holds for all
$q\in G\cdot p$, since $\F[V]^{G,\chi}$ is a $G$-stable subset in $\F[V]$.
Set $S=\{\frac ft\colon f\in \F[V],\quad t\in \F[V]^G\setminus \tilde{p}\F[V]\}$. This is a $G$-stable subring in $\mathsf q (\F[V])$ containing $\F[V]$.
Consider $S^{\chi}=S\cap \mathsf q(\F[V])^{\chi}$, where $\mathsf q(\F[V])^{\chi}= \{ s\in \mathsf q(\F[V]) \colon  g\cdot s=\chi(g)s  \quad \text{for all} \  g\in G\}$.
Then $\mathsf v_q(S^{\chi})=\mathsf v_q(\F[V]^{G,\chi})$ for all $q\in G\cdot p$, since for any  denominator $t$ of an element
$\frac ft$ of $S$ we have $\mathsf v_q(t)=0$. Now suppose that contrary to our statement we have $e(p)\le \mathsf v_p(\F[V]^{G,\chi})$, and hence
$\mathsf v_q(h)\le \mathsf v_q(S^{\chi})$ for all $q\in G\cdot p$. In particular this means that $\F[V]^{G,\chi}\neq \{0\}$. Then $\mathsf v_q(h^{-1}S^{\chi})\ge 0$ holds for all $q\in G\cdot p$.
Now $S$ is a Krull domain with
$\mathfrak X (S) = \{\tilde{q} S\colon q\in G\cdot p\}$,
thus $h^{-1}S^{\chi}\subset S$ (see the discussion after Theorem \ref{2.1}),  implying
that $S^{\chi}\subset hS$.
Clearly $hS\cap S^{\chi}=hS^{\chi}$, so we conclude in turn that $S^{\chi}\subset hS^{\chi}$. Iterating this we deduce $\{0\}\neq S^{\chi}\subset \cap_{n=1}^{\infty} h^nS$, a contradiction.

3. It is well known that $\F[V]^{G,\chi} \ne \{0\}$ (see the proof of {\bf A4.} below). Write $v=\mathsf v_p(\F[V]^{G,\chi})$. Take $f\in \F[V]^{G,\chi}$ with $\mathsf v_p(f)=v$, say $f=\tilde{p}^vh$, where $h \in \F[V] $.
Note that both $f$ and $\tilde{p}$ are  relative invariants of $I_p$, hence so is $h$. Therefore $g\cdot h\in \F^\bullet h$, and $\tilde{p}\mid_{\F[V]} (g\cdot h-h)$ for all $g\in I_p$, implying that $h$ is an $I_p$-invariant. Any $\chi\in \Hom^0(G,\F^\bullet)$ contains  $I_p$ in its kernel (the unipotent normal subgroup $\ker(\delta_p)$ of $I_p$ has no non-trivial characters at all, and $Z=I_p/\ker(\delta_p)$ consists of pseudoreflections). Thus $f$ is $I_p$-invariant as well. Therefore  $\tilde{p}^v$ is $I_p$-invariant, so its weight
$\delta_p^v$ is trivial. Consequently the order $|Z|$ of $\delta_p$ in $\Hom(I_p,\F^\bullet)$ divides $v$. We have $e(p)=|Z|$ by 1., and on the other hand  $v<e(p)$ by 2.,  forcing $v=0$.
\end{proof}

For a relative invariant $f$, we denote by $w(f)$ the weight of $f$. This induces a homomorphism
$w \colon D\to \Hom(G,\F^\bullet)$ assigning to $\F^\bullet f\in D$ the weight $w(f)$ of the relative invariant $f$. Clearly, $w$ extends  to a group homomorphism $w \colon \mathsf{q}(D)\to \Hom(G,\F^\bullet)$.
 The kernel of $w$ consists of elements of the form $(\F^\bullet h)^{-1}\F^\bullet f$, where $f,h\in \F[V]^{G,\chi}$ for some character $\chi$. Now $f/h$ belongs to $\F(V)^G$, which is the field of fractions of $\F[V]^G$, so there exist $f_1,h_1\in \F[V]^G$ with $f/h=f_1/h_1$, implying
$(\F^\bullet h)^{-1}\F^\bullet f=(\F^\bullet h_1)^{-1}\F^\bullet f_1\in \mathsf{q}(H)$. Thus $\ker(w)=\mathsf{q}(H)$.
 Therefore $w$ induces a  monomorphism $\overline w \colon \mathsf q (D)/\mathsf q (H) \to \Hom(G,\F^\bullet)$.

\begin{theorem} \label{main-theorem}~
Let  $G \subset \GL(V)$, $H = (\F[V]^G \setminus \{0\})_{\red}$, and $D=(\F[V]^{G,\mathrm{rel}}\setminus\{0\})_{\red}$.
\begin{enumerate}
\item The embeddings \ ${\F[V]^G} \setminus \{0\} \ \overset{\varphi}{\hookrightarrow} \ \F[V]^{G,\mathrm{rel}} \setminus \{0\} \ \overset{\psi}{\hookrightarrow} \ {\F[V]}^{\bullet}$ \ are cofinal divisor homomorphisms.

\smallskip
\item $D$ is factorial, $P/G \subset E$ is the set of  prime elements in $D$, and $\mathcal C ( \varphi)$ is a torsion group.

\smallskip
\item The monoid $D_0 = \{ \gcd_D (X) \colon X \subset H \ \text{finite} \} \subset D$ is free abelian with basis $\{ q^{e(q)} \colon q\in P/G \}$, where  $e(q)=\min\{\mathsf{v}_q(h) \colon q\mid_D h\in H\}$, and the embedding $H \hookrightarrow D_0$ is a divisor theory.

\smallskip
\item We have  $D_0 = \{f\in D\colon w(f)\in  \Hom^0(G,\F^\bullet)\}$  and
$\overline w \mid_{\mathsf q (D_0)/\mathsf q (H)} \colon \mathcal C ( \F[V]^G) = \mathsf q (D_0)/\mathsf q (H) \to \Hom^0(G,\F^\bullet)$ is an isomorphism.
\end{enumerate}
\end{theorem}

Theorem \ref{main-theorem} immediately implies the following corollary which can be found in Benson's book (\cite[Theorem 3.9.2]{Be93a}) and which goes back to Nakajima \cite{Na82z} (see also \cite{Fl-Wo11a} for a discussion of this theorem).

\begin{corollary}[Benson-Nakajima] \label{benson-nakajima}
The class group of $\F[V]^G$ is isomorphic to $\Hom^0(G,\F^\bullet)$, the subgroup of the character group  consisting of the characters that contain all pseudoreflections in their kernels.
\end{corollary}

\begin{proof}[of Theorem \ref{main-theorem}]
1. Since $\F[V]^G = \F(V)^G \cap \F[V]$, the embedding $\psi \circ \varphi \colon \F[V]^G \hookrightarrow \F[V]$ is a divisor homomorphism, and hence $\varphi$ is a divisor homomorphism. Furthermore, if the quotient of two relative invariants lies in $\F[V]$, then it is a relative invariant whence $\psi$ is a divisor homomorphism. In order to show that
the embeddings are cofinal, let $0 \ne f \in \F[V]$ be given. Then $f^* = \prod_{g \in G} g f \in \F[V]^G$ and $f \t f^*$, so the embedding $\psi \circ \varphi$ is
cofinal and hence $\varphi$ and $\psi$ are cofinal.

\smallskip
2. Suppose that $\{f_1,\ldots,f_r \} \subset \F[V]$ represents a $G$-orbit in $P$. Then $g\cdot (f_1\ldots f_r)$ is a non-zero scalar multiple of $f_1\ldots f_r$, hence
$f_1\ldots f_r\in \F[V]^{G,\mathrm{rel}}$. This shows that $P/G\subset E$ is in fact contained in $D$. Conversely, take an irreducible element $\F^\bullet f$ in the monoid $D$ (so $f$ is a relative invariant). Take any irreducible divisor $f_1$ of $f$ in $\F[V]$. Since $g\cdot f\in \F^\bullet f$, the polynomial $g\cdot f_1$ is also the divisor of $f$. Denoting by $f_1,\ldots, f_r$ polynomials representing the $G$-orbit of $\F^\bullet f_1$ in $P$, we conclude that $f_1\ldots f_r$ divides $f$ in $\F[V]$, hence $\F^\bullet f_1\ldots f_r$ divides $\F^\bullet f$ in $D$ as well, so
 $\F^\bullet f_1\ldots f_r=\F^\bullet f$.
 This implies that $D$ is the submonoid of $E=\mathcal{F}(P)$ generated by $P/G$.

 In order to show that $\mathcal C (\varphi)$ is a torsion group, let $f \in D$ be given. We have to find an $m \in \N$ such that $f^m \in H$. Clearly, this holds with $m$ being the order in $\Hom(G,F^\bullet)$  of the weight of the relative invariant corresponding to $f$.

\smallskip
3.  Since $\mathcal C (\varphi)$ is a torsion group,  Proposition~\ref{prop:torsion-divtheory} implies that the embedding $H \hookrightarrow D_0$ is a divisor theory, and that
$D_0$ is free abelian with basis $\{ q^{e(q)} \colon q\in P/G \}$, where  $e(q)=\min\{\mathsf{v}_q(h) \colon q\mid_D h\in H\}$
(note that if $q\in P/G$ is the $G$-orbit of $p\in P$, then $\mathsf{v}_q(h)=\mathsf{v}_p(h)$, where the latter is the exponent of $p$ in $h\in E=\mathcal{F}(P)$).

4. It remains to prove the following three assertions.

\begin{enumerate}

\smallskip
\item[{\bf A1.}\,] $D_0=\{f\in D\colon w(f)\in  \Hom^0(G,\F^\bullet)\}$.

\smallskip
\item[{\bf A2.}\,] $w (D_0) =  \Hom^0(G,\F^\bullet)$.

\smallskip
\item[{\bf A3.}\,] $\overline w \mid_{\mathsf q (D_0)/\mathsf q (H)} \colon \mathsf q (D_0)/\mathsf q (H) \to w (D_0)$ is an isomorphism.

\end{enumerate}

{\it Proof of \,{\bf A1}}.\, Set $D^0=\{f\in D\colon w(f)\in  \Hom^0(G,\F^\bullet)\}$. We show first $D_0\subset D^0$.  Let $\chi$ be a character of $G$, and assume  that $\chi(g)\neq 1$ for some pseudoreflection $g\in G$. Let $f$ be a relative invariant with $w(f)=\chi$.
Then for any $v$ with $gv=v$ we have
$f(v)=f(g^{-1}v)=(gf)(v)=\chi(g)f(v)$, hence $f(v)=0$. So  $l\mid_{\F[V]} f$, where $l$ is a non-zero linear form on $V$ that vanishes on the reflecting hyperplane of $g$. Denoting by $l=l_1,\ldots,l_r$ representatives of the $G$-orbit of $\F^\bullet l$, we find that the relative invariant $q= l_1\ldots l_r$ divides $f$.
Thus $\gcd_D \{ f\in D\mid w(f)=\chi \} \neq 1$.
Now suppose that for some $\F^\bullet k\in D_0$ we have that $w(k)$ does not belong to  $\Hom^0(G,\F^\bullet)$.
By definition of $D_0$ there exist $h_1,\ldots,h_n\in D$ with $\gcd_D (h_1,\ldots,h_n)=1$ and $kh_1,\ldots,kh_n\in H$. Clearly
$w(h_i)=w(k)^{-1} \notin  \Hom^0(G,\F^\bullet)$, hence by the above considerations  $\gcd_D (h_1,\ldots,h_n) \neq 1$, a contradiction.

Next we show $D^0\subset D_0$.  Let $d$ be an  element in the monoid $D^0$.
By Lemma~\ref{lemma:nakajima}.3  for any prime divisor $p\in P$ of $d$ there exists an $h_p\in D$ such that $w(h_p)=w(d)^{-1}$ and
$p \nmid_E h_p$.
Denote by $m>1$ the order of $w(d)$ in the group of characters. Clearly $d^m\in H$ and $dh_p\in H$.  Moreover, $\gcd_E (d^m,dh_p : p\in P,\quad p\mid_E d)=d$.

{\it Proof of \,{\bf A2}}.\, The statement follows from {\bf A1},  as soon as we show that $\F[V]^{G,\chi} \neq {0}$ for all
$\chi \in \Hom(G, \F^{\bullet})$.
For any character $\chi\in\Hom(G,\F^\bullet)$ the group $\bar G=G/\ker(\chi)$ is isomorphic to a cyclic subgroup of $\F^\bullet$, hence its order is not divisible by $\mathrm{char}(\F)$. Moreover, $\bar G$ acts faithfully on the field $T=\F(V)^{\ker(\chi)}$, with $T^{\bar G}=\F(V)^G$. By the Normal Basis Theorem,
$T$ as a $\bar G$-module over $T^{\bar G}$ is isomorphic to the regular representation of $\bar G$, hence contains the representation $\chi$ as a summand with multiplicity $1$.
This shows in particular that $T^{\bar G}$ contains a relative invariant of weight $\chi$. Multiplying this by an appropriate element of $T^{\bar G}\cap \F[V]= \F[V]^G$ we get an element of
$\F[V]^{G,\chi}$.   So all characters of $G$ occur as the weight of a relative invariant in $\F[V]$.

{\it Proof of \,{\bf A3}}.\, Since $\overline w \colon \mathsf q (D)/\mathsf q (H) \to \Hom(G,\F^\bullet)$ is a monomorphism, the map $\overline w \mid_{\mathsf q (D_0)/\mathsf q (H)} \colon \mathsf q (D_0)/\mathsf q (H) \to  w ( \mathsf q (D_0))$ is an isomorphism. Note finally that   $w ( \mathsf q (D_0)) = \mathsf q ( w (D_0)) = w (D_0)$.
\end{proof}

As already mentioned, not only the class group but also  the distribution of prime divisors in the classes is crucial for the arithmetic of the domain. Moreover, the class group together with the distribution of prime divisors in the classes are characteristic (up to units) for the domain. For a precise formulation we need one more definition.

Let $H$ be a Krull monoid, $H_{\red} \hookrightarrow \mathcal F(\mathcal P)$ a divisor theory, and let $G$ be an abelian group and  $(m_g)_{g \in G}$ be a family of cardinal numbers.  We say that $H$ has  \emph{characteristic} $(G, (m_g)_{g \in G})$ if there is a group isomorphism $\Phi: G \rightarrow \mathcal C(H)$ such that $m_g=|\mathcal P \cap \Phi(g)|$.  Two reduced Krull monoids are isomorphic if and only if they have the same characteristic (\cite[Theorem 2.5.4]{Ge-HK06a}). We pose the following problem.

\begin{problem} \label{characteristic-problem}
Let $G$ be a finite group, $\F$ be a  field, and
$V$ be a finite dimensional $\F$-vector space endowed with a linear action of $G$.
Determine the characteristic of $\F[V]^G$.
\end{problem}

Let all assumptions be as in Problem \ref{characteristic-problem} and suppose further that $G$ acts trivially on one variable. Then $\F[V]^G$ is a polynomial ring in this variable and hence every class contains a prime divisor by \cite[Theorem 14.3]{Fo73}.

\subsection{\bf   The abelian case}~ \label{subsec:abelian}

{\it Throughout this subsection, suppose that $G$ is abelian, $\F$ is  algebraically closed, and $\mathrm{char}(\F) \nmid |G|$.}
\smallskip

 The assumption on algebraic closedness is not too restrictive, since for any field $\F$ the set $\F[V]^G$ spans the ring of invariants over the algebraic closure $\overline{\F}$  as a vector space over $\overline{\F}$.
The assumption on the characteristic guarantees that every $G$-module is  completely reducible (i.e. is the direct sum of irreducible $G$-modules).
The dual space $V^*$ has a basis $\{x_1,\ldots,x_n\}$ consisting of $G$-eigenvectors whence $g\cdot x_i=\chi_i(g)x_i$ for all $i \in [1,n]$ where  $\chi_1, \ldots, \chi_n \in \Hom (G, \F^{\bullet})$.  We set $\widehat G = \Hom (G, \F^{\bullet})$,
$\widehat G_V = \{ \chi_1,\ldots,\chi_n \} \subset \widehat G$, and note that  $G \cong \widehat G$.
Recall that a completely reducible $H$-module $W$ (for a not necessarily abelian group $H$) is called \emph{multiplicity free} if it is the direct sum of pairwise non-isomorphic irreducible $H$-modules. In our case $V$ is multiplicity free if and only if the characters $\chi_1,\ldots,\chi_n$ are pairwise distinct.

It was   B. Schmid (\cite[Section 2]{Sc91a}) who first formulated a correspondence between a minimal generating system of $\F[V]^G$  and minimal product-one sequences over the character group (see also \cite{F-M-P-T08a}). The next proposition describes in detail the structural interplay. In particular, Proposition \ref{prop:bschmid}.2 shows that all (direct and inverse) results on minimal zero-sum sequences over $\widehat G_V$ (see Subsections \ref{3.C} and \ref{3.D}) carry over to $\mathcal A (M^G)$.

\begin{proposition} \label{prop:bschmid}
Let $M \subset \F[x_1, \ldots, x_n]$ be  the multiplicative monoid of monomials,
$\psi \colon M \to  \mathcal F (\widehat G_V)$ be the unique monoid homomorphism defined by $\psi (x_i) = \chi_i$ for all $i \in [1,n]$, and let $M^G \subset M$ denote the submonoid of  $G$-invariant monomials.
\begin{enumerate}
\item $\F[V]^G$ has $M^G$ as an $\F$-vector space basis, and $\F[V]^G$ is minimally generated as an $\F$-algebra by  $\mathcal A (M^G)$.
\item  The homomorphism $\psi \colon M \to   \mathcal F (\widehat G_V)$ and its restriction
      $\psi \t_{M^G}\colon M^G \to \mathcal B ( \widehat G_V)$
are degree-preserving transfer homomorphisms.
Moreover, $M^G$ is a reduced finitely generated Krull monoid, and $\mathcal A (M^G) = \psi^{-1} \big( \mathcal A (\widehat G_V) \big)$.

\item $\psi \t_{M^G}$ is an isomorphism if and only if $V$ is a multiplicity free $G$-module.
\item $\beta_k (G,V) =\mathsf D_k(M^G)= \mathsf D_k (\widehat G_V)$ and $\beta_k(G) = \mathsf D_k (G)$ for all $k \in \N$.
\end{enumerate}
\end{proposition}

\begin{proof}
1.  Each monomial spans a $G$-stable subspace in $\F[V]$, hence a polynomial is $G$-invariant if and only if all its monomials are $G$-invariant, so $M^G$ spans $\F[V]^G$.
The elements of $M^G$ are linearly independent, therefore  $\F[V]^G$ can be identified with the monoid algebra of $M^G$ over $\F$, which shows  the second statement.

2. $M$ and $\mathcal F (\widehat G_V)$ are free abelian monoids and $\psi$ maps primes onto primes. Thus $\psi \colon  M \to   \mathcal F (\widehat G_V)$ is a surjective degree-preserving monoid homomorphism and it is a transfer homomorphism.
Let $\pi \colon \mathcal{F}(\widehat G)\to \widehat G$ be the monoid homomorphism defined by $\pi ( \chi) = \chi$ for all $\chi \in \widehat G$. Then $\ker(\pi)=\mathcal{B}(\widehat G)$.
Taking into account that
for a monomial $m\in M$  $G$ acts on the space $\F m$  via the character $\pi(\psi(m))$, we conclude that
for a monomial $m\in M$ we have that $m\in M^G$ if and only if $\psi(m)\in  \mathcal B ( \widehat G_V)$.
This implies that the restriction  $\psi \t_{M^G}$ of the transfer homomorphism $\psi$ is also a transfer homomorphism.
Therefore $M^G$ is generated by $\mathcal A (M^G)=\psi^{-1} \big( \mathcal A (\widehat G_V) \big)$.
Since $\mathcal A(\widehat G_V)$ is finite, and $\psi$ has finite fibers, we conclude that the monoid $M^G$ is finitely generated.
Since $M$ is factorial and $\F[V]^G \subset \F[V]$ is saturated by Theorem \ref{main-theorem}, it follows that
\[
M \cap \mathsf q (M^G) \subset M \cap \F[V] \cap \mathsf q ( \F[V]^G ) \subset M \cap \F[V]^G = M^G
\]
whence $M^G \subset M$ is saturated and thus $M^G$ is a Krull monoid.

3. $V$ is a multiplicity free $G$-module if and only if  $\chi_1,\ldots,\chi_n$ are pairwise distinct. Since $\psi \colon  M \to   \mathcal F (\widehat G_V)$ maps the  primes $x_1, \ldots, x_n$ of $M$ onto the primes $\chi_1, \ldots, \chi_n$ of  $\mathcal F (\widehat G_V)$, $\psi$ is an isomorphism if and only if $\chi_1,\ldots,\chi_n$ are pairwise distinct.

4. Let $k \in \N$ and $M^G_+=M^G \setminus \{1\}$. Then $M^G \setminus (M_+^G)^{k+1} = \mathcal M_k (M^G)$. Since $\psi \t_{M^G}\colon M^G \to \mathcal B ( \widehat G_V)$
is degree-preserving transfer homomorphism, Proposition~\ref{3.7}.3 implies that $\mathsf D_k(M^G)=\mathsf D_k(\widehat G_V)$. Since $\F[V]^G$ is spanned by $M^G$, $(\F[V]^G_+)^{k+1}$ is spanned by $(M^G_+)^{k+1}$. Therefore the top degree of a homogeneous $G$-invariant not contained in $(\F[V]^G_+)^{k+1}$ coincides with
the maximal  degree of a monomial in $M^G_+ \setminus (M^G_+)^{k+1}= \mathcal M_k (M^G)$. Thus $\beta_k (G,V)= \mathsf D_k (M^G)$. For the $k$th Noether number $\beta_k (G)$ we have
\[
\begin{aligned}
 \beta_k(G) & = \sup\{ \beta_k(G,W): W \text{ is a $G$-module over $\F$} \} \\
   & = \sup \{ \mathsf D_k (\widehat G_W) \colon W \text{ is a $G$-module over $\F$} \} = \mathsf D_k (\widehat G)
\end{aligned}
\]
because for the regular representation $V_{\text{\rm reg}}$ we have $\widehat G_{V_{\text{\rm reg}
}} = \widehat G$.
\end{proof}

Recalling the notation of   Theorem \ref{main-theorem}, we have
\[
H = (\F[V]^G \setminus \{0\})_{\red} \quad \text{and} \quad D_0  = \{ {\gcd_D} (X) \colon X \subset H \ \text{finite} \} \subset D=(\F[V]^{G,\mathrm{rel}}\setminus\{0\})_{\red} \,.
\]
Furthermore, $M \subset \F[V]=\F[x_1, \ldots , x_n]$ is the monoid of monomials,  $M^G = M \cap \F[V]^G$, and we can view $M$ as a submonoid of $H$ and then $M^G = M \cap H$.
Since $M \subset H$ is saturated, $M = \mathsf q (M) \cap H$, and
\[
\mathsf q (M)/\mathsf q (M^G) = \mathsf q (M)/\mathsf q (M\cap H) = \mathsf q (M)/(\mathsf q (M) \cap \mathsf q (H)) \cong \mathsf q (M)\mathsf q (H)/\mathsf q (H) \subset \mathsf q (D)/\mathsf q (H) \,,
\]
we consider $\mathsf q (M)/\mathsf q (M^G)$ as a subset of $\mathsf q (D)/\mathsf q (H)$.

\begin{proposition} \label{prop:xi^e(xi)}~
Let all notation be as above and set $M_0 = M \cap D_0$.
\begin{enumerate}
\item  $M_0 \subset D_0$ is divisor closed whence $M_0$ is  free abelian, and $\mathcal{A}(M_0)=M\cap \mathcal{A}(D_0)=\{x_1^{e(x_1)},\ldots,x_n^{e(x_n)}\}$.

\item We have $e(x_i)=\min\{k\in\N \colon \chi_i^k\in \langle \chi_j\mid j\neq i \rangle \}$.

\item  $\Hom^0(\rho(G),\F^\bullet)$ is generated by $\{ \chi_1^{e(x_1)}, \ldots,\chi_n^{e(x_n)} \}$ and  $\F[x_1^{e(x_1)},\ldots,x_n^{e(x_n)}]=\F[V]^{G_1}$, where $G_1$ denotes the subgroup of $\rho(G)$ generated by the pseudoreflections in $\rho(G)$.

\item The embedding $M^G \hookrightarrow M_0$ is a divisor theory,
      \[
\overline w \mid_{\mathsf q (M_0)/\mathsf q (M^G)} \colon \mathcal C ( M^G) = \mathsf q (M_0)/\mathsf q (M^G) \to \Hom^0( \rho (G),\F^\bullet)
      \]
      is an isomorphism, and $\overline w ( \mathcal{C}(M^G)^*) =\{\chi_1^{e(x_1)},\ldots,\chi_n^{e(x_n)}\}$.
\end{enumerate}
\end{proposition}

\begin{proof}
1. If the product of two polynomials in $\F[V]$ has a single non-zero term, then both polynomials must have only one non-zero term.
Thus,  if $ab\in M$ for some $a,b\in D$, then both $a$ and $b$ belong to $M$. Hence $M \subset D$ is divisor closed implying that  $M_0 \subset D_0$ is divisor-closed. Therefore $\mathcal{A}(M_0)=M\cap \mathcal{A}(D_0)$.

By Theorem \ref{main-theorem}.3, $\mathcal A (D_0) = \{ q^{e(q)} \colon q \in \mathcal A (D) \}$. The divisor closedness of  $M$ in  $D$ implies that if $q^{e(q)}\in M$, then $q\in M \cap \mathcal A (D) =\mathcal A (M)= \{x_1,\ldots,x_n\}$. Thus  $M\cap \mathcal{A}(D_0)=\{x_1^{e(x_1)},\ldots,x_n^{e(x_n)}\}$.

2. For $i \in [1,n]$, we have
\[
e(x_i)=\min\{\mathsf  v_{x_i}(h) \colon x_i\mid_Dh , h \in H\} = \min\{\mathsf  v_{x_i}(m) \colon x_i\mid_Dm, m \in  M^G  \} \,,
\]
where the second equality holds because  for all $h \in H$ we have \newline $\mathsf v_{x_i} (h) = \min \{ \mathsf v_{x_i} (m) \colon m \ \text{ranges over the monomials of} \ h \}$. Note that a monomial $m = \prod_{i=1}^n x_i^{a_i}$ lies in $M^G$ if and only if $\prod_{i=1}^n \chi_i^{[a_i]}$ is a product-one sequence over $\widehat G$ if and only if $\chi_i^{a_i} = \prod_{j \ne i} \chi_j^{-a_j}$. Thus $\min\{\mathsf  v_{x_i}(m) \colon x_i\mid_Dm, m \in  M^G  \} = \min\{k\in\N \colon \chi_i^k\in \langle \chi_j\mid j\neq i \rangle \}$.

3. By Theorem \ref{main-theorem}.4,
$\Hom^0( \rho (G),\F^\bullet) = w (D_0)$ and hence $\Hom^0( \rho (G),\F^\bullet)$ is generated by $w ( \mathcal A (D_0))$.
Thus by 1., it remains to show that $\langle w ( \mathcal A (D_0)) \rangle = \langle w ( \mathcal A (M_0)) \rangle$. Since $\mathcal A (M_0) \subset \mathcal A (D_0)$, it follows that $\langle w ( \mathcal A (D_0)) \rangle  \supset \langle w ( \mathcal A (M_0)) \rangle$. To show the reverse inclusion, let $a \in \mathcal A (D_0)$. For any monomial $m$ occurring in $a$, we have $w (m) = w (a)$. By Theorem \ref{main-theorem}.4,  $D_0 =  \{f\in D\colon w(f)\in  \Hom^0( \rho (G),\F^\bullet)\}$  whence $m \in M \cap D_0 = M_0$ and clearly $w (m) \in \langle w ( \mathcal A (M_0)) \rangle$.

Recall that each monomial in $\F[V]$ spans a $G$-invariant subspace. Thus $f\in\F[V]$ is $G_1$-invariant if and only if all monomials of $f$ are $G_1$-invariant. Furthermore,  a monomial $m$ is $G_1$-invariant if and only if $w (m)$ contains $G_1$ in its kernel; equivalently  (by the characterization of $D_0$)  $m \in M \cap D_0 = M_0$. Thus $\F[V]^{G_1}$ is generated by $\mathcal A (M_0)$ and hence the assertion follows from 1.

4. Since $M\subset D$, $M_0 \subset D_0$ and $M^G \subset H$ are divisor closed and since
 the embedding $H\subset D_0$ is a divisor theory (Theorem~\ref{main-theorem}.4), $M^G \hookrightarrow M_0$ is a divisor homomorphism into a free abelian monoid. Let $m \in M_0$. Then $m \in D_0$ and there is a finite subset $Y \subset H$ such that $m = \gcd_{D_0} (Y)$. Let $X \subset D_0 \cap M = M_0$ be the set of all  monomials occurring in some $y \in Y$. Then $m = \gcd_{D_0} (X) = \gcd_{M_0} (X)$, where the last equality holds because $M_0 \subset D_0$ is divisor closed.

Restricting the isomorphism
\[
\overline w \mid_{\mathsf q (D_0)/\mathsf q (H)} \colon \mathcal C ( \F[V]^G) = \mathsf q (D_0)/\mathsf q (H) \to \Hom^0( \rho (G),\F^\bullet)
\]
 from Theorem \ref{main-theorem}, we obtain a monomorphism
\[
 \overline w\mid_{\mathsf q (M_0)/\mathsf q (M^G)} \colon \mathcal C (M^G) = \mathsf q (M_0)/\mathsf q (M^G) \to \Hom^0( \rho (G),\F^\bullet) \,.
\]
By 1. and 3., the image contains the generating set $\{ \chi_1^{e(x_1)}, \ldots,\chi_n^{e(x_n)} \}$ of the group $\Hom^0( \rho (G),\F^\bullet)$ and hence the above monomorphism is an isomorphism. The last statement follows from 1. by $\overline{w}(\mathcal{C}(M^G)^*)=\overline{w}(\mathcal{A}(M_0))$.
\end{proof}

\begin{proposition}\label{prop:diagram}~
Let $M \subset \F[x_1, \ldots, x_n]$ be  the multiplicative monoid of monomials,
and  $M^G \subset M$  the submonoid of  $G$-invariant monomials.
\begin{enumerate}
\item Every class of $\mathcal C ( \F[V]^G )$ contains a prime divisor.

\item We have the following commutative diagram of monoid homomorphisms
\[
\xymatrix@C=2cm{
H \ar[r]^{\theta_1} & \mathcal B( \mathcal C (H) ) \ar[r]^{w_1}_{\cong} & \mathcal B( \Hom^0( \rho (G),\F^\bullet)  ) \\
  & \mathcal B( \widehat G_V) \ar[ru]^{\nu} & \\
M^G \ar[rr]^{\theta_2} \ar[ru]^{\psi \t_{M^G}} \ar@{^{(}->}[uu]& & \mathcal B ( \mathcal C(M^G)^*) \ar@{^{(}->}[uu]^{w_2}
}
\]
where
\begin{itemize}
\item $\theta_1$ and $\theta_2$ are  transfer homomorphisms of Krull monoids as given in Proposition \ref{3.8}.

\item $w_1$ is the extension to the monoid of product-one sequences of the group isomorphism $\overline w \mid_{\mathsf q (D_0)/\mathsf q (H)}$ given in Theorem \ref{main-theorem}.4

\item $w_2$ is the extension to the monoid of product-one sequences of the restriction to $\mathcal C(M^G)^*$ of the group isomorphism $\overline w \mid_{\mathsf q (M_0)/\mathsf q (M^G)}$ given in Proposition \ref{prop:xi^e(xi)}

\item $\psi$    is given in Proposition \ref{prop:bschmid}.

\item  $\nu$ will be defined below (indeed,  $\nu$ is a transfer homomorphism as given in Proposition \ref{3.9}).
\end{itemize}

\item  If $\widehat G_V=\widehat G$, then every class of  $\mathcal{C}(M^G)$ contains a prime divisor.

\end{enumerate}
\end{proposition}

\begin{proof}
 1. By Proposition~\ref{prop:bschmid}.1, $\F[V]^G$ is the monoid algebra of  $M^G$ over $\F$. Thus,  by \cite[Theorem 8]{chang},
every  class of $\F[V]^G$ contains a prime divisor.

2. In order to show that the diagram is commutative, we fix an $m \in M^G$. We consider the divisor theory $M^G \hookrightarrow M_0$ from
Proposition~\ref{prop:xi^e(xi)}  and factorize $m$ in $M_0$, say $m = \prod_{i=1}^n \big(x_i^{e (x_i)} \big)^{a_i}$ where $a_1, \ldots, a_n \in \N_0$. Since $\overline w (x_i^{e(x_i)}) = \chi_i^{e(x_i)}$ for all $i \in [1,n]$, it follows that
\[
(w_2 \circ \theta_2) (m) = (\chi_1^{e(x_1)})^{[a_1]} \bdot \ldots \bdot (\chi_n^{e(x_n)})^{[a_n]} \in \mathcal B( \Hom^0( \rho (G),\F^\bullet)  ) \,.
\]
Next we view $m$ as an element in $H$ and consider the divisor theory $H \hookrightarrow D_0$. Since $M_0 \subset D_0$ is divisor closed,  $m = \prod_{i=1}^n \big(x_i^{e (x_i)} \big)^{a_i}$ is a factorization of $m$ in $D_0$. Therefore $(w_1 \circ \theta_1) (m) = (w_2 \circ \theta_2) (m)$.

By definition of $\psi$, we infer that
\[
\psi (m) = \chi_1^{[ e(x_1) a_1]} \bdot \ldots \bdot \chi_n^{[e (x_n) a_n]} \,.
\]
We define a partition of $\widehat G_V = G_1 \cup G_2$, where $G_2 = \{ \chi_i \colon \chi_i = \chi_j \ \text{for some distinct} \ i, j \in [1,n] \}$ and $G_1 = \widehat G_V \setminus G_2$.  Let $\nu \colon \mathcal B (\widehat G_V) \to  \mathcal B( \Hom^0( \rho (G),\F^\bullet)  )$ be defined as in Proposition \ref{3.9} (with respect to the partition $G_0 = G_1 \uplus G_2$). By Proposition \ref{prop:xi^e(xi)}.2, $e(x_i)=1$ if $\chi_i \in G_2$, and   $e (x_i)$  equals  the number $e(\chi_i)$  in Proposition \ref{3.9} if $\chi_i \in G_1$. Therefore  it follows that
\[
\nu ( \psi (m) )= (\chi_1^{e(x_1)})^{[a_1]} \bdot \ldots \bdot (\chi_n^{e(x_n)})^{[a_n]} \,,
\]
and hence the diagram commutes.

3. In a finite abelian group all elements are contained in the subgroup generated by the other elements, with the only exception of the generator of a $2$-element group.
Therefore unless $G$ is the $2$-element group and the non-trivial character occurs with multiplicity one in the sequence $\chi_1 \bdot \ldots \bdot \chi_n$, all the $e(x_i)=1$  by Proposition~\ref{prop:xi^e(xi)}.3,  and the elements $x_i$ are all prime  in $M_0$, so they represent all the divisor classes, as $i$ varies in $[1,n]$.
In the missing case we have $\F[V]^G=\F[x_1,\ldots,x_{n-1},x_n^2]$ (after a renumbering of the variables if necessary), hence both class groups are trivial, and $x_1$
and $x_2^2$ are prime elements in the unique class.
 \end{proof}

Thus Proposition \ref{prop:diagram}.1 gives a partial answer to Problem \ref{characteristic-problem}. Using that notation it states that $m_g \ge 1$ for all $g \in  \mathcal C ( \F[V]^G )$.


\begin{example}\label{example:nunotsurjective}
The set $\mathcal{C}(M^G)^*$ may be a proper subset of $\mathcal C (M^G)$, and consequently
the monoid homomorphism $\nu:  \mathcal B ( \widehat G_V)\to  \mathcal B ( \Hom^0( \rho (G),\F^\bullet)   )$ is not surjective in general.

1. Indeed, let $G$ be cyclic of order $3$, $g \in G$ with $\ord (g)=3$, and the action on $\F[x_1,x_2,x_3]$ is given by  $g\cdot x_i=\omega x_i$, where $\omega$ is a primitive cubic root of $1$. Then $\chi_1=\chi_2=\chi_3=\chi$, so $e(x_1)=e(x_2)=e(x_3)=1$,  implying  $\overline w (\mathcal{C}(M^G)^*) =\{\chi\}$ (each of the $x_i$ is a prime element in the class $\chi$),
whereas  $\overline w (\mathcal{C}(M^G)) =\{ \chi,\chi^2,\chi^3=1\}$, the $3$-element group.
Thus   $ \mathcal{B}(\widehat G_V)=\{\chi^{[3k]} \colon k\in\N_0\}$,  and   $\nu(\mathcal{B}(\widehat G_V))$ is the free abelian monoid $\mathcal{F}(\{\chi^3\})$ generated by $\chi^3=1\in\widehat G$.
The polynomials $x_1^2+x_2x_3$ and  $x_1^3+x_2^2x_3$ are irreducible, they are relative invariants of   weight $\chi^2$ and $\chi^3$, so they represent  prime elements    of $D_0$ in the remaining classes  $\chi^2$ and  $\chi^3=1$.

2. To provide an example with a multiplicity free module,
let $G$ be cyclic of order $5$, $g \in G$ with $\ord (g)=5$, and the action  on $\F[x_1,x_2,x_3]$ is given by $g\cdot x_1=\omega x_1$, $g\cdot x_2=\omega^2 x_2$, $g\cdot x_3=\omega^3x_3$,
where $\omega$ is a primitive fifth root of $1$. Then setting $\chi=\chi_1$, we have $\chi_2=\chi^2$, $\chi_3=\chi^3$ and $\overline w (\mathcal{C}(M^G)) =\langle \chi \rangle$ is the $5$-element group, so  $V$ is multiplicity free. Still we have $e(x_1)=e(x_2)=e(x_3)=1$, so  $\overline w (\mathcal{C}(M^G)^* ) =\{\chi,\chi^2,\chi^3\}$ (and
$x_1,x_2,x_3$ are the prime elements of $M_0$ in these classes).  The remaining classes  $\chi^4$ and $\chi^5=1$ contain the prime elements of $D_0$ represented by
 $x_2^2+x_1x_3$ and  $x_1^5+x_2x_3$.
\end{example}

\subsection{{\bf   A  monoid associated with $G$-modules}}~ \label{5.D}

\centerline{\it Throughout this subsection, suppose that $\mathrm{char}(\F)\nmid |G|$.}

\smallskip
In this subsection we  discuss   a monoid associated with representations of not necessarily abelian groups which in the case of abelian groups recovers the monoid of
$G$-invariant monomials. Decompose $V$ into the direct sum of  $G$-modules:
 \begin{align}\label{decomp} V= V_1 \oplus ... \oplus V_r \end{align}
 and denote by $\rho_i\colon G\to \GL(V_i)$ the corresponding group homomorphisms.
Then \eqref{decomp} induces a decomposition  of $\F[V]$  into multihomogeneous components as follows.
The coordinate ring $\F[V]$ is the symmetric algebra
$\Sym(V^*)= \bigoplus_{n=0}^{\infty} \Sym^n (V^*)$.  Writing
$\F[V]_{a} = \Sym^{a_1}(V^*_1) \otimes ... \otimes \Sym^{a_r}(V^*_r)$ we have
$\Sym^n(V^*)=  \oplus_{|a| = n}\F[V]_{a}$, and hence
$\F[V] = \oplus_{a \in \N_0^r} \F[V]_{a}$.
The summands $\F[V]_{a}$ are $G$-submodules in $\F[V]$, and $\F[V]_{a}\F[V]_{b}\subset \F[V]_{a+b}$, so $\F[V]$ is a $\N_0^r$-graded algebra.
Moreover, $\F[V]^G$ is spanned by its multihomogeneous components $\F[V]^G_{a}=\F[V]^G\cap \F[V]_{a}$. For $f\in \F[V]_{a}$ we call $a$ the
\emph{multidegree} of $f$.
We are in the position to define
\begin{align}\label{block_def}
\mathcal{B}(G,V) = \{ a \in \N_0^r \colon \F[V]^G_{a} \neq \{ 0\} \}
\end{align}
the set of multidegrees of multihomogeneous $G$-invariants. We give  precise information on $\mathcal{B}(G,V)$ in terms of quantities associated to the direct summands $V_i$ of $V$.
For $i \in [1,r]$ denote by $c_i$ the greatest common divisor of the elements of $\mathcal{B}(G,V_i)$, and  $F_i$  the Frobenius number of the numerical semigroup $\mathcal{B}(G,V_i)\subset \N_0$, so $F_i$ is the minimal positive integer $N$ such that $\mathcal{B}(G,V_i)$ contains $N+kc_i$ for all $k \in \N_0$.

\begin{proposition} \label{B(G,V)-is-C}~

\begin{enumerate}
\item $\mathcal B (G,V) \subset \N_0^r$ is a reduced finitely generated {\rm C}-monoid.
\item For each $i \in [1, r]$ and all $a \in \N_0^r $ satisfying $a_i\ge b(G,V_i)+F_i$ we have
      \begin{equation}\label{eq:alpha}
      a \in \mathcal  B (G,V) \qquad \text{if and only if} \qquad c_ie_i + a \in  \mathcal B (G,V)  \,.
      \end{equation}
\item For each $i \in [1, r]$ we have $c_i=|\rho_i(G)\cap \F^\bullet \mathrm{id}_{V_i}|$.
\end{enumerate}
\end{proposition}

\begin{proof}
1. Take $a, b \in \mathcal{B}(G, V)$, so  there exist non-zero $f\in \F[V]^G_{a}$ and $h\in \F[V]^G_{b}$.
Now $0\neq fh\in \F[V]^G_{a+b}$, hence $a + b \in \mathcal{B}(G,V)$.
This shows that $\mathcal B (G,V)$ is a submonoid of $\N_0$.
Moreover, the multidegrees of a multihomogeneous $\F$-algebra generating system of $\F[V]^G$ clearly generate the monoid $\mathcal B (G,V)$. Thus  $\mathcal B (G,V)$ is finitely generated by Theorem~\ref{4.1}.

To show that  $\mathcal B (G,V)$ is also a C-monoid,  recall that by Proposition~\ref{finitelygenerated}.3 a finitely generated submonoid $H$ of $\N_0^r$ is a C-monoid
if and only if each standard basis element $e_i\in\N_0^r$ has a multiple in $H$. Now this condition holds for $\mathcal{B}(G,V)$, since by Theorem~\ref{4.1}.2 $\F[V_i]^G\subset \F[V]^G$ contains a homogeneous element of positive degree for each $i\in [1,r]$.

\smallskip
2. By symmetry it is sufficient to verify \eqref{eq:alpha} in the case $i=1$.
Suppose $a\in  \mathcal B (G,V)$, so there is a non-zero $G$-invariant $f\in \Sym^{a_1}(V_1^*)\otimes \ldots\otimes \Sym^{a_r}(V_r^*)$.
Decompose $\Sym^{a_1}(V^*_1)=\bigoplus_jW_j$ into a direct sum of irreducible $G$-modules.
This gives a direct sum decomposition
$\Sym^{a_1}(V_1^*)\otimes \ldots\otimes \Sym^{a_r}(V_r^*)=\bigoplus_j(W_j\otimes \Sym^{a_2}(V^*_2)\otimes\ldots\otimes\Sym^{a_r}(V^*_r))$.
It follows that $\Sym^{a_1}(V^*_1)$ contains an irreducible $G$-module direct summand $W$ such that $W\otimes \Sym^{a_2}(V^*_2)\otimes\ldots\otimes \Sym^{a_r}(V^*_r)$ contains a non-zero $G$-invariant. By definition of $b(G,V_1)$ we know that $\F[V_1]$ is generated as an $\F[V_1]^G$ module by its homogeneous components of degree $\le b(G,V_1)$.
Therefore there exists a $d\le b(G,V_1)$ such that the degree $d$ homogeneous component of $\F[V]$ contains a $G$-submodule $U\cong W$, and $a_1\in d+\mathcal{B}(G,V_1)$.
Now for any homogeneous $h\in \F[V_1]^G$ we have
$hU\otimes \Sym^{a_2}(V^*_2)\otimes\ldots\otimes\Sym^{a_r}(V^*_r))\subset \F[V]_{(d+\deg(h),a_2,\ldots,a_r)}$ contains a non-zero $G$-invariant, since it is isomorphic to
$W\otimes \Sym^{a_2}(V^*_2)\otimes\ldots\otimes\Sym^{a_r}(V^*_r))$.
It follows that $(k,a_2,\ldots,a_r)\in\mathcal{B}(G,V)$ for all $k\in d+\mathcal{B}(G,V_1)$, in particular, for all $k\in \{d+F_1,d+F_1+c_1,d+F_1+2c_1,\ldots\}$.

3. Let $i \in [1,r]$, and to simplify notation set $W=V_i$, $c=c_i$, and $\phi=\rho_i$.
Recall that $\F[W]^A=\F[W]^B$ for some finite subgroups $A,B\subset \GL(W)$ implies that $A=B$. Indeed, the condition implies equality $\F(W)^A=\F(W)^B$ of the corresponding quotient fields,
and so both $A$ and $B$ are the Galois groups of the field extension $\F(W)$ over $\F(W)^A=\F(W)^B$, implying $A=B$.
Now denote by $Z\subset \GL(W)$ the subgroup of scalar transformations
$Z=\{\omega \mathrm{id}_{W}\colon \omega^{c}=1\}$, so $Z$ is a central cyclic subgroup of $\GL(W)$  of order $c$. Clearly every homogeneous element of $\F[W]$ whose degree is a multiple of $c$ is invariant under $Z$. It follows that $\F[W]^G\subset \F[W]^Z$, hence denoting by $\tilde G$ the subgroup $\phi(G)Z$ of $\GL(W)$, we have $\F[W]^G=\F[W]^{\tilde G}$. It follows that $\phi(G)=\tilde G$, i.e. $Z\subset \phi(G)$, and so $c=|Z|$ divides the order of
$\phi(G)\cap \F^\bullet \mathrm{id}_W$. Conversely, if $\lambda\mathrm{id}_W$ belongs to $\rho(G)$, then every element of $\F[W]^G$ must be invariant under the scalar transformation $\lambda\mathrm{id}_W$, whence all  homogeneous components of $\F[W]^G$ have degree divisible by the order of $\lambda$, so the order of the cyclic
group $\phi(G)\cap \F^\bullet \mathrm{id}_W$ must divide $c$.
\end{proof}

In general $\mathcal{B}(G,V)$ is not a Krull monoid. To provide an example, consider  the two-dimensional irreducible representation $V$ of the symmetric group $S_3 = D_6$. Its ring of polynomial invariants is generated by an element of degree 2 and  3, hence $\mathcal{B}(G,V) = \langle 2,3 \rangle \subset (\N_0, +)$ , which  is not Krull.

\begin{proposition}\label{prop:module-davenport}
For every $k \in \N$ we have  $\mathsf{D}_k(\mathcal{B}(G,V)) \le \beta_k(G,V)$.
\end{proposition}

\begin{proof}
Let $k \in \N$. Take $a \in \mathcal{B}(G,V)$ such that $|a| > \beta_k(G,V)$.
By \eqref{block_def} a multihomogeneous invariant $f \in \F[V]^G_a$ exists.
As $\deg(f) =|a|> \beta_k(G,V)$ it follows that $f = \sum_{i=1}^N  f_{i,1}\ldots f_{i,k+1}$
for some non-zero multihomogeneous invariants $f_{i,j}$ of positive degree. Denoting by $a_{i,j}\in\N_0^r$ the multidegree of $f_{i,j}$, we have
that $a=a_{i,1}+\ldots +a_{i,k+1}$, where $0\neq a_{i,j}\in \mathcal{B}(G,V)$. This shows that all $a\in \mathcal{B}(G,V)$ with $|a|>\beta_k(G,V)$ factor into the product of more than $k$ atoms, implying the desired inequality.
\end{proof}

\begin{remarks} \label{remark:block monoid of G module}~
1. Let $G$ be abelian and suppose that  $\F$ is algebraically closed.
Then we may take in \eqref{decomp} a decomposition of $V$  into the direct sum of $1$-dimensional submodules and so $V_i^*$,
is spanned by a variable $x_i$ as in Subsection~\ref{subsec:abelian}. Then  $\F[V]_{a}$ is spanned by the monomial
$x_1^{a_1}\cdots x_r^{a_r}$ and $a \in \mathcal{B}(G,V)$ holds if and only if the corresponding monomial is $G$-invariant.
So in this case $\mathcal{B}(G,V)$ can be naturally identified with $M^G$ and the transfer homomorphism $\psi \t_{M^G}$ of Proposition~\ref{prop:bschmid} can be
thought of as a transfer homomorphism  $\mathcal{B}(G,V) \to \mathcal B ( \widehat G_V)$,
which is an isomorphism if $V$ is multiplicity free.
However, this transfer homomorphism does not seem to have an analogues for non-abelian $G$ (i.e. the study of $\mathcal{B}(G,V)$ can not be reduced to the multiplicity free case), as it is shown by the example below.

2. The binary tetrahedral group $G=\widetilde A_4\cong SL_2(\F_3)$ of order $24$ has a 2-dimensional complex irreducible representation $V$ such that
$\F[V]^G$ is minimally generated by elements of degree $6,8,12$ (see for example \cite[Appendix A]{Be93a}), hence $\mathcal{B}(G,V)=\{0,6,8,12,14,16,18,\ldots\}$.
On the other hand under this representation $G$ is mapped into the special linear group of $V$, so on $V\oplus V$ the function maping $((x_1,x_2),(y_1,y_2))\mapsto
\det\left(\begin{array}{cc}x_1 & y_1 \\x_2 & y_2\end{array}\right)$ is a $G$-invariant of multidegree $(1,1)$, implying that  $(1,1)\in \mathcal{B}(G,V\oplus V)$.
This shows that the transfer homomorphism $\tau:\N_0^2\to \N_0$, $(a_1,a_2)\mapsto a_1+a_2$ does not map $\mathcal{B}(G,V\oplus V)$ into $\mathcal{B}(G,V)$, as
$\tau(1,1)=2\notin \mathcal{B}(G,V)$.
\end{remarks}

Recall that the multigraded Hilbert series of $\F[V]^G$ in $r$ indeterminates $T=(T_1,...,T_r)$ is
\[H(\F[V]^G,T)= \sum_{a \in \N_0^r}\dim_{\F}( \F[V]^G_a)T_1^{a_1} \cdots T_r^{a_r}, \quad \text{and hence}
\]
\[\mathcal{B}(G,V) = \{a \in \N_0^r  \colon \mbox{ {\rm the coefficient of} }T^{a}\mbox{  {\rm in} }H(\F[V]^G, T)\mbox{  {\rm is nonzero}} \}.\]
By this observation Proposition~\ref{prop:module-davenport} can be used for finding lower bounds on the Noether number $\beta(G,V)$, thanks to  the following classical result of Molien (see for example \cite[Theorem 2.5.2]{Be93a}):

\begin{proposition}
Given a $G$-module $V= V_1\oplus...\oplus V_r$ over $\C$,
let $\rho_i(g) \in \GL(V_i)$ be the linear transformation defining the action of $g\in G$ on $V_i$. Then
we have
\[H(\C[V]^G,T)
= \frac 1{|G|} \sum_{g \in G} \prod_{i=1}^r \frac1{\det(\mathrm{id}_{V_i}-\rho_i(g)\cdot T_i)}.
\]
\end{proposition}

\begin{example}[see p. 54-55 in \cite{Ne-Sm02a}]
Consider the alternating group $A_5$ and its $3$-dimensional representation over $\C^3$ as the group of symmetries of an icosahedron. The Hilbert series then equals
\[ \frac{1+T^{15}}{(1-T^2)(1-T^6)(1-T^{10})} \]
whence it is easily seen that $\mathcal{B}(A_5, \C^3) = \langle 2,6,10,15\rangle$ and consequently $\beta(A_5) \ge \mathsf{D}(\mathcal{B}(A_5,\C^3)) =15$. Note  that this lower bound is stronger than what we could get from  $\beta(G) \ge \max_{H\subsetneq G} \beta(H)$, since $\beta(H)\le |H|\le 12$ for any proper subgroup $H$ of $A_5$.
\end{example}

\section{\bf Constants from Invariant Theory and their counterparts in Arithmetic Combinatorics} \label{sec:5}

In Subsection \ref{sec:5.1} we compare known reduction lemmas for the Noether number with reduction lemmas for the Davenport constants achieved in previous sections. We demonstrate  how to use  them to determine the precise value of Noether numbers and Davenport constants in new examples.  In Subsection \ref{5.A} we consider an invariant theoretic analogue of the constant $\eta (G)$ (for the definition of $\eta (G)$ see the discussions before Proposition \ref{2.8} and Lemma \ref{3.1}).

\smallskip
\centerline{\it Throughout this section,  suppose that  $\mathrm{char}(\F) \nmid |G|$.}

\subsection{The  Noether number versus the Davenport constant}\label{sec:5.1}~

In the non-abelian case no  structural connection (like Proposition~\ref{prop:bschmid}) is known between the $G$-invariant polynomials and the product-one sequences over $G$. Nevertheless, a variety of   features of the $k$th Noether numbers and the $k$th Davenport constants are strikingly similar, and we offer a detailed comparison.

Recall that   $\beta_k(G)=b_k(G)+1$  (\eqref{eq:beta=b+1}) and that  $ \mathsf{d}_k(G)+1\le \mathsf{D}_k(G)$ (Proposition~\ref{2.8}.1).

\begin{enumerate}
\item The  inequalities
\begin{align}\label{eq:trivial}
(a) \quad \beta_k(G)\le k\beta(G) &&
(b) \quad {\mathsf d}_k(G)+1\le k({\mathsf d}(G)+1) &&
(c) \quad \mathsf D_k(G)\le k\mathsf D(G)
\end{align}

\item  Reduction lemma for normal subgroups $N \triangleleft G$:
\begin{align}\label{red_norm}
(a) \quad \beta_k(G)  \le \beta_{\beta_k(G/N)}(N) &&
(b)  \quad {\mathsf d}_k(G)  \le {\mathsf d}_{\mathsf d_k(N) +1 }(G/N) &&
\end{align}

\item Reduction lemma for  arbitrary subgroups $H \le G$ with index $l=[G:H]$:
\begin{align}\label{red_sub}
(a) \ \beta_k(G)  \le \beta_{kl}(H)\le l\beta_k(H) &&
(b) \ {\mathsf d}_k(G)+1  \le l({\mathsf d}_k(H)+1) &&
(c) \ \mathsf D_k(G)  \le l \mathsf D_k(H)
\end{align}

\item Supra-additivity: for a normal subgroup $N \triangleleft G$  we have
\begin{align}\label{subadd}
(a) \quad b_{k+r-1}(G)  \ge b_k(N) + b_r(G/N) \mbox{ if }G/N\mbox{ is abelian}
\end{align}
\[
(b) \quad \mathsf d_{k+r-1}(G)  \ge\mathsf  d_k(N) + \mathsf d_r(G/N)
\]

\item Monotonicity: for an arbitrary subgroup $H \le G$ we have
\begin{align}\label{mono}
(a) \quad \beta_k(G)  \ge \beta_k(H)  &&
(b) \quad \mathsf d_k(G)\ge \mathsf d_k(H) &&
(c) \quad \mathsf D_k(G)  \ge \mathsf D_k(H)
\end{align}

\item Almost linearity in $k$: there are positive constants $C, C', C'',k_0,k_0',k_0''$ depending only on $G$ such that
\begin{align}\label{qlinear}
(a) \ \beta_k(G)  = k\sigma(G) + C \text{ for all }k > k_0 \mbox{ if }\mathrm{char}(\F)=0  &&
(b) \  \mathsf d_k(G)  =k \mathsf e (G) +C'  && \\ \notag
\text{ for all  } k >k_0' \quad \text{and} \quad (c) \  \mathsf D_k(G)  =k \mathsf e (G) +C'' \text{ for all  } k >k_0''
\end{align}

\item The following functions are non-increasing in $k$:
\begin{align} \label{non-increasing}
(a) \quad \beta_k (G)/k  \quad  \text{if} \quad  \mathrm{char} (\F) = 0 &&
(b) \quad \mathsf D_k (G)/k
\end{align}
\end{enumerate}

The inequality \eqref{eq:trivial} (a) is observed in \cite{Cz-Do15a}, (b) is shown in Proposition \ref{gen-dav-5}.4, whereas (c) is observed in the beginning of Subsection \ref{2.E}.

For the proof of \eqref{red_norm} (a) see  \cite[Lemma 1.5]{Cz-Do15a} and for part (b) see Proposition~\ref{gen-dav-5}.2.
Note that the roles of  $N$ and $G/N$ are swapped in the formulas (a) respectively (b), but in the abelian case they amount to the same.

The first inequality in part (a) of \eqref{red_sub} is proved  in \cite[Corollary 1.11]{Cz-Do15a} for cases when (i) $\mathrm{char}(\F) =0$ or $\mathrm{char}(\F) > [G:H]$; (ii) $H$ is normal in $G$ and $\mathrm{char}(\F) \nmid [G:H]$; (iii) $\mathrm{char}(\F)$ does not divide $|G|$. It is conjectured, however that it holds in fact whenever $\mathrm{char}(F) \nmid [G:H]$ (see \cite{kemper-separating}).
By  \cite[Lemma 4.3]{Cz-Do13c}, we have $\beta_{kl}(H)\le l\beta_k(H)$  for all positive integers $k,l$, implying the second inequality in part (a).
Parts (b) and (c) of  \eqref{red_sub} appear in Proposition \ref{gen-dav-5} (3. and 5.)

Part (a) of \eqref{subadd} appears in \cite[Theorem 4.3 and Remark 4.4]{Cz-Do14a} while part (b) is proved in Proposition~\ref{gen-dav-5}.1.

Parts (b) and (c)  of \eqref{mono} are immediate from the definitions, while part (a) follows from an argument of B. Schmid (\cite[Proposition~5.1]{Sc91a}) which also shows that $\beta_k(G, \mathrm{Ind}_H^G V) \ge \beta_k(H,V)$ for all $k \ge 1$ (see \cite[Lemma 4.1]{Cz-Do14a}).

Part (a) of \eqref{qlinear} is proved in   \cite[Proposition 4.5]{Cz-Do13c} (the constant $\sigma(G)$ will be discussed in Subsection \ref{5.A}, and  for \eqref{qlinear} (b) and (c) we refer to Proposition~\ref{2.7}.2 and Proposition~\ref{2.8}.2.

Part (a) of \eqref{non-increasing} is proved in \cite[Section 4]{Cz-Do13c} and for \eqref{non-increasing} (b) we refer to Proposition~\ref{2.7}.3.

Furthermore, for a normal subgroup $N \triangleleft G$ we have
\begin{align}\label{submult}
(a) \quad \beta(G)  \le {\beta(G/N)}\beta(N) &&
(b) \quad \mathsf D(G)  \le \mathsf D(N) \mathsf  D(G/N) \,,
\end{align}
where in (b) we assume that $N \cap G' = \{1\}$.
Here part (a) is originally due to B. Schmid (\cite[Lemma 3.1]{Sc91a})
and it is an immediate consequence of  \eqref{eq:trivial} (a)  and  \eqref{red_norm} (a)
while part (b)  is proven  in \cite[Theorem 3.3]{Ge-Gr13a}.

The above reduction lemmas on the Noether numbers are key tools in the proof of the following theorem.

\begin{theorem}\label{thm:betaindex2}~
Let $k \in \N$.
\begin{enumerate}
\item $\beta_k(A_4) = 4k+2$ and $\beta(\tilde{A_4}) = 12$, where  $A_4$ is the alternating group of degree $4$ and $\tilde{A}_4$ is the binary tetrahedral group.

\smallskip
\item If $G$ is a non-cyclic group with a cyclic subgroup of index two, then
      \[
      \beta_k(G) = \frac{1}{2} |G| k +
               \begin{cases}
               2 	& \text{ if } G=Dic_{4m}, \text{  $m>1$};\\
               1	& \text{ otherwise. }
               \end{cases}
      \]
      where $\text{Dic}_{4m} = \langle a,b: a^{2m}=1, b^2 = a^m, bab^{-1} = a^{-1} \rangle$ is the dicyclic group.

\smallskip
\item \begin{align*}
      \beta(G)\ge \frac 12 |G| \quad \text{if and only if} \quad  &\text{$G$ has a cyclic subgroup of index at most two or} \\
      &  \text{$G$ is isomorphic to $C_3\oplus C_3$, $C_2 \oplus C_2 \oplus C_2$, $A_4$ or $ \tilde{A}_4$}
      \end{align*}
\end{enumerate}
\end{theorem}

\begin{proof} For
1. see \cite[Theorem 3.4 and Corollary 3.6]{Cz-Do15a},  for 2. see  \cite[Theorem 10.3]{Cz-Do14a},  and 3. can be found in \cite[Theorem~1.1]{Cz-Do15a}.
\end{proof}

It is  worthwhile to compare Theorem~\ref{thm:betaindex2}.3 with the statement from \cite{olson-white} asserting that $\mathsf d(G)<\frac 12|G|$ unless $G$ has a cyclic subgroup of index at most two.
If $G$ is abelian, then Lemma \ref{dandDabeliancase} and Proposition \ref{prop:bschmid} imply $\mathsf d (G)+1 = \beta (G) = \mathsf D (G)$. Combining Theorems~\ref{3.12} and \ref{thm:betaindex2} we obtain that all groups  $G$ having a cyclic subgroup of index at most two satisfy the inequality $\mathsf d (G)+1 \le \beta (G) \le \mathsf D (G)$.
Moreover, for these groups $\beta (G)=\mathsf{d}(G)+1$, except for the dicyclic groups, where $\beta (G)=\mathsf{d}(G)+2$. 
On the other hand, it was shown in \cite{C-D-S} that for the Heisenberg group $H_{27}$ of order 27 we have 
$\mathsf D (H_{27}) < \beta (H_{27})$.

\begin{problem} \label{Noether-Davenport} 
Study the relationship between the invariants $\mathsf d (G)$, $\beta (G)$, and $\mathsf D (G)$. \\ In particular,
\begin{itemize}
\item Characterize the groups $G$ satisfying $\mathsf d (G)+1 \le \beta (G)$.
\item Characterize the groups $G$ satisfying $\beta (G) \le \mathsf D (G)$.
\end{itemize}
\end{problem}

In the following examples we demonstrate how  the reduction results presented at the beginning of this section do work. This allows us to  determine  Noether numbers and Davenport constants of non-abelian groups, for which they were not  known before.

\begin{example}\label{example:C_pq rtimes C_q}
Let $p, q $ be primes such that $ q \mid p-1$.

1. Consider  the non-abelian semi-direct product $G=C_p \rtimes C_q$. A conjecture  attributed to Pawale (\cite{wehlau}) states that
$\beta(C_p \rtimes C_q) =p+q-1$
and many subsequent research was done in this direction (\cite{Do-He00a}, \cite{Cz-Do15a}). Currently it is fully proved only for the cases $q=2$ in \cite{Sc91a} and  $q=3$ in \cite{Cz14d} whereas for arbitrary $q$ we have only upper bounds in  \cite{Cz-Do15a}, proved using  known results related to the Olson constant  of the cyclic group of order $p$.
Theorem \ref{3.13}.3 implies that $\mathsf{d}(G)+1=p+q-1$  and hence $\mathsf d (G)+1$ coincides with the conjectured value for $\beta(G)$.

2. In view of the great difficulties related to Pawale's conjecture
it is quite remarkable that we can determine the exact value of the Noether number
for the non-abelian semidirect product $C_{pq} \rtimes C_q$.
Indeed, this group contains an index $p$ subgroup isomorphic to $C_q\oplus C_q$, hence
$\beta(C_{pq} \rtimes C_q) \le \beta_p(C_q \oplus C_q)$ by \eqref{red_sub}. By   Proposition~\ref{prop:bschmid} 4. we have $\beta_p(C_q \oplus C_q)= \mathsf D_p(C_q\oplus C_q)$, and finally, $\mathsf D_p(C_q\oplus C_q)=pq+q-1$ by Theorem \ref{gen-dav-abelian}. Thus we have $\beta(C_{pq} \rtimes C_q) \le pq+q-1$.
The reverse inequality also holds, since $\beta(C_{pq} \rtimes C_q)$ contains a normal subgroup $N\cong C_{pq}$ with $G/N\cong C_q$,
so by   \eqref{subadd} and \eqref{eq:beta=b+1} we have $\beta(C_{pq} \rtimes C_q)\ge \beta(C_{pq})+\beta(C_q)-1=pq+q-1$.
So we have $\beta(C_{pq} \rtimes C_q) =pq+q-1$.

Next we determine the small Davenport constant of this group. Since $C_{pq}$ is a normal subgroup and the corresponding factor group is $C_q$, we have by
Proposition \ref{gen-dav-5}.1 that $\mathsf d(C_{pq} \rtimes C_q)\ge \mathsf d(C_{pq})+\mathsf d(C_q)=p+q-2$. The reverse inequality $\mathsf d(C_{pq} \rtimes C_q)\le p+q-2$ follows from Theorem \ref{3.13}.4, since
$C_{pq} \rtimes C_q$ contains also a normal subgroup $N\cong C_p$ such that $G/N\cong C_q\oplus C_q$.
Consequently, by Lemma \ref{3.1}.2.(a) we have
\[
\mathsf D(C_{pq} \rtimes C_q)\ge \mathsf d(C_{pq} \rtimes C_q)+1=pq+q-1 \,.
\]
\end{example}


\begin{example} \label{S_4}
The symmetric group $S_4$ has a normal subgroup $N\cong C_2 \oplus C_2$ such that $S_4 /N \cong D_6$.
We know that $\beta(D_6) = 4$ (say by Theorem~\ref{thm:betaindex2} 2.).
Thus by \eqref{red_norm} and Theorem \ref{gen-dav-abelian} we have
$\beta(S_4) \le \beta_{\beta(D_6)}(C_2\oplus C_2) = \mathsf D _4(C_2\oplus C_2) = 2 \cdot 4 +1 = 9$.

Now let $V$ be the standard $4$-dimensional permutation representation of $S_4$ and $\mathrm{sign}: S_4 \to \{ \pm 1\}$ the sign character.
It is not difficult to prove the algebra isomorphism
 $\F[V \otimes \mathrm{sign}]^{S_4} \cong \F[V]^{S_4}_{even} \oplus \Delta_4 \F[V]^{S_4}_{odd} $
 where  $\Delta_4$ is the Vandermonde determinant in 4 variables, $\F[V]^{S_4}_{even}$ is the span of the even degree homogeneous components of $\F[V]^{S_4}$, and
 $\F[V]^{S_4}_{odd}$ is the span of the odd degree homogeneous components of $\F[V]^{S_4}$.
 Moreover, the  algebra  $\F[V]^{S_4}_{even} \oplus \Delta_4  \F[V]^{S_4}_{odd}$ is easily seen to be minimally generated by 
 $\sigma_2, \sigma_1^2, \sigma_1\sigma_3, 
 \sigma_4, \sigma_3^2, \sigma_1\Delta_4, \sigma_3\Delta_4$,
 where $\sigma_i$ is the $i$-th elementary symmetric polynomial.
As a result $\beta(S_4, V \otimes \mathrm{sign}) = \deg(\sigma_3 \Delta_4)= 3 +\binom 4 2 = 9$.
So we conclude that $\beta(S_4)=9$ (and not $10$, as it is claimed on page 14 of \cite{kraft-procesi}).
\end{example}


\begin{example} \label{Pauli}
Let $G$ be the group generated by the complex Pauli matrices
\[\left(\begin{array}{rr}
0 & 1 \\
1 & 0  \\
\end{array}\right), \quad
\left(\begin{array}{rr}
0 & -i  \\
i & 0  \\
\end{array}\right), \quad
\left(\begin{array}{rr}
1 & 0 \\
0 & -1 \\
\end{array}\right).\]
This is a pseudoreflection group, hence the ring of invariants on $V=\C^2$ is generated by two elements, namely
$\mathbb{C}[x,y]^G = \mathbb{C}[x^4+y^4, x^2y^2]$. Moreover, $b(G,V)$ is the sum of the degrees of the generators minus $\dim(V)$ (again because $G$ is a pseudoreflection group, see \cite{chevalley}), so $b(G,V)=6$.
It follows by \eqref{eq:beta=b+1} that
$\beta(G)=b(G)+1 \ge b(G,V) +1 =7$.

On the other hand,  $G $ is a non-abelian semi-direct product$ (C_4 \oplus C_2) \rtimes C_2$.
Therefore $G$ has a normal subgroup $N$ such that $N\cong G/N\cong C_2 \oplus C_2$ and thus
\[ \beta(G) \le \beta_{\beta(C_2\oplus C_2)}(C_2\oplus C_2) = \mathsf D_3(C_2 \oplus C_2) = 7.\]
So we conclude that $\beta(G) = 7$.
\end{example}

\subsection{{\bf   The constants $\sigma (G, V)$ and $\eta (G,V)$}}~ \label{5.A}

\begin{definition}\label{def:sigma}~

\begin{enumerate}
\item  Let  $\sigma(G,V)$  denote the smallest  $d \in \N_0 \cup \{\infty\}$ such that $\F[V]^G$ is a finitely generated module over a subring $\F[f_1,\ldots,f_r]$ such that $\max \{ \deg(f_i) \colon i \in [1,r]\} = d$.  We define $\sigma(G)=\sup \{  \sigma(G,W) \colon W \ \text{is a $G$-module} \}$.

\item Let $S \subset \F[V]^G$ be the $F$-subalgebra of $\F[V]^G$ generated by its elements of degree at most $\sigma(G,V)$.
Then  $\eta(G,V)$ denotes the maximal degree of generators of $\F[V]^G_+$ as an $S$-module.
\end{enumerate}
\end{definition}

One motivation to study $\sigma(G,V)$ and $\eta(G,V)$ is that  by a straightforward induction argument (\cite[Section 4]{Cz-Do13c}) we have
\[
\beta_k(G,V) \le (k-1)\sigma(G,V) + \eta(G,V) \,.
\]
By    \cite[Proposition~6.2]{Cz-Do13c},  $\sigma(C_p \rtimes C_q) = p$ (this is also true in characteristic $q$, see \cite[Proposition 4.5]{Elm-Ko}).

If $\F$ is algebraically closed, then, by Hilbert's Nullstellensatz,
$\sigma(G,V)$ is the smallest  $d$ such that there exist homogeneous invariants of degree at most $d$ whose common zero locus is the origin.
It is shown in Lemma~5.1, 5.4 and 5.6 of \cite{Cz-Do13c}
 (some extensions to the modular case and for linear algebraic groups are given in  \cite{Elm-Ko}) that
\begin{itemize}
\item $\sigma(G) \le \sigma(G/N)\sigma(N)$  if  $N \triangleleft G$;
\item  $\sigma(H) \le \sigma(G) \le [G:H]\sigma(H)$  if  $H \le G$;
\item  $\sigma(G) = \max \{ \sigma(G,V) \colon V \text{ is an irreducible $G$-module}\}$.
\end{itemize}

\begin{proposition}\label{prop:abeliansigma}
Let  $G$ be abelian.
\begin{enumerate}
\item  $\sigma(G)=\exp(G)= \mathsf e(G) $.
\item $\eta(G) = \sup \{ \eta(G,W) \colon W \ \text{is a $G$-module} \}$.
\end{enumerate}
\end{proposition}

\begin{proof}
For 1. see  \cite[Corollary 5.3]{Cz-Do13c}. To prove 2., let
$T\in\mathcal F(\widehat G)$ with $|T|=\eta(G)-1$ such that $T$ has no product-one subsequence $U$ with $|U|\in [1, \mathsf e(G)]$.
Let $V$ be the regular representation of $G$, and  denote by $S$ the subalgebra of $\F[V]^G$ generated by its elements of degree at most $\sigma(G)=\mathsf e(G)$.
Now  $\psi \colon M \to \mathcal{F}(\widehat G)$ is an isomorphism (see the proof of Proposition~\ref{prop:bschmid}.3.).
Thus $\psi^{-1}(T)\in M$ is not divisible by a $G$-invariant monomial of degree smaller than $\mathsf e(G)$. Since both $S$ and $\F[V]$ are spanned by monomials, it follows that $\psi^{-1}(T)\in M$ is not contained in the $S$-submodule of $\F[V]^G_+$ generated by elements of degree less than $\deg(\psi^{-1}(T))$. This shows that for the regular representation $V$ of $G$ we have $\eta(G,V)\ge \eta(\widehat G)$.

On the other hand let  $W$ be an arbitrary $G$-module, and $m\in M$ a monomial with $\deg(m)>\eta(G)$. Then $\psi(m)$ has a product-one subsequence with length at most $\mathsf e(G)=\sigma(G)$, hence  $m$ is divisible by a $G$-invariant monomial of length at most $\sigma(G)$ (see the beginning of the proof of Proposition~\ref{prop:bschmid}.2). This shows the inequality $\eta(G,W)\le \eta(\widehat G)$. Taking into account the isomorphism $\widehat G\cong G$ we are done.
\end{proof}

For the state of the art on $\eta (G)$ (in the abelian case) we refer to \cite[Theorem 5.8.3]{Ge-HK06a}, \cite{Fa-Ga-Zh11a,Fa-Ga-Wa-Zh13a}. Proposition \ref{prop:abeliansigma} inspires the following problem.

\begin{problem}
Let $G$ be a finite non-abelian group. Is $\sup \{ \eta(G,W) \colon W \ \text{is a $G$-module} \}$ finite? Is it related to $\eta ( \mathcal B (G))$ (see Subsection \ref{2.E} and \ref{3.A})?
\end{problem}

\acknowledgement{This work was supported by the {\it Austrian Science Fund FWF} (Project No. P26036-N26) and  by OTKA K101515 and PD113138.  }


\bibliographystyle{amsplain}
\providecommand{\bysame}{\leavevmode\hbox to3em{\hrulefill}\thinspace}
\providecommand{\MR}{\relax\ifhmode\unskip\space\fi MR }
\providecommand{\MRhref}[2]{%
  \href{http://www.ams.org/mathscinet-getitem?mr=#1}{#2}
}
\providecommand{\href}[2]{#2}

\end{document}